\documentclass[12pt,reqno]{amsart}
\usepackage{color}
\usepackage{todonotes,comment,mathrsfs}
\usepackage{stackengine}
\usepackage{latexsym}
\usepackage{amsmath,amssymb,dsfont,amsfonts}
\usepackage{mathrsfs}
\usepackage{verbatim}
\usepackage{graphicx}
\usepackage{graphics}
\usepackage{geometry}
\usepackage{booktabs}
\usepackage[shortlabels]{enumitem}
\usepackage{xcolor}
\definecolor{olive}{rgb}{0.3, 0.4, .1}
\definecolor{fore}{RGB}{249,242,215}
\definecolor{back}{RGB}{51,51,51}
\definecolor{title}{RGB}{255,0,90}
\definecolor{dgreen}{rgb}{0.,0.6,0.}
\definecolor{gold}{rgb}{1.,0.84,0.}
\definecolor{JungleGreen}{cmyk}{0.99,0,0.52,0}
\definecolor{BlueGreen}{cmyk}{0.85,0,0.33,0}
\definecolor{RawSienna}{cmyk}{0,0.72,1,0.45}
\definecolor{Magenta}{cmyk}{0,1,0,0}

\usepackage[shortlabels]{enumitem}
\newtheorem{proposition}{Proposition}[section]
\newtheorem{theorem}{Theorem}[section]

\newtheorem{corollary}{Corollary}[section]
\newtheorem{lemma}{Lemma}[section]
\newtheorem{remark}{Remark}[section]
\newtheorem{example}{Example}[section]

\numberwithin{equation}{section}

\allowdisplaybreaks

\title[The isometry of symmetric-Stratonovich integral]{The isometry of symmetric-Stratonovich integrals w.r.t. Fractional Brownian motion $H< \frac{1}{2}$}

\author{Alberto OHASHI$^1$}
\author{Francesco RUSSO$^2$}
\author{Frederi Viens$^3$}

\address{$1$ Departamento de Matem\'atica, Universidade de Bras\'ilia, 13560-970, Bras\'ilia - Distrito Federal, Brazil}\email{amfohashi@gmail.com}
\address{$2$ ENSTA ParisTech, Institut Polytechnique de Paris,
 Unit\'e de Math\'ematiques appliqu\'ees, 828, boulevard des Mar\'echaux, F-91120 Palaiseau, France}
 \email{francesco.russo@ensta-paris.fr}
\address{$3$ Department of Statistics, Rice University, 77251-1892, Houston - Texas, USA}\email{frederi.viens@rice.edu }

\begin{document}

\begin{abstract}
In this work, we present a detailed analysis on the exact expression of the $L^2$-norm of the symmetric-Stratonovich stochastic integral driven by a multi-dimensional fractional Brownian motion $B$ with parameter $\frac{1}{4} < H < \frac{1}{2}$. Our main result is a complete description of a Hilbert space of integrand processes which realizes the $L^2$-isometry where none regularity condition in the sense of Malliavin calculus is imposed. The main idea is to exploit the regularity of the conditional expectation of the tensor product of the increments $B_{t-\delta,t+\delta}\otimes B_{s-\epsilon,s+\epsilon}$ onto the Gaussian space generated by $(B_s,B_t)$ as $(\delta,\epsilon)\downarrow 0$. The Hilbert space is characterized in terms of a random Radon $\sigma$-finite measure on $[0,T]^2$ off diagonal which can be characterized as a product of a non-Markovian version of the stochastic Nelson derivatives. As a by-product, we present the exact explicit expression of the $L^2$-norm of the pathwise rough integral in the sense of Gubinelli. 

\end{abstract}



\maketitle


\tableofcontents

\section{Introduction}

Let $B = (B^{(1)}, \ldots, B^{(d)})$ be a $d$-dimensional Fractional Brownian motion (henceforth abbreviated by FBM) with exponent $\frac{1}{4}< H < \frac{1}{2}$ on a probability space $(\Omega,\mathcal{F},\mathbb{P})$ over the positive line $\mathbb{R}_+$. That is, $B$ is a $d$-dimensional Gaussian process with covariance

$$\mathbb{E}\big[ B^{(i)}_s B^{(j)}_t \big]= \frac{1}{2}\{t^{2H}+ s^{2H}-|t-s|^{2H}\}\delta_{ij},$$
for $1\le i,j\le d$. We set 

$$R(s,t):= \frac{1}{2}\{t^{2H}+ s^{2H}-|t-s|^{2H}\}$$
for $(s,t) \in \mathbb{R}_+^2$ and $\gamma(s) := R(s,s)$. In the sequel, increments of paths $t\mapsto f_t$ will be denoted by $f_{s,t}:= f_t-f_s$. 

For a given positive constant $T$, let $g:[0,T]\times \mathbb{R}^d\rightarrow \mathbb{R}^d$ be a Borel function, where $\mathbb{R}^d$ is equipped with the usual inner product $\langle \cdot, \cdot \rangle$. For a given $\epsilon>0$, let us denote

\begin{equation}\label{Ioapp}
I^0(\epsilon,g,dB):= \frac{1}{2\epsilon}\int_0^T \langle g(s,B_s), B_{s-\epsilon, s+\epsilon}  \rangle ds,
\end{equation}
where we set $B_r:=0$ for $r\le 0$. In this paper, we study the existence of the symmetric-Stratonovich integral in the sense of stochastic calculus of regularization (see e.g. \cite{russo1993forward}):
\begin{equation}\label{epsilonconv}
\int_0^T g(t,B_t)d^0B_t:=\lim_{\epsilon\downarrow 0} I^0(\epsilon,g,dB),
\end{equation}
when the limit exists in $L^2(\mathbb{P})$. Significantly, we present a concrete, explicitly defined separable Hilbert space $\mathcal{L}_R(\mathbb{R}^d)$ associated with the covariance structure $R$ such that

$$
\mathbb{E}\Bigg| \int_0^T g(t,B_t)d^0B_t\Bigg|^2 = \| g (\cdot,B)\|^2_{\mathcal{L}_R(\mathbb{R}^d)},
$$
for a large class of functions $g$ admitting only integrability conditions w.r.t. $R$. We choose to work with the $R$ defined above, which is specific to FBM, though our methodology could be easily adapted to many other examples of Gaussian processes with H\"older-continuous paths, particularly those which satisfy a form of self-similarity. For the sake of simplicity, in this paper, we will not discuss the technical aspects of this point further.



The Hilbert space isometry structure of stochastic integrals has been one of the fundamental pillars in probability theory since the pioneering works of It\^o \cite{ito} and Kunita and Watanabe \cite{KW}. As a reminder to the reader, in the case of integration against martingales, the isometry is constructed via the fundamental and classical concepts of filtration $\mathbb{F} = (\mathcal{F}_t)_{t\ge 0}$ and quadratic variation $[X]$ as 

\begin{equation}\label{SMiso}
\mathbb{E}\Bigg| \int_0^T Y_tdX_t \Bigg|^2 = \mathbb{E}\int_0^T |Y_t|^2 d[X]_t,
\end{equation}
where the driving noise $X$ is an $\mathbb{F}$-square integrable continuous martingale and $Y$ is an $\mathbb{F}$-progressively measurable adapted process such that the right-hand side of (\ref{SMiso}) is finite. The integral on the left-hand side of (\ref{SMiso}) is the stochastic integral in the sense of It\^o \cite{ito} and Kunita and Watanabe \cite{KW}. In this case, the underlying Hilbert space is $L^2(\Omega\times [0,T], \mu_X)$ w.r.t. the measure

$$\mu_X (A) := \mathbb{E}\int_0^T \mathds{1}_A(s)d[X]_s,$$
for a measurable subset $A \subset \Omega\times [0,T]$. The isometry (\ref{SMiso}) gives a fundamental tool for the study of strong solutions of stochastic differential equations (SDEs) for semimartingale driving processes. The combination of the isometry with the regularity of the measure $\mu_X$ provide the minimal regularity conditions for well-posedness of SDEs. One notes in (\ref{SMiso}) how the existence of a finite quadratic variation is essential to the isometry, and equivalently, is essential to defining the associated Hilbert space and its component measure $\mu_X$ defined above.

Beyond the semimartingale case, the stochastic integration is much more involved, not least because processes which are more irregular than Martingales typically have no quadratic variation. Many techniques have been employed over the last decades to circumvent this issue, sometimes proposing alternative Hilbert-space representations which might intuitively be appropriate for expressing some simplified form of isometry. Let $X$ be a continuous
Gaussian process with positive definite covariance kernel $K$ and let $H(K)$ be the Reproducing Kernel Hilbert space (henceforth abbreviated by RKHS) of $X$ (see e.g. Chapter 3 in \cite{adler}). In this work, with a slight abuse of notation, any Hilbert space which is isometrically isomorphic to $H(K)$ will be called as RKHS.

Starting with a covariance kernel $K$, one can construct a canonical Hilbert space $\mathcal{H}_X$ defined as the closure of the set of all stepwise constant functions via the scalar product 

$$\langle \mathds{1}_{[0,t]},\mathds{1}_{[0,s]} \rangle_{\mathcal{H}_X} = K(s,t),$$ 
for every $0\le s,t\le T$. The Hilbert space $\mathcal{H}_X$ is isometrically isomorphic to $H(K)$. One should observe that $\mathcal{H}_X$ is the fundamental object in the study of the stochastic calculus for Gaussian processes. Indeed, any stochastic integral should satisfy

\begin{equation}\label{gele}
\int_0^T f_sdX_s = \sum_{i=1}^n a_i X_{t_i, t_{i+1}},
\end{equation}
whenever $f = \sum_{i=1}^na_i \mathds{1}_{(t_i,t_{i+1}]}$ is a stepwise constant deterministic function. More importantly, a stochastic integral of a deterministic $f$ w.r.t. $X$ is the limit in $L^2(\mathbb{P})$ for a sequence of the form (\ref{gele}) if, and only if, $f \in \mathcal{H}_X$. It is in that narrow sense, applicable to integrals with deterministic integrands, that the RKHS captures a form of isometry for integrating against FBM and other Gaussian processes.

There is, however, one important and well-known case of a Gaussian process where $\mathcal{H}_X$ provides the exact class of integrands much beyond the deterministic ones. That case is the classical Brownian motion, and in that case, the exact class is the space $L^2(\Omega; \mathcal{H}_X)$ of the $\mathcal{H}_X$-valued square-integrable random elements. In other words, in reference to the measure $\mu_X$ discussed earlier, for Brownian motion, we have $L^2(\Omega\times [0,T], \mu_X) = L^2(\Omega; \mathcal{H}_X)$, where $\mathcal{H}_X = L^2([0,T])$. Beyond the Brownian motion case, the Malliavin calculus has been employed towards a Hilbert space theory of stochastic integrals for Gaussian processes (in particular FBM with $H \neq \frac{1}{2}$) in terms of the so-called divergence operator $(\boldsymbol{\delta}, \text{dom}~\boldsymbol{\delta})$ which, when restricted to deterministic integrands in $\mathcal{H}_X\subset \text{dom}~\boldsymbol{\delta}$, is the so-called Paley-Wiener integral which realizes  

$$\|\boldsymbol{\delta}(f)\|^2_{L^2(\mathbb{P})} = \| f\|^2_{\mathcal{H}_X},$$
for every $f \in \mathcal{H}_X$. By using the multiplication rule of smooth random variables with the divergence operator (see e.g. Prop 1.3.3 in \cite{nualart2006}) and/or interpreting $\boldsymbol{\delta}$ as an integral (inspired by the Brownian motion case), many authors have developed a stochastic calculus for Gaussian driving noises for stochastic processes in the Malliavin-Watanabe space $\mathbb{D}^{1,2}(\mathcal{H}_X)$ (see e.g. \cite{nualart2006} for the definition). The drawback of this approach relies on the fact that $\mathbb{D}^{1,2}(\mathcal{H}_X)$ is much smaller than $L^2(\Omega; \mathcal{H}_X)$ and it requires severe regularity conditions which prevents the solvability of multi-dimensional non-linear SDEs. 

One also notes that the coincidence between the divergence operator and the Ito integral for adapted integrands, which is a motivation for interpreting $\boldsymbol{\delta}(f)$ as an integral, clashes with the fact that it is the adjoint of a derivative operator, which is a motivation for thinking of $\boldsymbol{\delta}(f)$ as a gradient, rather than an integral. This has not prevented a number of authors from working in the direction of interpreting $\boldsymbol{\delta}(f)$ as an integral, see e.g. \cite{decreusefond}, \cite{alos2001stochastic}, \cite{carmona2003}, \cite{viensSK}, \cite{cheridito2005}, \cite{cass2019}, \cite{tindelsong} and other references therein, though it does provide an indication of why this interpretation is fraught with technical difficulties. 

Stochastic calculus for real-valued FBM for $\frac{1}{6} < H< \frac{1}{2}$ in the sense via regularization was studied by \cite{gnrv}, \cite{gradno}, \cite{grv} and \cite{cheridito2005} for smooth integrands in the one-dimensional case. In this approach, the symmetric-Stratonovich integral exists if, and only if, $\frac{1}{6} < H < \frac{1}{2}$.
Connections with the divergence operator was established by \cite{alosleon}. See also the monographs \cite{nualart2006}, \cite{biaginibook} and \cite{Russo_Vallois_Book} and other references therein. Defining stochastic integrals using this class of approaches is closer to an intuitive notion of integral than is the notion of divergence in the Malliavin calculus, and is arguably a precursor to fully pathwise methods.


A pathwise integration w.r.t.~paths of Gaussian processes (including FBM with $H < \frac{1}{2}$) can be implemented in the sense of \cite{gubinelli2004} and a complete theory of differential equations in the sense of rough paths is now well-understood. See e.g \cite{friz}, \cite{hairerbook} and other references therein. Recently, a mixed theory of rough paths based on L\^e's so-called stochastic sewing lemma \cite{Le} has been developed by \cite{friz2020rough}, \cite{friz2021}, \cite{gerencser} and \cite{matsuda}. In particular, based on an extension of \cite{Le}, Matsuda and Perkowski \cite{matsuda} prove $L^2(\mathbb{P})$-convergence of the Stratonovich integral for $\frac{1}{4}< H < \frac{1}{2}$ (see Prop. 3.5 in \cite{matsuda}) for the class of integrands of the form $g(B)$, where $g$ is any bounded and globally $\gamma$-H\"older continuous function with $\gamma > \frac{1}{2H}-1$. The authors also provide $L^2(\mathbb{P})$-convergence in the regime $\frac{1}{6} < H \le \frac{1}{4}$ under the assumptions that $\nabla g$ is $\gamma-1$-H\"older continuous (note that $\gamma$ is larger than 1 in this case) and $\nabla g$ is symmetric.

Our work in the present paper takes a rather different strategy compared to the previous literature. The main goal is the description of the exact Hilbert space structure of the symmetric-Stratonovich integral solely based on the underlying RKHS, without relying on Malliavin-Watanabe spaces, and therefore, without imposing unnecessary regularity assumptions on the level of the underlying Wiener space. Our philosophy is close in spirit to the theory developed by It\^o in the classical Brownian motion case, in that the Hilbert space structure is expressed via an isometry relation w.r.t.~the underlying RKHS. It is well-known that the forward stochastic integral
$\int_0^\cdot B d^-B$
does not exist
for $H < \frac{1}{2}$, see e.g.~\cite{alosleon} and Lemma 6.1
of \cite{Russo_Vallois_Book}. Therefore, we choose to focus our article on the symmetric-Stratonovich integral (\ref{epsilonconv}), which is known to exist (see e.g.~\cite{gnrv}, \cite{cheridito2005}), and which has a natural interpretation as a true integral. By deriving a fully explicit isometry formula for the symmetric-Stratonovich, our work is the necessary starting point for a program which develops the stochastic calculus for this integral w.r.t.~Gaussian processes along a path (no pun intended) which attempts to mirror the original path of It\^o for Brownian motion.  In order to illustrate our results in the simplest possible way, avoiding unnecessary technicalities, we choose to work with FBM, but as mentioned earlier, our methodology will certainly hold for a larger class of Gaussian processes.

There already exist some characterizations of the RKHS of FBM for $0 < H < \frac{1}{2}$ described in terms of fractional integrals. In this direction, see e.g \cite{nualart2006}, \cite{biaginibook} and other references therein. In this approach, FBM is viewed as a functional of a standard Brownian motion via a singular kernel. In fact, scholars working in this direction know that much work can be achieved using so-called Wiener chaos calculus, by converting multiple Wiener integrals w.r.t.~FBM into multiple Wiener integrals w.r.t.~the ordinary Brownian motion. This does not, however, address the nature of integration w.r.t.~FBM in its ``native'' sense. 

In the present work, we adopt instead the intrinsic characterization of the RKHS in terms of the covariance kernel $R$ as described in \cite{krukrusso}. In this approach, the RKHS (henceforth denoted by $L_R(\mathbb{R}^d))$ is equipped with the norm 
\begin{equation}\label{rkhsnorm}
\| f\|^2_{L_R(\mathbb{R}^d)} = \int_0^T |f_t|^2 m(dt) +\frac{1}{2}\int_{[0,T]^2}|f_{s,t}|^2|\mu|(dsdt),
\end{equation}
where 
\begin{equation}\label{measures}
m(dt):= \frac{\partial R}{\partial t}(T,t)dt,\quad |\mu|(dsdt) := \Big|\frac{\partial^2 R}{\partial t \partial s}(s,t)\Big|dsdt,
\end{equation}
and $\frac{\partial^2 R}{\partial t \partial s}(s,t) <0$ and $\frac{\partial R}{\partial t}(T,t)>0$ for $s\neq t$. See \cite{krukrusso}, \cite{OhashiRusso} and Section \ref{Isosection} herein for further details on the space $L_R(\mathbb{R}^d)$ and the nature of its typical elements. Since $m$ is a positive finite measure, the leading term in the $L^2$ norm in (\ref{rkhsnorm}), which constrains the pathwise singularity of the elements $f$ in $L_R(\mathbb{R}^d)$, is the term involving the $\sigma$-finite measure $|\mu|$ since the increments of those functions $f$ must be integrable against it. To be precise, the action of this term on functions $f$ in (\ref{rkhsnorm}) 
can be identified as a Schwartz distribution on $[0,T]^2$ but it is also a $\sigma$-finite negative measure on $[0,T]^2$ off the diagonal, and is absolutely continuous w.r.t.~Lebesgue measure. See \cite{krukrusso} and \cite{OhashiRusso} for details. In this sense, the Hilbert space $L^2(|\mu|)$ of functions in $[0,T]^2$ fully describes $L_R(\mathbb{R}^d)$ modulo an integrability condition w.r.t.~$m$. One can check (see Prop. 6.9 in \cite{krukrusso}) that the class of continuous functions over $[0,T]$ with modulus of continuity $V$ such that

\begin{equation}\label{rkhintr}
\int_{[0,T]^2} V^2(|t-s|) |\mu|(dsdt)< \infty
\end{equation}
is contained in $L_R(\mathbb{R}^d)$. A typical modulus of continuity $V$ satisfying (\ref{rkhintr}) has the form $V(y)  = y^{\frac{1}{2}-H} \psi(y)$ for a strictly increasing function $\psi$ in a neighborhood of zero such that

\begin{equation}\label{fCONDINTR}
\int_0^\eta \frac{1}{y}\psi^2(y)dy < \infty,
\end{equation}
for some $0 < \eta <1$. Thus, $L_R(\mathbb{R}^d)$ contains a rather general class of continuous functions $f$ with modulus of continuity ranging from $\gamma$-H\"older-type with $\gamma > \frac{1}{2}-H$ to more elaborated ones satisfying (\ref{fCONDINTR}), such as cases where $\psi$ is a sufficiently integrable slowly varying function (see e.g. Chapter 1 in \cite{bingham}) but $f$ is not $\gamma$-H\"older-continuous for any $\gamma > \frac{1}{2}-H$. See the examples and remarks in Section \ref{examplesection} for details.

As mentioned, this article presents a natural isometry property (\ref{isoINTR}) associated with the underlying RKHS in the spirit of It\^o's classical Hilbert space theory of stochastic integration. Our results show that the $L^2(\mathbb{P})$-norm of the Stratonovich integral w.r.t.~FBM for $\frac{1}{4} < H< \frac{1}{2}$ is described by a well-defined set of functions whose regularity is largely described by the regularity of elements in the RKHS $L_R(\mathbb{R}^d)$ which we described broadly in the previous two paragraphs. In this sense, our isometry shares a similar structure to the classical Brownian motion case $H=\frac{1}{2}$. This article indicate that the role played by the RKHS in the stochastic calculus for Gaussian processes is most likely universal for the whole class of FBM with $\frac{1}{4} < H < \frac{1}{2}$ since the isometry formula (\ref{isoINTR}) fully describes a large and precise class of integrands for the stochastic Stratonovich integral. 

The remainder of this introduction presents quantitative details of the claims in the above paragraph. The technique used in the present work is the introduction of a polarization scheme to analyze the limits in (\ref{Ioapp}) and (\ref{epsilonconv}); this requires a careful analysis of the existence and characterization of the limits  
$$\Lambda(s,t):=\lim_{(\epsilon,\delta)\downarrow 0}\frac{1}{4\epsilon\delta}\mathbb{E}\Big[B_{s-\epsilon,s+\epsilon}\otimes  B_{t-\delta,t+\delta}\big| B_s,B_t\Big]$$
and
\begin{equation}\label{limiINTR}
\Big\| \int_0^T g(t,B_t)d^0B_t\Big\|^2_{L^2(\mathbb{P})} = \lim_{(\epsilon,\delta)\downarrow 0}\langle I^0(\epsilon,g,dB), I^0(\delta,g,dB)\rangle_{L^2(\mathbb{P})},
\end{equation}
where $\otimes$ denotes the tensor on $\mathbb{R}^d$ and the pair $(\epsilon,\delta)\downarrow 0$ goes \textit{simultaneously} to zero. 
Theorems \ref{mainTH1} and \ref{mainTH2} 
show the $L^2(\mathbb{P})$-norm of the stochastic Stratonovich integral defined in (\ref{Ioapp}) is described by two matrix-valued terms
\begin{equation}\label{capPsi}
\Psi_t:=\frac{\partial R}{\partial t}(t,T) I_{d\times d} + \mathcal{M}_t
\end{equation}
and 
\begin{equation}\label{capLambda}
\Lambda(s,t)= \frac{\partial^2 R}{\partial t\partial s}(s,t) I_{d\times d} + \mathcal{W}(s,t),
\end{equation}
for $(s,t) \in [0,T]^2$ with $s\neq t$, where $I_{d\times d}$ is the identity matrix and where $\mathcal{M}$ and $\mathcal{W}$ are explicit stochastic processes which live in the second Wiener chaos of the underlying FBM (see their definitions in the statement of Theorem \ref{mainTH1}). A comparison with the expression for the RKHS's norm (\ref{rkhsnorm}) shows that the two terms $\frac{\partial R}{\partial t}(t,T)$ and $ \frac{\partial^2 R}{\partial t\partial s}(s,t)$ which appear in (\ref{capPsi}) and (\ref{capLambda}) are identical to the two terms which fully define the RKHS's topology via the the $L^2(\mathbb{P})$-isometry for the Paley-Wiener integral, i,e. when integrands are deterministic. The case of stochastic integrands which is the object of study in the current paper, is for integrands of the form $Y_\cdot = g(\cdot,B_\cdot)$. In this paper, we defined a Hilbert space  $\mathcal{L}_R(\mathbb{R}^d)$ of such stochastic integrands for which the integral in (\ref{Ioapp}) exists (see Theorem (\ref{epsilonconv})), and we show that the $L^p(\mathbb{P})$-norms of the stochastic components $\mathcal{M}$ and $\mathcal{W}$ in each of the two fundamental quantities (\ref{capPsi})) and (\ref{capLambda}) which appear in our isometry formula are in fact dominated by their associated deterministic counterparts $\frac{\partial R}{\partial t}(t,T)$ and $\frac{\partial^2 R}{\partial t\partial s}(s,t)$, respectively.



We show there exists an explicit deterministic Lebesgue $q$-integrable kernel $\kappa$ over $[0,T]^2$ for every $1 < q < \frac{1}{1-H}$ such that 

\begin{equation}\label{MWest2}
\| \mathcal{M}_t\|_2\lesssim \frac{\partial R}{\partial t}(t,T),\quad \| \mathcal{W}(s,t)\|_2\lesssim \Big|\frac{\partial^2 R}{\partial t \partial s}(s,t)\Big| + \kappa(s,t)
\end{equation}
for $(s,t) \in [0,T]^2$ with $s \neq t$ and $\|\cdot\|_p$ denotes the standard $L^p(\mathbb{P})$-norm for $1\le p < \infty$.  By the equivalence of $\|\cdot\|_p$-norms on the second Wiener chaos, we have  


\begin{equation}\label{domin1}
\|\Psi_t\|_p \approxeq  \frac{\partial R}{\partial t}(t,T),
\end{equation}
\begin{equation}\label{domin2}
\Big|\frac{\partial^2 R}{\partial t \partial s}(s,t)\Big| \lesssim \| \Lambda(s,t)\|_p \lesssim \Big|\frac{\partial^2 R}{\partial t \partial s}(s,t)\Big| + \kappa(s,t) 
\end{equation}
for every $s\neq t$ and $1 < p < \infty$. 
Observe 
$$L^2(\Omega; L_R(\mathbb{R}^d)) = \{Y:\Omega\rightarrow L_R(\mathbb{R}^d); \| Y_\cdot\|_2 \in L^2(m)~\text{and}~\|Y_{\cdot,\cdot\cdot}\|_2 \in L^2(|\mu|) \}.$$ 
This article concentrates on the subset $\mathbb{L}^2(\Omega; L_R(\mathbb{R}^d))$ of processes $Y \in L^2(\Omega; L_R(\mathbb{R}^d))$ such that $Y_\cdot = g(\cdot, B_\cdot)$ for some $g$. The fact that $m(dt)$ is a positive finite measure imposes almost none restrictions on $\| Y_\cdot \|_2$ and the main constraint imposed on $\mathbb{L}^2(\Omega; L_R(\mathbb{R}^d))$ comes from the condition $\|Y_{\cdot,\cdot\cdot}\|^2_2 \in L^2(|\mu|)$, i.e., 

\begin{equation}\label{constRKHS}
\int_{[0,T]^2} \| Y_{s,t} \|^2_2 |\mu|(dsdt) < \infty. 
\end{equation}
The nontrivial question to be answered is on the validity of the inclusion 
$$\mathbb{L}^2(\Omega,L_R(\mathbb{R}^d)) \subset \mathcal{L}_R(\mathbb{R}^d).$$ 
The main results of this article demonstrate this is almost the case. Theorems \ref{mainTH1}, \ref{mainTH2}, (\ref{domin1}) and (\ref{domin2}) show that typical stochastic integrands $Y \in \mathbb{L}^2(\Omega, L_R(\mathbb{R}^d))$ for which there exists $p>1$ such that 

\begin{equation}\label{Certainly_condition}
\int_{[0,T]^2} \| Y_{s,t} \|^2_{2p} |\mu|(dsdt) < \infty~\text{and}~\sup_{0\le s,t\le T}\| Y_{s,t}\|_{2p} < \infty
\end{equation}
belong to the domain $\mathcal{L}_R(\mathbb{R}^d)$ whose topology realizes the isometry (\ref{isoINTR}). The reader should observe (\ref{Certainly_condition}) is only slightly stronger than the regularity condition (\ref{constRKHS}) which fully describes $\mathbb{L}^2(\Omega, L_R(\mathbb{R}^d))$.  

In any case, in our chosen state-dependent framework where the integrand $Y_t  = g(t,B_t)$, for this $Y$ to belong to $\mathcal{L}_R(\mathbb{R}^d)$, thanks to (\ref{domin1}) and (\ref{domin2}), the space regularity of the function $g$ is essentially inherited from that of functions in the RKHS, where the \textit{marginal} $t \mapsto Y_t$ and \textit{increment} $(s,t) \mapsto Y_{s,t}$ stochastic dependence of the integrands are inherited from the deterministic terms $\frac{\partial R}{\partial t}(t,T)$ and $\frac{\partial^2 R}{\partial t \partial s}(s,t)$, respectively, in the RKHS. The reader will consult (\ref{MARINsplit}) and (\ref{split}) in section \ref{strategy}, as well as Lemmas \ref{sharpW}, \ref{West} and \ref{almRK}, for all technical details on this near complete coincidence between the regularity of the deterministic functions in the RKHS and the regularity of their stochastic counterparts in $\mathcal{L}_R(\mathbb{R}^d)$. These details include the regularity of the aforementioned kernel $\kappa$, and the need to work in an $L^{2p}(\mathbb{P})$ space with $p>1$ because of the randomness of $\mathcal{M}$ and $\mathcal{W}$.

We stress that the present article does not rely on $L^p(\mathbb{P})$-bounds ($p\ge 2$) on the 3-increments of a two-parameter process commonly used in the rough path philosophy in the sense of Gubinelli \cite{gubinelli2004} or in the recent stochastic sewing lemma approach (see e.g \cite{Le} and \cite{matsuda}). As a result, we are able to implement a finer analysis than that of the $\gamma$-H\"older scale of regularity imposed on the integrand functions for $\gamma > \frac{1}{2H}-1$. Certainly, our integrands are still constrained by the integrability condition (\ref{Certainly_condition}) inherited from the RKHS $L_R(\mathbb{R}^d)$, for some $p>1$, as seen in Theorem \ref{mainTH2}. However, as seen via the examples given in Section \ref{examplesection}, the condition (\ref{Certainly_condition}) is satisfied for integrands of the form $Y = g(B)$, where $g$ is not $\gamma$-H\"older for any $\gamma > \frac{1}{2H}-1$, because their modulus of continuity is equal to $r \mapsto r^{\frac{1}{2H}-1} L(r)$ where $L$ is a slowly-varying function (for definition see e.g Chapter 1 of \cite{bingham}). 

One of the important consequences of Theorem \ref{mainTH1} is given by Corollary \ref{roughCOR} which presents the exact $L^2(\mathbb{P})$-norm of the pathwise rough integral in the sense of \cite{gubinelli2004} only as a function of the marginals (0-increments) and first-order increments of the integrands and \emph{not} on Gubinelli's derivative. This is possible due to the recent investigations by Ohashi and Russo \cite{OhashiRusso} and Liu, Selk and Tindel \cite{tindelliu}, where the authors connect symmetric - Stratonovich calculus with rough integrals in the sense of Gubinelli\cite{gubinelli2004}. 

To the best of our knowledge, our paper is the first to provide an exact isometry formula for the $L^2(\mathbb{P})$-norm of the symmetric-Stratonovich stochastic integral w.r.t.~a Gaussian process, which only involves integrability conditions of the integrand in the spirit of the the underlying RKHS. We advocate for the development of a natural stochastic calculus for Gaussian processes, in that it should remain only on the level of the RKHS and the covariance function. This paper reveals that this is possible for FBM in the regime $\frac{1}{4} < H < \frac{1}{2}$. Ours being the first foray in this direction, we have chosen to consider only state-dependent integrands, though we believe an extension of our Hilbert space $\mathcal{L}_R(\mathbb{R}^d)$ and our isometry formula to more general integrands should be possible.

We conclude this introduction with a note on the limits of our strategy of exploiting the Hilbert space structure inherited from the RKHS. In the case $0 < H \le \frac{1}{4}$, we observe that $B \not\in L^2(\Omega; L_R(\mathbb{R}^d))$ due to the fact that 
$$\int_{[0,T]^2}\mathbb{E}|B_{s,t}|^2 |t-s|^{2H-2}dsdt=+\infty.$$
Therefore, in the regime $0 < H \le \frac{1}{4}$, one should adopt a different strategy to obtain a meaningful isometry property. This is beyond the scope and the intent of this article. In particular, noting that some of the literature results quoted earlier in this introduction are able to handle the case $\frac{1}{6} < H \le \frac{1}{4}$, such as via the use of rough paths theory, stochastic sewing lemma or via the Malliavin calculus, we believe that no correspondence can be drawn between the techniques needed to descend into that range of $H$ and the natural extension of the stochastic calculus for Gaussian processes for which this paper represents an initial step.

\

\noindent \textbf{Notation}: In this article, $0 < T < \infty$ is a positive finite terminal time, 

$$[0,T]^2_\star  := \{(s,t) \in [0,T]^2; s\neq t\},$$

$$\Delta_T:=\{(s,t) \in [0,T]_\star; 0\le s < t\le T\},$$ 
and the increment of a one-parameter function $f$ is denoted by $f_{s,t}:= f_t - f_s$ for $(s,t) \in \mathbb{R}^2_+$ Throughout this article, we write $(\epsilon,\delta)\downarrow 0$ to denote $\epsilon \downarrow 0$ and $\delta \downarrow 0$ simultaneously. We further write $a\lesssim b$ for two positive quantities to express an estimate of the form $a \le C b$, where $C$ is a generic constant which may differ from line to line. The sign function is denoted by $\text{sgn}$. We also write $a\wedge b = \min\{a;b\}$ and $a\vee b = \max\{a;b\}$ for any real numbers $a,b$. Moreover, for any two natural numbers $(i,j)$, we write $\delta_{ij} = 1$ when $i=j$ and $\delta_{ij}=0$, otherwise. The identity matrix of size $d$ is denoted by $I_{d\times d}$.

If $a = (a^i)^d_{i=1}, b = (b^j)_{j=1}^d \in \mathbb{R}^d$, then $a\otimes b:=(a^ib^j)_{1\le i,j\le d}$ denotes the tensor product of $(a,b)$, $\langle \cdot, \cdot \rangle $ denotes the standard Frobenius inner product on the space of $n\times m$-matrices for $n,m\ge 1$. Whenever clear from the context, $|\cdot|$ denotes any norm of a finite-dimension vector space. For a given column vector $a$, $a^\top$ denotes the transpose of $a$. For a random variable $X$, we denote $\|X\|^p_p:= \mathbb{E}|X|^p$ for $1\le p < \infty$. If $\theta = (\theta_i)_{i=1}^d$ and $\xi = (\xi_j)_{j=1}^d$ are two column random vectors, then the joint covariance matrix of $(\theta,\xi)$ is denoted by 

$$\text{Cov} \big( \theta, \xi \big): = \text{cov}\big(\theta_i, \xi_j \big); 1\le i\le k,~1\le j\le \ell,$$
where $\text{cov}$ is the covariance operation. We will denote 

$$\varphi(s,t):= R(s,t) - \gamma(s),$$
for $(s,t) \in \Delta_T$, where $\gamma(s) = R(s,s)$.

\section{Main results}\label{MRdisc}

This section presents the main results of this paper. Fix $\frac{1}{4} < H < \frac{1}{2}$. In the sequel, we denote

\begin{equation}\label{dmetric}
\textbf{d}^{\bar{\gamma},\gamma}\big( (t,x);(s,y) \big):= |t-s|^{\bar{\gamma}} + |x-y|^\gamma,
\end{equation}
for $(t,x),(s,y) \in [0,T]\times \mathbb{R}^d$. For $0 < \gamma,\bar{\gamma} \le 1$, let $C^{\bar{\gamma},\gamma}([0,T]\times \mathbb{R}^d)$ be the space of continuous functions $g:[0,T]\times \mathbb{R}^d\rightarrow \mathbb{R}^d$ equipped with the seminorm  

$$\|g\|_{\bar{\gamma},\gamma}:=\sup_{(t,x)\neq (s,y)}\frac{|g(t,x) - g(s,y)|}{\textbf{d}^{\bar{\gamma},\gamma}\big( (t,x);(s,y)\big)} < \infty.$$


Let $\Theta_{s,t}$ be the determinant of the covariance matrix of the Gaussian vector $(B^{(1)}_s, B^{(1)}_t)$. In the sequel, we conveniently recall the definition of the Stratonovich integral $\int Yd^0B$ given by (\ref{epsilonconv}) as the $L^2(\mathbb{P})$-limit of the Riemann integrals $\lim_{\epsilon \downarrow 0}I^0(\epsilon, Y,dB)$ for integrands of the form $Y_t = g(t,B_t)$. 

\begin{theorem}\label{mainTH1}
Let $g \in C^{\bar{\gamma},\gamma}([0,T]\times \mathbb{R}^d)$ with $\gamma > \frac{1}{2H}-1$, $\bar{\gamma}> \frac{1}{2}-H$ and $Y_t=g(t,B_t)$. Then, $\lim_{\epsilon\downarrow 0}I^{0}(\epsilon,Y,B)$ exists in $L^2(\mathbb{P})$ and

\begin{equation}\label{isoFORMULA}
\mathbb{E}\Bigg| \int_0^T Y_td^{0}B_t\Bigg|^2 = \mathbb{E}\| Y\|^2_{L_R(\mathbb{R}^d)} -\frac{1}{2}\mathbb{E}\int_{[0,T]^2} \big\langle Y_{s,t}\otimes Y_{s,t}, \mathcal{W}(s,t)\big\rangle dsdt
\end{equation}
\begin{equation*}\label{preiso}
+\mathbb{E}\int_0^T \big \langle Y_t\otimes Y_t, \mathcal{M}_t\big\rangle dt,
\end{equation*}
where

\begin{equation}\label{Mproc}
\mathcal{M}^{ij}_t:=\Big\{ \frac{B^{(i)}_tB^{(j)}_t-\textit{cov}\big(B^{(i)}_t;B^{(j)}_t\big)}{t^{2H}}\Big\}\frac{R(T,t)}{t^{2H}}\frac{d\gamma}{dt}(t)
\end{equation}


\begin{eqnarray}
\label{Wproc}\mathcal{W}^{ij}(s,t)&:=& \lambda_{11}(s,t)\lambda_{21}(s,t)\big(B^{(i)}_sB^{(j)}_s-\text{cov}(B^{(i)}_s; B^{(j)}_s)\big) \\
\nonumber&+& \lambda_{11}(s,t)\lambda_{22}(s,t) \big( B^{(i)}_s B^{(j)}_t - \text{cov}(B^{(i)}_s; B^{(j)}_t) \big)\\
\nonumber&+& \lambda_{12}(s,t) \lambda_{21}(s,t) \big( B^{(i)}_t B^{(j)}_s - \text{cov}(B^{(i)}_t; B^{(j)}_s) \big)\\
\nonumber&+& \lambda_{12}(s,t)\lambda_{22}(s,t)\big(B^{(i)}_t B^{(j)}_t - \text{cov}(B^{(i)}_t; B^{(j)}_t) \big),
\end{eqnarray}
for $1\le i,j\le d$, where 
$$\lambda_{11}(s,t) := \frac{1}{\Theta_{s,t}} \Big\{\frac{1}{2}\gamma'(s) \gamma(t) - R(s,t)\frac{\partial R}{\partial s}
 (s,t) \Big\},$$
$$\lambda_{12}(s,t):= \frac{1}{\Theta_{s,t}} \Big\{\frac{\partial R}{\partial s}(s,t)
 \gamma(s) - R(s,t)\frac{1}{2}\gamma'(s) \Big\},$$
$$\lambda_{21}(s,t):= \frac{1}{\Theta_{s,t}} \Big\{\frac{\partial R}{\partial t}
(t,s) \gamma(t) - R(s,t)\frac{1}{2}\gamma'(t) \Big\},$$
$$\lambda_{22}(s,t):=\frac{1}{\Theta_{s,t}} \Big\{\frac{1}{2}\gamma'(t) \gamma(s) - R(s,t)\frac{\partial R}{\partial t}
 (t,s) \Big\},$$
for $(s,t) \in [0,T]^2_\star$. 

\end{theorem}
The pointwise differentiability of $(s,t)\mapsto R(s,t)$ for $s\neq t$ plays an important role in the proof of Theorem \ref{mainTH1}. It provides the pointwise existence of $\mathcal{W}(s,t)$ which might be interpreted as a non-Markovian version of product of stochastic symmetric derivatives in the sense of Nelson \cite{nelson} w.r.t.~the sigma-algebra generated by the ``present'' of the FBM driving noise. See Remark \ref{nelsondisc} for further details.


As far as the H\"older-modulus of continuity is concerned, the conditions $\gamma > \frac{1}{2H}-1$ and $\bar{\gamma}> \frac{1}{2}-H$ in Theorem \ref{mainTH1} are sharp for the existence of (\ref{epsilonconv}) in the class of H\"older continuous functions. Indeed, we observe $\gamma H> \frac{1}{2}-H$ for every $\gamma > \frac{1}{2H}-1$, $C^\alpha \subset L_R(\mathbb{R}^d)$ for every $\alpha > \frac{1}{2}-H$ but this inclusion fails for $\alpha = \frac{1}{2}-H$ (see Prop 6.9 in \cite{krukrusso} and Chapter 5 in \cite{nualart2006}). We stress that $\mathcal{W}(s,t)dsdt$ is only a random $\sigma$-finite measure over $[0,T]^2$ and, for this reason, the existence of (\ref{epsilonconv}) requires a regularity condition on the integrand's increments which is dictated by $\Lambda$ via (\ref{domin2}). We choose to work with 

$$C_H^{+}:= \bigcup_{\bar{\gamma}> \frac{1}{2}-H, \gamma> \frac{1}{2H}-1 } C^{\bar{\gamma},\gamma} ([0,T]\times \mathbb{R}^d )$$
as the initial set where we construct the domain of the extension of the Stratonovich integral (\ref{epsilonconv}). 

In \cite{OhashiRusso}, the authors have proved equivalence of the rough integral  
$$(Y,Y')\mapsto \big(\int Y d\mathbf{X},Y\big)$$ 
in the sense of Gubinelli \cite{gubinelli2004} with the symmetric-Stratonovich integral in a suitable class of stochastic controlled processes $(Y,Y')$ (see \cite{GOR}) admitting some regularity conditions in the sense of Malliavin calculus and $\mathbf{X} = (B,\mathbb{X})$ is given in the Stratonovich form (\ref{secin}). Theorem 1.1 in \cite{OhashiRusso} covers the case $Y_t = g(B_t)$ and $Y'_t = \nabla g(B_t)$ if $g \in C^{2}_b(\mathbb{R}^d)$. Therefore, a direct implication of Theorem \ref{mainTH1} is the following result. 

\begin{corollary}\label{roughCOR}
Let $\mathbf{X} = (B,\mathbb{X})$ be the H\"older rough path associated with the FBM with $\frac{1}{3} < H < \frac{1}{2},$ where the area process $\mathbb{X}$ is given by

\begin{equation}\label{secin}
\mathbb{X}_{u,v} = \int_u^v (B_r-B_u)\otimes d^0B_r; ~0\le u\le v\le T.
\end{equation}
Suppose that $g \in C^{2}_b(\mathbb{R}^d)$ and let $Y=g(B)$. Then 
$$ \mathbb{E}\Bigg|\int_0^T Y_s d\mathbf{X}_s\Bigg|^2 = \mathbb{E}\| Y\|^2_{L_R(\mathbb{R}^d)} -\frac{1}{2}\mathbb{E}\int_{[0,T]^2} \big\langle Y_{s,t}\otimes Y_{s,t}, \mathcal{W}(s,t)\big\rangle dsdt
$$
$$+\mathbb{E}\int_0^T \big \langle Y_t\otimes Y_t, \mathcal{M}_t\big\rangle dt.$$

\end{corollary}

As discussed above, Theorem \ref{mainTH1} allows us to extend the notion of the Stratonovich integral (\ref{epsilonconv}) beyond the set $C^+_H$. Let $\mathcal{E}^0$ be the set of all processes of the form $g(t,B_t)$ where $g:[0,T]\times \mathbb{R}^d\rightarrow \mathbb{R}^d \in C^+_H$. We equip this vector space with the positive semi-inner product inherited from the right-hand side of (\ref{isoFORMULA}). Let $\mathcal{L}_R(\mathbb{R}^d)$ be the (Hilbert space) completion of $\big(\mathcal{E}^0, \langle \cdot, \cdot \rangle_{\mathcal{E}^0}\big)$. Then, we can extend the domain of the symmetric-Stratonovich stochastic integral from $\mathcal{E}^0$ to $\mathcal{L}_R(\mathbb{R}^d)$ via (\ref{isoFORMULA}). In the sequel, we denote this extension by

$$\int_0^T Y_t \hat{d}^0B_t,$$
for $Y \in \mathcal{L}_R(\mathbb{R}^d)$. We stress that $\int_0^T Y_t \hat{d}^0B_t = \int_0^T Y_t d^0B_t$, whenever $Y_\cdot = g(\cdot,B_\cdot) \in \mathcal{E}^0$. Then, we arrive at the isometry result.

\begin{proposition}\label{isoPROP}
The set $\mathcal{L}_R(\mathbb{R}^d)$ is a separable Hilbert space which realizes 

\begin{equation}\label{isoINTR}
\mathbb{E}\Bigg| \int_0^T Y_t \hat{d}^0B_t\Bigg|^2 = \| Y\|^2_{\mathcal{L}_R(\mathbb{R}^d)},
\end{equation}
for every $Y \in \mathcal{L}_R(\mathbb{R}^d)$.
\end{proposition}


In the sequel, we set 

\begin{equation}\label{rhoex}
\rho(dsdt):= |\mu|(dsdt) + \kappa(s,t)dsdt
\end{equation} 
for the  $q$-integrable kernel $\kappa$ as described in (\ref{MWest2}). For the precise expression of $\kappa$, we refer the reader to (\ref{kappaexp}) in Section \ref{Isosection}. 

\begin{theorem}\label{mainTH2}
If $g$ is a bounded function and for $Y_\cdot=g(\cdot,B_\cdot)$, there exists $p>1$ such that $(s,t) \ \mapsto \| Y_{s,t}\|_{2p}\in L^2(|\mu|)$, i.e.,  
\begin{equation}\label{uppernorm1}
\int_{[0,T]^2} \|Y_{s,t}\|^2_{2p} |\mu|(dsdt)<\infty, 
\end{equation}
then $Y \in \mathcal{L}_R(\mathbb{R}^d)$ and 
\begin{eqnarray}
\nonumber\|Y\|^2_{\mathcal{L}_R(\mathbb{R}^d)}&\lesssim& \int_0^T \|Y_t\|^{2}_{2p}m(dt) \\
\label{comIne}&+& \frac{1}{2}\int_{[0,T]^2} \|Y_{s,t}\|^2_{2p} \rho(dsdt).
\end{eqnarray}
\end{theorem}
We highlight the boundedness condition imposed in Theorem \ref{mainTH2} is just a technical restriction and it is not essential. One can make use of localization arguments as long as $(s,t)\mapsto \|Y_{s,t}\|_{2p} \in L^2(|\mu|)$ and $\sup_{t \in [0,T]}\| Y_t\|_{2p} < \infty $ for $p>1$ and, for conciseness, we restrict the analysis under this boundedness condition for the present paper. Theorem \ref{mainTH2} is the loose summary of Theorem \ref{mainTH3} up to a technical condition which resembles (\ref{uppernorm1}) shifted by constants in a small neighborhood of the origin at $\mathbb{R}_+\times \mathbb{R}^d$. For further details, see (\ref{fundASS2}) in Theorem \ref{mainTH3} and Remark \ref{UNcont}.

\subsection{Strategy of the proof}\label{strategy}
We consider the natural idea of exploring the underlying RKHS described in (\ref{rkhsnorm})  based on a polarization scheme and projecting $B_{s,t}\otimes B_{u,v}$ via conditional expectation onto the Gaussian space generated by $B_{c}, B_d$ where $s < c< t < u < d < v$. For a given integrand of the form $Y_t = g(t,B_t)$ and $(\epsilon,\delta) \in (0,1)^2$, we shall write 



$$
\Big \langle I^0(\epsilon,Y,dB), I^0(\delta,Y,dB) \Big\rangle_{L^2(\mathbb{P})} = \mathbb{E} \int_{[0,T]^2} \Big\langle \text{sym} \big( Y_s\otimes Y_t\big), \frac{1}{4\epsilon\delta}B_{s-\epsilon,s+\epsilon}\otimes  B_{t-\delta,t+\delta} \Big\rangle dsdt
$$
$$+ \mathbb{E} \int_{[0,T]^2} \Big\langle \text{ant} \big( Y_s\otimes Y_t\big), \frac{1}{4\epsilon\delta}B_{s-\epsilon,s+\epsilon}\otimes  B_{t-\delta,t+\delta} \Big\rangle dsdt:= I_{sym}(\epsilon,\delta) + I_{ant}(\epsilon,\delta),
$$
where $\text{ant} \big( Y_s\otimes Y_t\big)$ and $\text{sym} \big( Y_s\otimes Y_t\big)$ denotes the anti-symmetric and symmetric part of the matrix-valued process $Y_s\otimes Y_t$, respectively. The philosophy is to concentrate the analysis on the symmetric component (see Lemma \ref{antsymLemma}), where $\text{sym} \big( Y_s\otimes Y_t\big)$ can be represented (by polarization) as  
 
\begin{equation}\label{MARINsplit}
\text{sym} \big( Y_s\otimes Y_t\big)=\frac{1}{2} \Big\{ \underbrace{Y_s \otimes Y_s}_{\mathrm{marginal}} + \underbrace{Y_t\otimes Y_t}_{\mathrm{marginal}} -\underbrace{Y_{s,t}\otimes Y_{s,t}}_{\mathrm{1-increment}} \Big\}.
\end{equation}
Then, we can represent  
\begin{eqnarray}
\label{split}I_{sym}(\epsilon,\delta)&=& \frac{1}{2} \mathbb{E}\int_{[0,T]^2}  \Big\langle Y_s\otimes Y_s + Y_t\otimes Y_t , \frac{1}{4\epsilon\delta}B_{s-\epsilon,s+\epsilon}\otimes  B_{t-\delta,t+\delta} \Big\rangle dsdt\\
\nonumber&-&\frac{1}{2} \mathbb{E}\int_{[0,T]^2}  \Big\langle Y_{s,t}\otimes Y_{s,t} , \Lambda^0(\epsilon,\delta;s,t) \Big\rangle dsdt,
\end{eqnarray}
where 
$$\Lambda^0(\epsilon,\delta;s,t):=\frac{1}{4\epsilon\delta}\mathbb{E}\Big[B_{s-\epsilon,s+\epsilon}\otimes  B_{t-\delta,t+\delta}\big| B_s,B_t\Big].$$
The limit (as $(\epsilon,\delta)\downarrow 0$) of the marginal component in (\ref{split}) does not require too much regularity from $Y$ essentially because $\frac{\partial R}{\partial t}(T,t)dt$ is a finite positive measure. In the limit, the marginal component in (\ref{split}) gives rise to the one-parameter matrix-valued process

$$
\Psi_t= \mathcal{M}_t + \frac{\partial R}{\partial t}(T,t)I_{d\times d}. 
$$
defined in (\ref{capPsi}). The analysis of the marginal component in (\ref{split}) is investigated in Section \ref{secin}. The analysis of the increment component in (\ref{split}) is much more involved essentially because $\frac{\partial^2 R}{\partial t \partial s}(s,t)$ is a Schwartz distribution over $[0,T]^2$. We will also use the important fact it is a $\sigma$-finite measure restricted to $[0,T]^2_\star$ combined with the projection operator $\Lambda(\epsilon,\delta)$. The limit (as $(\epsilon,\delta)\downarrow 0$) of the increment component in (\ref{split}) requires regularity from the increments of $Y$ roughly of the form
$$\int_{[0,T]^2}\| Y_{s,t}\|^{2p}_{2p}\rho(dsdt) < \infty,$$
for some $p>1$, where $\rho$ is given by (\ref{rhoex}). In the limit, the increment component gives rise to the two-parameter matrix-valued process
\begin{eqnarray*}
\Lambda(s,t)&=&\lim_{(\epsilon,\delta)\downarrow 0}\Lambda^0(\epsilon,\delta;s,t)
\\
&=& \mathcal{W}(s,t) + \frac{\partial^2 R}{\partial t \partial s}(s,t)I_{d\times d }
\end{eqnarray*}
almost sure for each $s\neq t$ (see Lemma \ref{Lambdalimit}). Sections \ref{secsection} and \ref{mainsection} are devoted to the analysis of the limit

$$\lim_{(\epsilon,\delta)\downarrow 0} \Big\{ I_{sym}(\epsilon,\delta) + I_{ant}(\epsilon,\delta)\Big\}.$$
Lemma \ref{antsymLemma} reveals that $I_{ant}(\epsilon,\delta)$ will play no role in the isometry expression. The proof of Theorem \ref{mainTH2} is given in Section \ref{Isosection} and it is based on Theorem \ref{mainTH1} and Lemmas \ref{West} and \ref{almRK}.

\section{Preliminaries}\label{prelimi}
In this section, we introduce the basic objects for the proof of Theorem \ref{mainTH1}. 
\subsection{The projection operator}
In this section, we will describe the conditional expectation

\begin{eqnarray}\label{LambdaOP}
\Lambda^0(\epsilon,\delta; s,t)&=& \frac{1}{4 \epsilon\delta}\mathbb{E}\Big[ B_{s-\epsilon,s+\epsilon}\otimes B_{t-\delta,t+\delta}| B_s,B_t\Big]\\
\nonumber&=&\frac{1}{4 \epsilon\delta}\mathbb{E}\Big[ B^{(i)}_{s-\epsilon,s+\epsilon}B^{(j)}_{t-\delta,t+\delta}| B_s,B_t\Big];~1\le i,j\le d,
\end{eqnarray}
for $(s,t) \in [0,T]^2$. 
The projection operator $\Lambda^0(\epsilon,\delta; \cdot)$ onto the sigma-algebra generated by $(B_s,B_t)$ can be completely characterized by Gaussian linear regression which we now describe in detail. Let us denote

$$\mathbb{B}^{(i)}_{s,t}: = \left(
                                                                                        \begin{array}{c}
                                                                                          B^{(i)}_s \\
                                                                                          B^{(i)}_{t} \\
                                                                                        \end{array}
                                                                                      \right)
\quad \mathbf{B}_{s,t} = \left(
                                 \begin{array}{c}
                                   \mathbb{B}^{(1)}_{s,t} \\
                                   \vdots \\
                                   \mathbb{B}^{(d)}_{s,t} \\
                                 \end{array}
                               \right),
$$
for $(s,t) \in [0,T]^2$ and $1\le i\le d$. Observe

$$
\Lambda^0(\epsilon,\delta; s,t)= \frac{1}{4 \epsilon\delta}\mathbb{E}\Big[ B_{s-\epsilon,s+\epsilon}\otimes B_{t-\delta,t+\delta}| \mathbf{B}_{s,t}\Big]
$$
for $(s,t) \in [0,T]^2$.


For each $1\le i,j\le d$, we observe $\big( \mathbf{B}^\top_{s,t}, B^{(i)}_{s-\epsilon,s+\epsilon}, B^{(j)}_{t-\delta,t+\delta}\big)$ is a $2d+2$-dimensional Gaussian vector. Therefore, the classical linear regression analysis yields the following representation

\begin{equation}\label{pr}
\left(
  \begin{array}{c}
    B^{(i)}_{s-\epsilon,s+\epsilon} \\
    B^{(j)}_{t-\delta,t+\delta} \\
  \end{array}
\right)
 = \mathbb{E}\Bigg [ \left(
  \begin{array}{c}
    B^{(i)}_{s-\epsilon,s+\epsilon} \\
    B^{(j)}_{t-\delta,t+\delta} \\
  \end{array}
\right) \Big| \mathbf{B}_{s,t} \Bigg] + N_{s,t}(\epsilon,\delta),
\end{equation}
where $N_{s,t}(\epsilon,\delta)$ is a (zero-mean) 2-dimensional Gaussian vector independent from the sigma-algebra generated by $\mathbf{B}_{s,t}$. Moreover, we have the following representation

\begin{equation}\label{2dcexp1}
\mathbb{E}\Bigg [ \left(
  \begin{array}{c}
    B^{(i)}_{s-\epsilon,s+\epsilon} \\
    B^{(j)}_{t-\delta,t+\delta} \\
  \end{array}
\right) \Big| \mathbf{B}_{s,t} \Bigg] = \mathcal{N}^{\epsilon,\delta}_{s,t}(i,j) \Sigma^{-1}_{s,t} \mathbf{B}_{s,t},
\end{equation}
where

$$\mathcal{N}^{\epsilon,\delta}_{s,t}(i,j):=\text{Cov}\Big((B^{(i)}_{s-\epsilon,s+\epsilon}, B^{(j)}_{t-\delta,t+\delta})^\top; \mathbf{B}_{s,t}\Big), $$

$$\Sigma_{s,t}:=\text{Cov}\Big(\mathbf{B}_{s,t}; \mathbf{B}_{s,t}\Big).
$$
We can represent $\mathcal{N}^{\epsilon,\delta}_{s,t}(i,j)$ and $\Sigma_{s,t}^{-1}$ as follows: $\Sigma_{s,t}$ is a $2d\times 2d$-square matrix partitioned into a  block diagonal form

$$\Sigma_{s,t} = \left(
                   \begin{array}{cccc}
                     \Sigma_{s,t}(1,1) & 0 & \ldots & 0 \\
                     0 & \Sigma_{s,t}(2,2) & \ldots & 0 \\
                     \vdots & \vdots & \ddots & \vdots \\
                     0 & 0 & \ldots & \Sigma_{s,t}(d,d) \\
                   \end{array}
                 \right),
$$
where

$$\Sigma_{s,t}(i,i):= \text{Cov} \big( \mathbb{B}^{(i)}_{s,t}; \mathbb{B}^{(i)}_{s,t} \big),$$
for $(s,t) \in [0,T]^2$ and $1\le i\le d$. By the very definition, for $(s,t) \in [0,T]^2$, we have
$$\Sigma_{s,t}(i,i) = \left(
                   \begin{array}{cc}
                     \gamma(s) & R(s,t) \\
                     R(s,t) & \gamma(t) \\
                   \end{array}
                 \right),
$$

$$\text{det} (\Sigma_{s,t}(i,i)) = \gamma(s) \gamma(t) - R^2(s,t), $$


$$\Sigma^{-1}_{s,t}(i,i) = \frac{1}{\Theta_{s,t}} \left(
                   \begin{array}{cc}
                     \gamma(t) & -R(s,t) \\
                     -R(s,t) & \gamma(s) \\
                   \end{array}
                 \right).
$$

Here, in order to keep notation simple, we denote 

$$\Theta_{s,t}:= \text{det}(\Sigma_{s,t}(i,i)),$$
for $(s,t) \in [0,T]^2$. By construction,

$$\Sigma^{-1}_{s,t} = \left(
                   \begin{array}{cccc}
                     \Sigma^{-1}_{s,t}(1,1) & 0 & \ldots & 0 \\
                     0 & \Sigma^{-1}_{s,t}(2,2) & \ldots & 0 \\
                     \vdots & \vdots & \ddots & \vdots \\
                     0 & 0 & \ldots & \Sigma^{-1}_{s,t}(d,d) \\
                   \end{array}
                 \right).
$$
By the independence between $B^{(i)}$ and $B^{(j)}$ for $i\neq j$, we observe we can represent $\mathcal{N}^{\epsilon,\delta}_{s,t}(i,j)$ as follows


$$
\mathcal{N}^{\epsilon,\delta}_{s,t}(i,j)=\left(
  \begin{array}{ccccc}
    \alpha^{i}_1(\epsilon,s,t) & \ldots  & \alpha^{i}_i(\epsilon,s,t) & \ldots  & \alpha^{i}_d(\epsilon,s,t)  \\
    \beta^{j}_1(\delta,s,t) & \ldots & \ldots \beta^{j}_j(\delta,s,t)  & \ldots & \beta^{j}_d(\delta,s,t) \\
  \end{array}
\right).
$$
Here, the first row is represented by $\alpha^{i}_\ell(\epsilon,s,t):=(0,0)$ for every $\ell\neq i$ and
$$\alpha^{i}_i(\epsilon,s,t):= \Big(\text{cov}\big(B^{(1)}_{(s-\epsilon)\vee 0, s+\epsilon}; B^{(1)}_s\big), \text{cov}\big(B^{(1)}_{(s-\epsilon)\vee 0, s+\epsilon}; B^{(1)}_{t}\big)\Big).$$
The second row is represented by $ \beta^{j}_m(\delta,s,t) := (0,0)$ for every $m\neq j$ and

$$\beta^{j}_j(\delta,s,t):= \Big(\text{cov}\big(B^{(1)}_{(t-\delta)\vee 0, t+\delta}; B^{(1)}_s\big), \text{cov}\big(B^{(1)}_{(t-\delta)\vee 0, t+\delta}; B^{(1)}_{t}\big)\Big).$$
From $(\alpha^{i}_i(s,t),\beta^j_j(s,t)$, we can construct the following submatrix of $\mathcal{N}^{\epsilon,\delta}_{s,t}(i,j)$, namely:

$$
\mathcal{N}^{\epsilon,\delta}_{s,t}:=\left(
    \begin{array}{cc}
      \mathbf{n}_{11}(\epsilon,s) & \mathbf{n}_{12}(\epsilon,s,t) \\
      \mathbf{n}_{21}(\delta,s,t) &\mathbf{n}_{22}(\delta,s,t) \\
    \end{array}
  \right),
$$
where

$$\mathbf{n}_{11}(\epsilon,s):= R(s,s+\epsilon) - R(s,(s-\epsilon)\vee 0),
\quad \mathbf{n}_{12}(\epsilon,s,t):=R(s+\epsilon,t) - R((s-\epsilon)\vee 0,t),
$$
and 
$$\mathbf{n}_{21}(\delta,s,t): = R(s,t+\delta) - R(s,(t-\delta) \vee 0), \quad \mathbf{n}_{22}(\delta,s,t):=R(t+\delta,t) - R((t-\delta) \vee 0,t),
$$
for $(\epsilon,\delta) \in (0,1)^2$ and $(s,t) \in [0,T]^2$.

Let us denote

\begin{equation}\label{2dcexp2}
\left(
    \begin{array}{c}
      Z^{1;ij}_{s,t}(\epsilon) \\
      Z^{2;ij}_{s,t}(\delta) \\
    \end{array}
  \right)
:= \mathbb{E}\Bigg [ \left(
  \begin{array}{c}
    B^{(i)}_{s-\epsilon,s+\epsilon} \\
    B^{(j)}_{t-\delta,t+\delta} \\
  \end{array}
\right) \Big| \mathbf{B}_{s,t} \Bigg],
\end{equation}
for $(s,t) \in [0,T]^2$. From (\ref{2dcexp1}), the coordinates of the conditional expectation (\ref{2dcexp2}) are given by

\begin{equation}\label{Z1epsilon}
Z^{1;ij}_{s,t}(\epsilon) = \frac{1}{\Theta_{s,t}}\big\{\lambda_{11}(\epsilon,s,t)B^{(i)}_s + \lambda_{12}(\epsilon,s,t)B^{(i)}_{t}\big\}
 \end{equation}
and

\begin{equation}\label{Z2delta}
Z^{2;ij}_{s,t}(\delta)= \frac{1}{\Theta_{s,t}} \big\{\lambda_{21}(\delta,s,t)B^{(j)}_s
+ \lambda_{22}(\delta,s,t)B^{(j)}_{t}\big\},
\end{equation}
where we set

$$\lambda_{11}(\epsilon,s,t):=\mathbf{n}_{11}(\epsilon,s)\gamma(t) - \mathbf{n}_{12}(\epsilon,s,t)R(s,t),$$
$$\lambda_{12}(\epsilon,s,t):=\mathbf{n}_{12}(\epsilon,s,t) \gamma(s) - \mathbf{n}_{11}(\epsilon,s)R(s,t),$$
$$\lambda_{21}(\delta,s,t):=\mathbf{n}_{21}(\delta,s,t) \gamma(t) - \mathbf{n}_{22}(\delta,s,t)R(s,t),$$
$$\lambda_{22}(\delta,s,t):=\mathbf{n}_{22}(\delta,s,t)\gamma(s)  - \mathbf{n}_{21}(\delta,s,t)R(s,t),$$
for $(s,t) \in [0,T]^2_\star$. We shall represent $\Lambda^0(\epsilon,\delta; \cdot)$ as follows.

\begin{lemma}\label{Sprerepr}
For any $(\epsilon,\delta) \in (0,1)^2$,
\begin{eqnarray*}
\Lambda^{0,ij}(\epsilon,\delta; s,t)&=& \frac{1}{4\epsilon \delta}\Big[ Z^{1;ij}_{s,t}(\epsilon) Z^{2;ij}_{s,t}(\delta) - \mathbb{E}[
Z^{1;ij}_{s,t}(\epsilon) Z^{2;ij}_{s,t}(\delta)] \Big]\\
&+& \frac{1}{4\epsilon \delta}\mathbb{E}\big[B^{(i)}_{s-\epsilon,s+\epsilon} B^{(j)}_{t-\delta,t+\delta}\big],
\end{eqnarray*}
for $(s,t) \in [0,T]^2$ and $1\le i,j\le d$.
\end{lemma}
\begin{proof}
Fix $1\le i,j\le d$ and $(\epsilon,\delta) \in (0,1)^2$. Let $N_{s,t}(\epsilon,\delta) = \left(
                                  \begin{array}{c}
                                    N^{(1)}_{s,t}(\epsilon) \\
                                    N^{(2)}_{s,t}(\delta) \\
                                  \end{array}
                                \right)
$ be the Gaussian vector in (\ref{pr}). From (\ref{pr}), we have

$$B^{(i)}_{s-\epsilon,s+\epsilon} = Z^{1;ij}_{s,t}(\epsilon) + N^{(1)}_{s,t}(\epsilon),\quad B^{(j)}_{t-\delta,t+\delta} = Z^{2;ij}_{s,t}(\delta) + N^{(2)}_{s,t}(\delta),$$
for $(s,t) \in [0,T]^2$. By recalling that $N_{s,t}(\epsilon,\delta)$ is independent from the sigma-algebra generated by $\mathbf{B}_{s,t}$ and it has zero mean, we have

\begin{eqnarray}
\frac{1}{4\epsilon\delta}\mathbb{E}\Big[B^{(i)}_{s-\epsilon,s+\epsilon} B^{(j)}_{t-\delta,t+\delta}| \mathbf{B}_{s,t}\Big]&=& \label{Nexp}\frac{1}{4\epsilon\delta}Z^{1;ij}_{s,t}(\epsilon) Z^{2;ij}_{s,t}(\delta)\\
\nonumber&+& \frac{1}{4\epsilon\delta}\mathbb{E}\big[N^{(1)}_{s,t}(\epsilon) N^{(2)}_{s,t}(\delta)\big],
\end{eqnarray}
for $(s,t) \in [0,T]^2$. By taking expectation in the right-hand side of (\ref{Nexp}), we must have

$$\frac{1}{4\epsilon\delta}\mathbb{E}\big[N^{(1)}_{s,t}(\epsilon) N^{(2)}_{s,t}(\delta)\big] = \frac{1}{4\epsilon\delta}\mathbb{E}\big[B^{(i)}_{s-\epsilon,s+\epsilon} B^{(j)}_{t-\delta,t+\delta}\big] - \frac{1}{4\epsilon\delta}\mathbb{E}\big[ Z^{1;ij}_{s,t}(\epsilon) Z^{2;ij}_{s,t}(\delta) \big],$$
for $(s,t) \in [0,T]^2$. This concludes the proof.
\end{proof}

The following simple result shows that the anti-symmetric component $I_{ant}$ does not play any role for the construction of the symmetric-Stratonovich integral.
\begin{lemma}\label{antsymLemma}
The anti-symmetric component $I_{ant}(\epsilon,\delta)$ is null for every $(\epsilon,\delta) \in (0,1)^2$. Therefore,
$$
\Big \langle I^0(\epsilon,Y,dB), I^0(\delta,Y,dB) \Big\rangle_{L^2(\mathbb{P})} = I_{sym}(\epsilon,\delta),
$$
for every $(\epsilon,\delta) \in (0,1)^2$.
\end{lemma}
\begin{proof}
Fix $(\epsilon,\delta) \in (0,1)^2$. For each $i\neq j$, we can write

$$\int_{[0,T]^2}\mathbb{E}\Big[ Y^{ij}_{ant}(s,t) \frac{1}{4\epsilon\delta}B^{(i)}_{s-\epsilon,s+\epsilon}B^{(j)}_{t-\delta,t+\delta} \Big]dsdt$$
$$+\int_{[0,T]^2}\mathbb{E}\Big[ Y^{ji}_{ant}(s,t) \frac{1}{4\epsilon\delta}B^{(j)}_{s-\epsilon,s+\epsilon}B^{(i)}_{t-\delta,t+\delta} \Big]dsdt$$
$$ = \frac{1}{2}\int_{[0,T]^2}\mathbb{E}\Big[ Y^{(i)}_sY^{(j)}_t \frac{1}{4\epsilon\delta}\{B^{(i)}_{s-\epsilon,s+\epsilon}B^{(j)}_{t-\delta,t+\delta} - B^{(j)}_{s-\epsilon,s+\epsilon}B^{(i)}_{t-\delta,t+\delta}\}  \Big]dsdt$$
$$+ \frac{1}{2}\int_{[0,T]^2}\mathbb{E}\Big[ Y^{(j)}_s Y^{(i)}_t \frac{1}{4\epsilon\delta}\{B^{(j)}_{s-\epsilon,s+\epsilon}B^{(i)}_{t-\delta,t+\delta} - B^{(i)}_{s-\epsilon,s+\epsilon}B^{(j)}_{t-\delta,t+\delta}\}  \Big]dsdt.$$

Therefore, we can write

$$I_{ant}(\epsilon,\delta) = \frac{1}{2} \sum_{i,j=1}^d \int_{[0,T]^2}\mathbb{E}\Big[ Y^{(i)}_sY^{(j)}_t \frac{1}{4\epsilon\delta}\{B^{(i)}_{s-\epsilon,s+\epsilon}B^{(j)}_{t-\delta,t+\delta} - B^{(j)}_{s-\epsilon,s+\epsilon}B^{(i)}_{t-\delta,t+\delta}\}  \Big]dsdt$$
\begin{equation}\label{antrep}
=\frac{1}{2}\sum_{i,j=1}^d \mathbb{E}\int_{[0,T]^2} Y^{(i)}_sY^{(j)}_t \Big[\Lambda^{0,ij}(\epsilon,\delta;s,t) - \Lambda^{0,ji}(\epsilon,\delta;s,t) \Big]dsdt.
\end{equation}
We observe the matrix-valued process $\Lambda ^{0,ij}(\epsilon,\delta); 1\le i,j\le d$ is not symmetric. Indeed, by applying Lemma \ref{Sprerepr}, we have

\begin{eqnarray}
\nonumber \Lambda^{0,ij}(\epsilon,\delta;s,t) - \Lambda^{0,ji}(\epsilon,\delta;s,t)&=\frac{1}{4\epsilon\delta\Theta^2_{s,t}}&\{\lambda_{11}(\epsilon,s,t)\lambda_{22}(\delta,s,t) - \label{antArg}\lambda_{12}(\epsilon,s,t)\lambda_{21}(\delta,s,t)\}\\
&\times& \{B^{(i)}_s B^{(j)}_{t} - B^{(j)}_s B^{(i)}_{t} \},
\end{eqnarray}
for $(s,t) \in [0,T]^2_\star$ and $i\neq j$. In view of (\ref{antrep}) and (\ref{antArg}), it sufficient to check

$$\mathbb{E}\big[ Y^{(i)}_s Y^{(j)}_t B^{(i)}_s B^{(j)}_t\big] = \mathbb{E}\big[ Y^{(i)}_s Y^{(j)}_t B^{(j)}_s B^{(i)}_t \big],$$
for every $(s,t) \in [0,T]^2_\star$ and $i\neq j$. Let us choose a version $A^{ij}(s,t; \cdot)$ of the conditional expectation $\mathbb{E}[Y^{(i)}_s Y^{(j)}_t| M,N]$, i.e.,

$$A^{ij}(s,t; M,N) = \mathbb{E}[Y^{(i)}_s Y^{(j)}_t| M, N],$$
for $(s,t) \in [0,T]^2_\star$ and any pair of random variables $(M,N)$. The independence between $B^{(i)}_s$ and $B^{(j)}_t$ yields
\begin{eqnarray*}
\mathbb{E}\big[ Y^{(i)}_s Y^{(j)}_t B^{(i)}_s B^{(j)}_t\big] &=& \mathbb{E}\big[ B^{(i)}_s B^{(j)}_t A^{ij}(s,t; B^{(i)}_s B^{(j)}_t) \big]\\
&=&\int_{\mathbb{R}^2} xy A^{ij}(s,t; x,y)\mathbb{P}_{(B^{(i)}_s,B^{(j)}_t)}(dxdy)\\
&=& \int_{\mathbb{R}^2} xy A^{ij}(s,t; x,y)\mathbb{P}_{(B^{(j)}_s,B^{(i)}_t)}(dxdy)\\
&=& \mathbb{E}\big[ B^{(j)}_s B^{(i)}_t A^{ij}(s,t; B^{(j)}_s B^{(i)}_t) \big]\\
&=& \mathbb{E}\big[ Y^{(i)}_s Y^{(j)}_t B^{(j)}_s B^{(i)}_t \big].
\end{eqnarray*}
This concludes the proof.

\end{proof}

Lemma \ref{Sprerepr} reveals that the planar increment of $R$ will play a role in the asymptotic behavior of $I_{sym}(\epsilon,\delta)$ as $(\epsilon,\delta)\downarrow 0$. More precisely, for $i=j$, we observe $\frac{1}{4\epsilon\delta} \mathbb{E}[B^{(i)}_{s-\epsilon,s+\epsilon}B^{(j)}_{t-\delta,t+\delta}]$ is written as

$$\frac{1}{2\epsilon2\delta} \mathbb{E}[B^{(i)}_{s-\epsilon,s+\epsilon}B^{(i)}_{t-\delta,t+\delta}] =\frac{1}{2\epsilon2\delta} \Delta_{I_{\epsilon,\delta}}R(s,t),$$
where $I_{\epsilon,\delta}:= (s-\epsilon,s+\epsilon]\times (t-\delta, t+ \delta]$ and

$$\Delta_{I_{\epsilon,\delta}}R(s,t):= R (s+\epsilon, t+\delta) + R(s-\epsilon,t-\delta) - R(s-\epsilon,t+\delta) - R(s+\epsilon,t-\delta).
$$
One can also view $\frac{1}{2\epsilon2\delta}\Delta_{I_{\epsilon,\delta}}R$ as the standard finite-difference approximation scheme for the singular derivative

\begin{equation}\label{2derR}
\frac{\partial^2R}{\partial t\partial s} = H(2H-1)|t-s|^{2H-2}; s\neq t.
\end{equation}

\begin{remark}\label{SingpointwiseR}
Fix $0 < H < \frac{1}{2}$ and $1\le i\le d$. For each $s\neq t$, we can take $(\epsilon,\delta) \in (0,1)^2$ small enough (and obviously depending on $(s,t)$) such that

 $$\frac{1}{2\epsilon2\delta} \mathbb{E}[B^{(i)}_{s-\epsilon,s+\epsilon}B^{(i)}_{t-\delta,t+\delta}] = \frac{1}{2\epsilon 2\delta} \int_{t-\delta}^{t+\delta} \int_{s-\epsilon}^{s+\epsilon} \frac{\partial^2R}{\partial t\partial s}(a,b)dadb.
 $$
Therefore, we have pointwise convergence
$$\frac{1}{2\epsilon2\delta} \mathbb{E}[B^{(i)}_{s-\epsilon,s+\epsilon}B^{(i)}_{t-\delta,t+\delta}] \rightarrow \frac{\partial^2R}{\partial t\partial s}(s,t),$$
as $(\epsilon,\delta)\downarrow 0$ for each $s\neq t$ and $1\le i\le d$.
\end{remark}
In view of the representation (\ref{split}), the following result is a first step of getting a hint on the limiting behavior of the increment component in $I_{sym}(\epsilon,\delta)$ as $(\epsilon,\delta)\downarrow 0$.

\begin{lemma}\label{Lambdalimit}
For each $(s,t) \in [0,T]^2_\star$, we have

\begin{equation}\label{asZlimit}
\Lambda(s,t):=\lim_{(\epsilon,\delta)\downarrow 0}\Lambda^0(\epsilon,\delta;s,t) = \mathcal{W}(s,t) + \frac{\partial^2R}{\partial t\partial s}(s,t)I_{d\times d},
\end{equation}
almost surely, where


\begin{eqnarray}
\label{Wreplambda}\mathcal{W}^{ij}(s,t)&:=& \lambda_{11}(s,t)\lambda_{21}(s,t)\big(B^{(i)}_sB^{(j)}_s-\text{cov}(B^{(i)}_s; B^{(j)}_s)\big) \\
\nonumber &+& \lambda_{11}(s,t)\lambda_{22}(s,t) \big( B^{(i)}_s B^{(j)}_t - \text{cov}(B^{(i)}_s; B^{(j)}_t) \big)\\
\nonumber&+& \lambda_{12}(s,t) \lambda_{21}(s,t) \big( B^{(i)}_t B^{(j)}_s - \text{cov}(B^{(i)}_t; B^{(j)}_s) \big)\\
\nonumber&+& \lambda_{12}(s,t)\lambda_{22}(s,t)\big(B^{(i)}_t B^{(j)}_t - \text{cov}(B^{(i)}_t; B^{(j)}_t) \big),
\end{eqnarray}
and

$$\lambda_{11}(s,t) := \frac{1}{\Theta_{s,t}} \Big\{\frac{1}{2}\gamma'(s) \gamma(t) - R(s,t)\frac{\partial R}{\partial s} (s,t) \Big\},$$
$$\lambda_{12}(s,t):= \frac{1}{\Theta_{s,t}} \Big\{ \frac{\partial R}{\partial s} (s,t)\gamma(s) - R(s,t)\frac{1}{2}\gamma'(s) \Big\},$$
$$\lambda_{21}(s,t):= \frac{1}{\Theta_{s,t}} \Big\{\frac{\partial R}{\partial t} (s,t) \gamma(t) - R(s,t)\frac{1}{2}\gamma'(t) \Big\},$$
$$\lambda_{22}(s,t):=\frac{1}{\Theta_{s,t}} \Big\{\frac{1}{2}\gamma'(t) \gamma(s) - R(s,t)\frac{\partial R}{\partial t} (s,t) \Big\}$$
for $(s,t) \in [0,T]^2_\star$ and $1\le i,j\le d$.
\end{lemma}
\begin{proof}
Just apply Remark \ref{SingpointwiseR}, Lemma \ref{Sprerepr}, make use of the representations (\ref{Z1epsilon}), (\ref{Z2delta}) and the pointwise differentiability of the covariance kernel $R$. In this case,

$$\lambda_{11}(s,t) = \lim_{\epsilon\downarrow 0}\frac{1}{2\epsilon}\lambda_{11}(\epsilon,s,t),\quad \lambda_{12}(s,t)=\lim_{\epsilon\downarrow 0}\frac{1}{2\epsilon}\lambda_{12}(\epsilon,s,t),$$
$$\lambda_{21}(s,t) = \lim_{\delta\downarrow 0}\frac{1}{2\delta}\lambda_{21}(\delta,s,t), \quad \lambda_{22}(s,t) = \lim_{\delta\downarrow 0}\frac{1}{2\delta}\lambda_{22}(\delta,s,t),$$
for each $(s,t) \in [0,T]^2_\star$. 
\end{proof}

\begin{remark}\label{nelsondisc}
By the very definition, we observe that 

$$\frac{1}{4\epsilon\delta}Z^{1,ij}_{s,t}(\epsilon)Z^{2,ij}_{s,t} = \frac{1}{2\epsilon}\mathbb{E}\Big[B^{(i)}_{s-\epsilon,s+\epsilon}\big| \mathbf{B}_{s,t} \Big]\frac{1}{2\delta} \mathbb{E}\Big[B^{(i)}_{t-\delta,t+\delta}\big| \mathbf{B}_{s,t} \Big] .$$
Hence, $\mathcal{W}^{ij}(s,t)$ can be interpreted as a sort of product of the stochastic derivatives introduced by the classical work of Nelson \cite{nelson} (in the Markovian case) and studied in the works \cite{darses} and \cite{darses1} in the non-Markovian case. In the language of the works \cite{nelson} and \cite{darses}, we are taking derivatives w.r.t.~present information and not w.r.t.~whole history of the Gaussian process. We observe the pointwise differentiability of $(s,t)\mapsto R(s,t)$ plays a central role for the pointwise existence of $\mathcal{W}$. Moreover, we stress the forward version of the almost sure pointwise limit (in time)

$$\lim_{(\epsilon,\delta)\downarrow 0}\frac{1}{4\epsilon\delta}\mathbb{E}\Big[B^{(i)}_{s,s+\epsilon}\big| \mathbf{B}_{s,t} \Big]\mathbb{E}\Big[B^{(i)}_{t,t+\delta}\big| \mathbf{B}_{s,t} \Big]
$$
does not exists when $0 < H < \frac{1}{2}$. Indeed, 

$$\frac{R(s,s+\epsilon) - R(s,s)}{\epsilon} = \frac{1}{2} \Big\{ \frac{(s+\epsilon)^{2H} - s^{2H}}{\epsilon} + \epsilon^{2H-1}   \Big\}.$$
This can also be interpreted as a side effect of the well-known fact that the stochastic forward integral only exists in the regime $\frac{1}{2}\le H < 1$.
\end{remark}

\begin{remark}\label{renor}
The random field $\Lambda$ is not integrable over the rectangle and it can be interpreted as a random Schwartz distribution on $[0,T]^2$ but also a random Radon sigma-finite measure on $[0,T]^2_\star$ which is absolutely continuous w.r.t.~Lebesgue. We also stress that each term $\lambda_{11}$, $\lambda_{12}$, $\lambda_{21}$ and $\lambda_{22}$ admits a rather strong singularity of the form $|t-s|^{-1}$ along the diagonal of $[0,T]^2$. See Lemma \ref{fhlemma}. Therefore, we get non-integrable divergent terms

$$
|t-s|^{2H}\lambda_{11}(s,t)\lambda_{21}(s,t)\big(B^{(i)}_sB^{(j)}_s-\text{cov}(B^{(i)}_s; B^{(j)}_s)\big),
$$
$$|t-s|^{2H}\lambda_{11}(s,t)\lambda_{22}(s,t) \big( B^{(i)}_s B^{(j)}_t - \text{cov}(B^{(i)}_s; B^{(j)}_t) \big),$$
$$|t-s|^{2H}\lambda_{12}(s,t) \lambda_{21}(s,t) \big( B^{(i)}_t B^{(j)}_s - \text{cov}(B^{(i)}_t; B^{(j)}_s) \big),
$$
$$|t-s|^{2H}\lambda_{12}(s,t)\lambda_{22}(s,t)\big(B^{(i)}_t B^{(j)}_t - \text{cov}(B^{(i)}_t; B^{(j)}_t) \big),
$$
even when multiplying by compensating terms of the form $|t-s|^{2H}$.
\end{remark}
Remark \ref{renor} suggests that despite pointwise convergence of the projection operator $\lim_{(\epsilon,\delta)\downarrow 0}\Lambda(\epsilon,\delta)$ is in force, the singularity of $R$ along the diagonal would prevent any attempt to make the full use of $\Lambda$ when proving convergence of $\lim_{(\epsilon,\delta)\downarrow 0}I_{sym}(\epsilon,\delta)$. The following calculation suggests that this is \textit{not} the case. 

In the sequel, let us denote 

\begin{equation}\label{etadef1}
\eta_{11}(s,t):=\lambda_{11}(s,t) + \lambda_{12}(s,t), \eta_{12}(s,t):= \lambda_{12}(s,t)
\end{equation}
\begin{equation}\label{etadef2}
\eta_{21}(s,t):=\lambda_{21}(s,t) + \lambda_{22}(s,t), \eta_{22}(s,t):= \lambda_{22}(s,t),
\end{equation}
for $(s,t) \in [0,T]^2_\star$. By using the fact that $\lambda_{11}(s,t) = \lambda_{22}(t,s)$ and $\lambda_{12}(s,t) = \lambda_{21}(t,s)$ for every $(s,t) \in [0,T]^2_\star$, we can represent

\begin{eqnarray}
\label{INTREP1}\mathcal{W}^{ij}(s,t) &=& \Big(\eta_{11}(s,t)B^{(i)}_s + \eta_{12}(s,t)B^{(i)}_{s,t} \Big)\Big(\eta_{21}(s,t)B^{(j)}_s + \eta_{22}(s,t) B^{(j)}_{s,t} \Big)\\
\nonumber&-& \mathbb{E}\Big[\Big(\eta_{11}(s,t)B^{(i)}_s + \eta_{12}(s,t)B^{(i)}_{s,t} \Big)\Big(\eta_{21}(s,t)B^{(j)}_s + \eta_{22}(s,t) B^{(j)}_{s,t} \Big)\Big],
\end{eqnarray}
for $0\le s < t\le T$ and

\begin{eqnarray}
\label{INTREP2}\mathcal{W}^{ij}(s,t) &=& \Big(\eta_{11}(t,s)B^{(j)}_t + \eta_{12}(t,s)B^{(j)}_{t,s} \Big)\Big(\eta_{21}(t,s)B^{(i)}_t + \eta_{22}(t,s) B^{(i)}_{t,s} \Big)\\
\nonumber&-& \mathbb{E}\Big[\Big(\eta_{11}(t,s)B^{(j)}_t + \eta_{12}(t,s)B^{(j)}_{t,s} \Big)\Big(\eta_{21}(t,s)B^{(i)}_t + \eta_{22}(t,s) B^{(i)}_{t,s} \Big)\Big],
\end{eqnarray}
for $0\le t < s\le T$. The appearance of the increments $B_{s,t}$ and $B_{t,s}$ and the more regular functions $\eta_{11}$ and $\eta_{21}$ in representations (\ref{INTREP1}) and (\ref{INTREP2}) give us hope that one can indeed make use of $\Lambda$ for the study of the asymptotics of  $I_{sym}(\epsilon,\delta)$ as $(\epsilon,\delta)\downarrow 0$. 
\subsection{The projection operator over the 2-simplex}
The above discussion suggests that it is important to view  $\Lambda^0(\epsilon,\delta; \cdot)$ in terms of increments $B_{s,t}$ rather than $B_s, B_t$ as follows. In case $0\le s <t$, we will write 

\begin{equation}\label{side1}
\Lambda^0(\epsilon,\delta; s,t)= \frac{1}{4 \epsilon\delta}\mathbb{E}\Big[ B_{s-\epsilon,s+\epsilon}\otimes B_{t-\delta,t+\delta}| B_s,B_{s,t}\Big].
\end{equation}
In case $0\le t < s$, we will write 

\begin{equation}\label{side2}
\Lambda^0(\epsilon,\delta; s,t)= \frac{1}{4 \epsilon\delta}\mathbb{E}\Big[ B_{s-\epsilon,s+\epsilon}\otimes B_{t-\delta,t+\delta}| B_t,B_{t,s}\Big]. 
\end{equation}
For this reason, we collect some important objects of the projection operator over $\Delta_T=\{(s,t); 0\le s < t\le T\}$ which will play a key role in our analysis. In this section, we restrict the discussion to the 2-simplex $\Delta_T$. The representation of (\ref{side2}) over the other $2$-simplex of $[0,T]^2 $ is totally analogous. Let us consider 

$$
\mathcal{O}^{\epsilon,\delta}_{s,t}:=\left(
    \begin{array}{cc}
      \mathbf{o}_{11}(\epsilon,s) & \mathbf{o}_{12}(\epsilon,s,t) \\
      \mathbf{o}_{21}(\delta,s,t) &\mathbf{o}_{22}(\delta,s,t) \\
    \end{array}
  \right),
$$
where the elements $\mathbf{o}_{11}, \mathbf{o}_{12}, \mathbf{o}_{21}, \mathbf{o}_{22}$ are given by 

$$\mathbf{o}_{11}(\epsilon,s):=\text{cov}\big(B^{(1)}_{(s-\epsilon)\vee 0, s+\epsilon}, B^{(1)}_s \big)= R(s,s+\epsilon) - R(s,(s-\epsilon)\vee 0),$$

\begin{eqnarray*}
\mathbf{o}_{12}(\epsilon,s,t)&:=&\text{cov}\big(B^{(1)}_{(s-\epsilon)\vee 0, s+\epsilon}; B^{(1)}_{s,t}\big)=R(s+\epsilon,t) - R((s-\epsilon)\vee 0,t)\\
& - & [R(s+\epsilon,s) - R((s-\epsilon)\vee 0,s)],
\end{eqnarray*}

$$\mathbf{o}_{21}(\delta,s,t):\text{cov}\big(B^{(1)}_{(t-\delta)\vee 0, t+\delta}; B^{(1)}_s\big) = R(s,t+\delta) - R(s,(t-\delta) \vee 0) $$
and

\begin{eqnarray*}
\mathbf{o}_{22}(\delta,s,t)&:=&\text{cov}\big(B^{(1)}_{(t-\delta)\vee 0, t+\delta}; B^{(1)}_{s,t}\big)=R(t+\delta,t) - R((t-\delta) \vee 0,t)\\
 &-& [R(t+\delta,s) - R((t-\delta)\vee 0,s)],
 \end{eqnarray*}
for $(\epsilon,\delta) \in (0,1)^2$ and $(s,t) \in \Delta_T$. Gaussian linear regression yields the following representation

\begin{equation}\label{Z1epsINCR}
Z^{1;ij}_{s,t}(\epsilon) = \frac{1}{\Theta_{s,t}}\Big\{\eta_{11}(\epsilon,s,t)B^{(i)}_s + \eta_{12}(\epsilon,s,t)B^{(i)}_{s,t}\Big\}
 \end{equation}
and

\begin{equation}\label{Z2deINCR}
Z^{2;ij}_{s,t}(\delta)= \frac{1}{\Theta_{s,t}} \Big\{\eta_{21}(\delta,s,t)B^{(j)}_s
+ \eta_{22}(\delta,s,t)B^{(j)}_{s,t}\Big\},
\end{equation}
where we set

$$\eta_{11}(\epsilon,s,t):=\mathbf{o}_{11}(\epsilon,s)|t-s|^{2H} - \mathbf{o}_{12}(\epsilon,s,t)\varphi(s,t),$$
$$\eta_{12}(\epsilon,s,t):=\mathbf{o}_{12}(\epsilon,s,t)s^{2H} - \mathbf{o}_{11}(\epsilon,s)\varphi(s,t),$$
$$\eta_{21}(\delta,s,t):=\mathbf{o}_{21}(\delta,s,t)|t-s|^{2H} - \mathbf{o}_{22}(\delta,s,t)\varphi(s,t),$$
$$\eta_{22}(\delta,s,t):=\mathbf{o}_{22}(\delta,s,t) s^{2H} - \mathbf{o}_{21}(\delta,s,t)\varphi(s,t),$$
for $(s,t) \in \Delta_T$. One can easily check that $\eta_{11}(\epsilon,s,t) = \lambda_{11}(\epsilon,s,t) + \lambda_{12}(\epsilon,s,t)$, $\eta_{21}(\delta,s,t) = \lambda_{21}(\delta,s,t) + \lambda_{22}(\delta,s,t)$, $\eta_{12}(\epsilon,s,t) = \lambda_{12}(\epsilon,s,t)$ and $\eta_{22}(\delta,s,t) = \lambda_{22}(\delta,s,t)$, for $(s,t) \in \Delta_T$ and $(\epsilon,\delta) \in (0,1)^2$.

The matrix $\mathcal{O}^{\epsilon,\delta}_{s,t}$ will play a key rule in our analysis and we can actually write it in terms of a single function $\ell:\mathbb{R}\rightarrow \mathbb{R}$ defined by

\begin{equation}\label{ellfunc}
\ell(x):=|1+x|^{2H} - |1-x|^{2H}; x \in \mathbb{R}.
\end{equation}
For a positive $\delta \in (0,1)$, we also define

$$\Delta\ell(\delta,s,t):=\Big\{ \frac{t^{2H}}{2}\ell\Big( \frac{\delta}{t}\Big)  - \frac{|t-s|^{2H}}{2}\ell\Big(\frac{\delta}{t-s}\Big) \Big\}$$
for $(s,t) \in \Delta_T$.

One can easily check we can write

$$
\mathcal{O}^{\epsilon,\delta}_{s,t} = \left(
                                      \begin{array}{cc}
                                        \frac{s^{2H}}{2} \ell\big( \frac{\epsilon}{s}\big) & \frac{|t-s|^{2H}}{2} \ell\big( \frac{\epsilon}{t-s}\big) \\
                                         \Delta \ell(\delta,s,t)& \frac{|t-s|^{2H}}{2} \ell\big( \frac{\delta}{t-s}\big) \\
                                      \end{array}
                                    \right),
$$
for $(s,t)\in \Delta_T$ such that $0 < s-\epsilon$ and $0 < t-\delta$. Other representations are given by

$$\mathcal{O}^{\epsilon,\delta}_{s,t} = \left(
                                      \begin{array}{cc}
                                        R(s,s+\epsilon) & R(t,s+\epsilon) - R(s,s+\epsilon) \\
                                         \Delta \ell(\delta,s,t)& \frac{|t-s|^{2H}}{2} \ell\big( \frac{\delta}{t-s}\big) \\
                                      \end{array}
                                    \right),
$$
for $(s,t)\in \Delta_T$ such that $0 \ge s-\epsilon$, $0 < t-\delta$ and

$$\mathcal{O}^{\epsilon,\delta}_{s,t} = \left(
                                      \begin{array}{cc}
                                        R(s,s+\epsilon) & R(t,s+\epsilon) - R(s,s+\epsilon) \\
                                         R(s,t+\delta)& R(t+\delta,t) - R(t+\delta,s) \\
                                      \end{array}
                                    \right),
$$
for $(s,t)\in \Delta_T$ such that $0 \ge s-\epsilon$ and $0 \ge t-\delta$.

The following inequality is crucial to estimate $\Lambda^0(\epsilon,\delta; s,t)$ along the diagonal of $[0,T]^2$ and the origin. 
\begin{lemma}\label{fhlemma}
Fix $H  \in (0,1)$. The following representation holds

\begin{equation}\label{detrep}
\Theta_{s,t} = |t-s|^{2H}A(s,t),
\end{equation}
for $(s,t) \in [0,T]^2_\star$, where
$$
A(s,t):=\frac{1}{4}\Big\{ 2s^{2H}+ 2t^{2H} - |t-s|^{2H} - \frac{(t^{2H}-s^{2H})^2}{|t-s|^{2H}} \Big\},
$$
for $(s,t) \in [0,T]^2_\star$. Moreover,

$$\frac{s^{2H} \wedge t^{2H}}{|A(s,t)|} \le \frac{2^{2-2H}}{4-2^{2H}};~(s,t) \in [0,T]^2_\star.$$
Therefore,

\begin{equation}\label{deterest}
\frac{1}{\Theta^2_{s,t}}\lesssim \frac{1}{|t-s|^{4H}(s^{4H}\wedge t^{4H})},
\end{equation}
for every $(s,t) \in [0,T]^2_\star$ such that $s\wedge t>0$.
\end{lemma}

\begin{proof}
Representation (\ref{detrep}) is obvious. Fix $0 < t\le T$. Let us write
$$
A(s,t)=\frac{1}{4}\Big\{ 2s^{2H} + \phi(s,t) \Big\},
$$
for $ s \neq t$, where

$$\phi(s,t):= t^{2H} - |t-s|^{2H} + t^{2H} - \frac{(t^{2H}-s^{2H})^2}{|t-s|^{2H}}; s \neq t.$$
If $\frac{1}{2}\le H < 1$, the function $a\mapsto a^{2H}$ is Lipschitz. If $0 < H < \frac{1}{2}$, the function $a\mapsto a^{2H}$ is $2H$-H\"older, then

$$\frac{|t^{2H} - s^{2H}|^2}{|t-s|^{2H}}\rightarrow 0$$
as $s\rightarrow t$.  Therefore, $\lim_{s\rightarrow t}\phi(s,t)=2t^{2H}$ and hence

$$\lim_{s\rightarrow t}\frac{s^{2H}}{A(s,t)} = \lim_{s\rightarrow t}\frac{s^{2H}}{\frac{1}{4} \{2s^{2H} + \phi(s,t)\}} = 1.$$
We claim

\begin{equation}\label{con1}
\lim_{s\downarrow 0}\frac{s^{2H}}{A(s,t)}=1,
\end{equation}
for any $ H \in (0,1)$. Indeed, let us write

$$\frac{s^{2H}}{A(s,t)} = \frac{4 s^{2H}}{2s^{2H} + \phi(s,t)} = \frac{4}{2+\frac{\phi(s,t)}{s^{2H}}}.$$
Observe

\begin{eqnarray*}
\frac{\partial \phi}{\partial s}(s,t) &=& 2H \text{sgn}(t-s)|t-s|^{2H-1} + \frac{1}{|t-s|^{4H}}\Big\{ 4Hs^{2H-1}(t^{2H}-s^{2H})|t-s|^{2H}\\
&-& 2H \text{sgn}(t-s)|t^{2H} - s^{2H}|^2|t-s|^{2H-1}\Big\},
\end{eqnarray*}
for $s\neq t$. By L'Hospital, we know that

\begin{eqnarray*}
\lim_{s\downarrow 0}\frac{\phi(s,t)}{s^{2H}} &=& \lim_{s\downarrow 0}\Bigg[ \frac{2(t^{2H}-s^{2H})}{(t-s)^{2H}} + \frac{(t-s)^{4H}-(t^{2H}-s^{2H})^2}{(t-s)^{1+2H}s^{2H-1}}\Bigg]\\
&=& \lim_{s\downarrow 0}\Bigg[ \frac{2(t^{2H}-s^{2H})}{(t-s)^{2H}} + \frac{\{(t-s)^{4H}-t^{4H}\} + \{s^{2H} (2t^{2H}-s^{2H})\}}{(t-s)^{1+2H}s^{2H-1}}\Bigg].
\end{eqnarray*}
We have

$$\lim_{s\downarrow 0}\frac{2(t^{2H}-s^{2H})}{(t-s)^{2H}}=2, \quad \lim_{s\downarrow 0} \frac{s^{2H} (2t^{2H}-s^{2H})}{(t-s)^{1+2H}s^{2H-1}}=0.$$
Mean value theorem yields
$$|(t-s)^{4H}-t^{4H}|\le 4H\max\{t^{4H-1}, (t-s)^{4H-1}\}s$$
so that
$$\Bigg|\frac{(t-s)^{4H}-t^{4H}}{(t-s)^{1+2H}s^{2H-1}}\Bigg|\le \frac{4H\max\{t^{4H-1}, (t-s)^{4H-1}\}}{(t-s)^{1+2H}}s^{2-2H}\rightarrow 0,$$
as $s\downarrow 0$. This shows (\ref{con1}). Now, let us consider

$$
g(s,t)=\left\{
\begin{array}{rl}
\frac{s^{2H}}{A(s,t)}; & \hbox{if} \ 0 < s < t \\
1;& \hbox{if} \ s=0~\text{or}~s=t. \\
\end{array}
\right.
$$

This function is continuous (it is $C^1$ over $(0,t)$) on the interval $[0,t]$ for each $t \in (0,T]$. Moreover,

$$\frac{\partial g}{\partial s}(s,t) = \frac{2Hs^{2H-1}A(s,t) - s^{2H} \frac{\partial A}{\partial s}(s,t) }{|A(s,t)|^2};0< s < t.$$
Clearly, $g(\cdot, t)$ has a unique critical point at $\frac{t}{2}$ over $(0,t)$, i.e., $\partial_s g\Big(\frac{t}{2},t\Big)=0$. Moreover,

$$g\Big(\frac{t}{2},t\Big) = \frac{(\frac{t}{2})^{2H}}{A\big(\frac{t}{2},t\big)} = - \frac{2^{2-2H}}{2^{2H}-4}\ge 1,$$
for every $t \in (0,T]$ and $ H \in (0,1)$. This completes the proof.
\end{proof}

\section{Analysis of the marginal component}\label{secsection}
In this section, we will analyse the limit

$$\lim_{(\epsilon,\delta)\downarrow 0}\frac{1}{2} \mathbb{E}\int_{[0,T]^2}\Big\langle  Y_s \otimes Y_s + Y_t \otimes Y_t, \frac{1}{4\epsilon\delta}B_{s-\epsilon,s+\epsilon}\otimes B_{t-\delta,t+\delta}\Big \rangle dsdt,$$
for $Y_t = g(t,B_t).$


\begin{lemma}\label{marglemma}
If $Y_\cdot=g(\cdot, B_\cdot)$ satisfies 

$$\mathbb{E}\sup_{0\le t\le T}|Y_t|^p < \infty,$$
for some $1 < p < \infty$, then

$$\lim_{(\epsilon,\delta)\downarrow 0}\frac{1}{2} \mathbb{E}\int_{[0,T]^2}\Big\langle Y_s\otimes   Y_s + Y_t\otimes  Y_t, \frac{1}{4\epsilon\delta}B_{s-\epsilon,s+\epsilon}\otimes B_{t-\delta,t+\delta}\Big\rangle dsdt$$
$$= \mathbb{E}\int_0^T \big\langle Y_t \otimes Y_t, \mathcal{M}_t\big\rangle dt$$
$$+ \mathbb{E}\int_0^T |Y_t|^2 \frac{\partial R}{\partial t }(T,t)dt.$$

\end{lemma}
\begin{proof}
We claim that if $f:[0,T]\times\mathbb{R}^d\rightarrow \mathbb{R} $ satisfies $\mathbb{E}\sup_{0\le t\le T}|f(t,B_t)|^p < \infty,$ for some $1 < p < \infty$, then 

$$\lim_{(\epsilon,\delta)\downarrow 0}\mathbb{E}\int_{[0,T]^2} f(t,B_t)\frac{1}{4\epsilon\delta}B^{(i)}_{s-\epsilon,s+\epsilon}B^{(j)}_{t-\delta,t+\delta}dsdt
$$ 
$$= \lim_{(\epsilon,\delta)\downarrow 0}\mathbb{E}\int_{[0,T]^2} f(s,B_s)\frac{1}{4\epsilon\delta}B^{(i)}_{s-\epsilon,s+\epsilon}B^{(j)}_{t-\delta,t+\delta}dsdt$$
$$= \mathbb{E}\int_0^T |f(t,B_t)|^2 \frac{\partial R}{\partial t }(T,t)dt\delta_{ij}$$
\begin{equation}\label{mareql}
+\frac{1}{2}\mathbb{E}\int_0^T f(t,B_t)\Big\{ \frac{B^{(i)}_t B^{(j)}_t-cov(B^{(i)}_t; B^{(j)}_t)}{t^{2H}} \Big\} \frac{R(T,t)}{t^{2H}}d\gamma(t),
\end{equation}
for each $1\le i,j\le d$. In this case, (\ref{mareql}) will allow us to conclude the proof. 

First, we observe we can write

\begin{eqnarray*}
\mathbb{E}\Bigg[\int_{[0,T]^2} f(t,B_t)\frac{1}{4\epsilon\delta}B^{(i)}_{s-\epsilon,s+\epsilon}B^{(j)}_{t-\delta,t+\delta}dsdt\Bigg] &=& \frac{1}{2\delta}\mathbb{E}\Bigg[\int_0^T f(t,B_t) B^{(j)}_{t-\delta,t+\delta} dt\frac{1}{2\epsilon}\int_{T-\epsilon}^{T+\epsilon} B^{(i)}_sds\Bigg]
\\
&-& \frac{1}{2\delta}\mathbb{E}\Bigg[\int_0^T f(t,B_t) B^{(j)}_{t-\delta,t+\delta} dt\frac{1}{2\epsilon}\int_{0}^{\epsilon} B^{(i)}_sds\Bigg].
\end{eqnarray*}
The idea is to project onto the sigma algebra generated by $B_t$ as follows

$$\frac{1}{2\delta}\mathbb{E}\Bigg[\int_0^T f(t,B_t) B^{(j)}_{t-\delta,t+\delta} dt\frac{1}{2\epsilon}\int_{T-\epsilon}^{T+\epsilon} B^{(i)}_sds\Bigg]
$$
\begin{equation}\label{mg1}
= \frac{1}{2\delta}\mathbb{E}\Bigg[\int_0^T \int_{T-\epsilon}^{T+\epsilon} f(t,B_t) \mathbb{E}\Big[ B^{(j)}_{t-\delta,t+\delta} B^{(i)}_s \big| B_t\Big] ds\frac{1}{2\epsilon}dt\Bigg].
\end{equation}
Gaussian regression yields the representation

\begin{equation}\label{regsim}
\nonumber\mathbb{E}\Big[ B^{(j)}_{t-\delta,t+\delta} B^{(i)}_s \big| B_t\Big]= q^{ij}(\delta;s,t)r^{ij}(t)
+ \mathbb{E}\big[ B^{(i)}_s B^{(j)}_{t-\delta,t+\delta}\big],
\end{equation}
where
$$q^{ij}(\delta; s,t)=\text{cov}\Big( B^{(j)}_{t-\delta,t+\delta}; \frac{B^{(j)}_t}{t_H} \Big)\text{cov}\Big( B^{(i)}_s; \frac{B^{(i)}_t}{t_H} \Big),$$
$$r^{ij}(t)=\Big\{ \frac{B^{(i)}_t B^{(j)}_t-cov(B^{(i)}_t; B^{(j)}_t)}{t^{2H}} \Big\}.$$

By writing $\mathbb{E}\big[ B^{(i)}_s B^{(j)}_{t-\delta,t+\delta}\big] = \mathbb{E}\big[ B^{(i)}_T B^{(j)}_{t-\delta,t+\delta}\big] - \mathbb{E}\big[ (B^{(i)}_T - B^{(i)}_s) B^{(j)}_{t-\delta,t+\delta}\big] $, we can decompose

$$\frac{1}{4\delta\epsilon}\mathbb{E}\Bigg[\int_0^T \int_{T-\epsilon}^{T+\epsilon} f(t,B_t) \mathbb{E}\big[ B^{(i)}_s B^{(j)}_{t-\delta,t+\delta}\big]dsdt\Bigg]$$
$$=\frac{1}{2\delta}\mathbb{E}\Bigg[\int_0^T  f(t,B_t) \mathbb{E}\big[ B^{(i)}_T B^{(j)}_{t-\delta,t+\delta}\big]dt\Bigg]$$
$$-\frac{1}{4\delta\epsilon}\mathbb{E}\Bigg[\int_0^T \int_{T-\epsilon}^{T+\epsilon} f(t,B_t) \mathbb{E}\big[ (B^{(i)}_T-B^{(i)}_s) B^{(j)}_{t-\delta,t+\delta}\big]dsdt\Bigg].$$
Observe

$$
\frac{1}{2\delta}\mathbb{E}\big[ (B^{(i)}_T-B^{(i)}_s) B^{(j)}_{t-\delta,t+\delta}\big] = \Bigg[\frac{\big(R(t+\delta,T) - R(t+\delta,s)\big)  - \big(R(t-\delta,T) - R(t-\delta,s)\big)}{2\delta}\Bigg]\delta_{ij},
$$
for $(s,t) \in [0,T]^2$ and hence

$$\frac{1}{2\delta}\mathbb{E}\Bigg[\int_0^T  f(t,B_t) \mathbb{E}\big[ B^{(i)}_T B^{(j)}_{t-\delta,t+\delta}\big]dt\Bigg]\rightarrow \mathbb{E}\Bigg[\int_0^T  f(t,B_t) \frac{\partial R}{\partial t}(t,T)dt\Bigg]\delta_{ij}, $$
as $\delta \downarrow 0$ and

$$\frac{1}{2\delta}\mathbb{E}\Bigg[\int_0^T \frac{1}{2\epsilon}\int_{T-\epsilon}^{T+\epsilon} f(t,B_t) \mathbb{E}\big[ (B^{(i)}_T-B^{(i)}_s) B^{(j)}_{t-\delta,t+\delta}\big]dsdt\Bigg]\rightarrow 0 $$
as $(\epsilon,\delta) \downarrow 0$ for any $1\le i,j\le d$. Again, by writing $\text{cov}\big( B^{(i)}_s; \frac{B^{(i)}_t}{t_H} \big) = \text{cov}\big( B^{(i)}_s - B^{(i)}_T; \frac{B^{(i)}_t}{t_H} \big)+ \text{cov}\big( B^{(i)}_T; \frac{B^{(i)}_t}{t_H} \big)$, we can decompose

$$\frac{1}{4\delta\epsilon}\mathbb{E}\Bigg[\int_0^T \int_{T-\epsilon}^{T+\epsilon} f(t,B_t) q^{ij}(\delta;s,t)r^{ij}(t)dsdt\Bigg]$$
$$ = \frac{1}{4\delta\epsilon}\mathbb{E}\Bigg[\int_0^T \int_{T-\epsilon}^{T+\epsilon} f(t,B_t) \text{cov}\Big( B^{(j)}_{t-\delta,t+\delta}; \frac{B^{(j)}_t}{t_H} \Big)\text{cov}\Big( B^{(i)}_s - B^{(i)}_T; \frac{B^{(i)}_t}{t_H} \Big) r^{ij}(t)dsdt\Bigg]$$
$$+\frac{1}{2\delta}\mathbb{E}\Bigg[\int_0^T f(t,B_t) \text{cov}\Big( B^{(j)}_{t-\delta,t+\delta}; \frac{B^{(j)}_t}{t_H} \Big)\text{cov}\Big(B^{(i)}_T; \frac{B^{(i)}_t}{t_H} \Big) r^{ij}(t)dt\Bigg]=:I^{ij}_1(\epsilon,\delta) + I^{ij}_2(\epsilon,\delta). $$
We claim that

\begin{equation}\label{clmar1}
I^{ij}_2(\epsilon,\delta)\rightarrow \frac{1}{2}\mathbb{E}\int_0^T f(t,B_t)r^{ij}(t) \frac{R(T,t)}{t^{2H}}d\gamma(t),
\end{equation}
and
\begin{equation}\label{clmar2}
I^{ij}_1(\epsilon,\delta)\rightarrow 0,
\end{equation}
as $(\epsilon,\delta)\downarrow 0$. Indeed, let us denote

$$F_{\delta}(s) = \frac{1}{2\delta}\int_0^s \text{cov}\Big( B^{(j)}_{t-\delta,t+\delta}; B^{(j)}_t \Big)dt.$$
We observe

\begin{eqnarray*}
F_{\delta}(s) &=& \frac{1}{2\delta}\int_{s-\delta}^s R(t+\delta,t)dt + \int_0^s
\Big\{\frac{R(t+\delta,t) - R(t-\delta,t)}{2\delta}\Big\}dt\\
&-&\frac{1}{2\delta}\int_{s-\delta}^s R(t+\delta,t)dt\rightarrow \frac{1}{2}\gamma(s),
\end{eqnarray*}
for each $s \in [0,T]$, as $\delta\downarrow 0$. Moreover, by applying Lemma \ref{growthELL}, we get

\begin{eqnarray*}
\| F_\delta\|_{TV}&=& \int_0^T t^{2H}\ell \Big(\frac{\epsilon}{t} \Big)\frac{1}{2\epsilon}
dt\\
&\le&  2^{2H-1}\int_0^T t^{2H-1}dt,
\end{eqnarray*}
for every $\delta \in (0,1)$. This shows (\ref{clmar1}). Similar argument can be made for (\ref{clmar2}). Therefore, we conclude

$$\frac{1}{2\delta}\mathbb{E}\Bigg[\int_0^T f(t,B_t) B^{(j)}_{t-\delta,t+\delta} dt\frac{1}{2\epsilon}\int_{T-\epsilon}^{T+\epsilon} B^{(i)}_sds\Bigg]$$
converges to the right-hand side of (\ref{mareql}). By applying a similar analysis, one can also easily check that

$$\frac{1}{2\delta}\mathbb{E}\Bigg[\int_0^T f(t,B_t) B^{(j)}_{t-\delta,t+\delta} dt\frac{1}{2\epsilon}\int_{0}^{\epsilon} B^{(i)}_sds\Bigg]\rightarrow 0,$$
as $(\epsilon,\delta)\downarrow 0$. By the symmetry of our argument, we may conclude the proof.
\end{proof}

\section{Analysis of the increment component and proof of Theorem \ref{mainTH1}}\label{mainsection}
In this section, we will concentrate the analysis on the increment component of $I_{sym}(\epsilon,\delta)$ in (\ref{split}). That is, we will investigate the limit
$$\lim_{(\epsilon,\delta)\downarrow 0}-\frac{1}{2} \int_{[0,T]^2}\Big\langle  Y_{s,t} \otimes Y_{s,t}, \Lambda^0(\epsilon,\delta;s,t)\Big \rangle dsdt.
$$

For this purpose, we decompose the matrix-valued process $\Lambda^{0,ij}(\epsilon,\delta)$ given in (\ref{LambdaOP}) into two distinct components according to Lemma \ref{Sprerepr}. The random component 

$$\frac{1}{4\epsilon \delta}\Big[ Z^{1;ij}_{s,t}(\epsilon) Z^{2;ij}_{s,t}(\delta) - \mathbb{E}[
Z^{1;ij}_{s,t}(\epsilon) Z^{2;ij}_{s,t}(\delta)] \Big]
$$
and the deterministic component
$$\frac{1}{4\epsilon \delta}\mathbb{E}\big[B^{(i)}_{s-\epsilon,s+\epsilon} B^{(j)}_{t-\delta,t+\delta}\big].
$$

\subsection{Deterministic component}
In this section, we investigate the limit of

$$\lim_{(\epsilon,\delta)\downarrow 0} -\frac{1}{2}\sum_{i=1}^d \int_{[0,T]^2} \mathbb{E}|Y^{(i)}_{s,t}|^2\frac{1}{4\epsilon\delta}\mathbb{E}\big[B^{(i)}_{s-\epsilon,s+\epsilon} B^{(i)}_{t-\delta,t+\delta}\big]dsdt,$$
where we recall $\frac{1}{4\epsilon\delta}\mathbb{E}\big[B^{(i)}_{s-\epsilon,s+\epsilon} B^{(i)}_{t-\delta,t+\delta}\big]=\frac{1}{4\epsilon \delta} \Delta_{I_{\epsilon,\delta}}R(s,t)$ for $(s,t) \in [0,T]^2_\star$.
\begin{lemma}\label{inpd1}
Let $g \in C^{\bar{\gamma},\gamma}([0,T]\times \mathbb{R}^d)$ with $ \bar{\gamma}> \frac{1}{2}-H, \gamma > \frac{1}{2H}-1$ and $Y_\cdot=g(\cdot, B_\cdot)$. For every $1 < p < \min \big\{\frac{1}{2-2H-2H\gamma}; \frac{1}{2-2H-2\bar{\gamma}}\big\}$, there exists a constant $C=C(H,\bar{\gamma},\gamma,T,p)$ such that

\begin{equation}\label{awayest}
\sup_{(\epsilon,\delta) \in (0,1)^2}\mathbb{E} \int_0^T \int_0^{t-\max\{2\epsilon+\delta; 2\delta +\epsilon\}}\big| Y_{s,t}\big|^{2p} \Big|\frac{1}{4\epsilon \delta} \Delta_{I_{\epsilon,\delta}}R(s,t)\Big|^pdsdt\le C\|g\|^{2p}_{\bar{\gamma},\gamma}. 
\end{equation}
\end{lemma}
\begin{proof}
Fix $\gamma > \frac{1}{2H}-1, \bar{\gamma} > \frac{1}{2}-H$ and $1 < p < \min \{\frac{1}{2-2H-2\gamma H}; \frac{1}{2-2H-2\bar{\gamma}}\}$. Fix $G_{\epsilon,\delta} = \{(s,t) \in \Delta_T; 0\le s < t-\max\{2\epsilon+\delta; 2\delta +\epsilon\}, 0\le t\le T\}$. First, we observe $(s-\epsilon,s+\epsilon)\cap (t-\delta, t+\delta) = \emptyset$ whenever $t-s > \epsilon + \delta$. Hence, we can write

\begin{equation}\label{secrepR}
\frac{1}{4\epsilon \delta} \Delta_{I_{\epsilon,\delta}}R(s,t) = H(2H-1)\frac{1}{2\epsilon 2\delta}\int_{t-\delta}^{t+\delta} \int_{s-\epsilon}^{s+\epsilon} (b-a)^{2H-2}dadb,
\end{equation}
for $t-s > \epsilon + \delta$.
From (\ref{secrepR}), we get

\begin{equation}\label{awa1}
\Big|\frac{1}{4\epsilon \delta} \Delta_{I_{\epsilon,\delta}}R(s,t)\Big| \le H(1-2H)|t-(\delta+\epsilon) - s|^{2H-2},
\end{equation}
whenever $t-s > \epsilon +\delta$. By assumption, 
$$|Y_{s,t}|^{2p}\le \|g\|^{2p}_{\bar{\gamma},\gamma} \{|t-s|^{2p \bar{\gamma}} + |B_{s,t}|^{2p\gamma}\},$$
for $(s,t) \in \Delta_T$. By triangle inequality 

\begin{eqnarray*}
\{|t-s|^{2p \bar{\gamma}} + |B_{s,t}|^{2p\gamma}\}&\lesssim& \{\delta^{2p\bar{\gamma}} + \epsilon^{2p\bar{\gamma}} + |t-s - (\epsilon+\delta)|^{2p\bar{\gamma}}\\
& +& |B_{t-\delta,t}|^{2p\gamma} +  |B_{s,s+\epsilon}|^{2p\gamma} + |B_{s+\epsilon,t-\delta}|^{2p\gamma}  \}\\
&=:& I_{\epsilon,\delta}(s,t) + J_{\epsilon,\delta}(s,t),
\end{eqnarray*}
for every $(s,t) \in \Delta_T$. Then,

\begin{eqnarray*}
\mathbb{E}\int_{G_{\epsilon,\delta}}|Y_{s,t}|^{2p}\big|\frac{1}{4\epsilon \delta} \Delta_{I_{\epsilon,\delta}}R(s,t)\big|^p dsdt&\le& \|g\|^{2p}_{\bar{\gamma},\gamma} \mathbb{E}\int_{G_{\epsilon,\delta}} |I_{\epsilon,\delta}(s,t)+ J_{\epsilon,\delta}(s,t)|\\
&\times& |t-s-(\epsilon+\delta)|^{(2H-2)p}dsdt.
\end{eqnarray*}
Observe that $\delta \le t-s - (\epsilon+\delta) \Longleftrightarrow s < t-(\epsilon + 2\delta)$ and $\epsilon \le t-s - (\epsilon+\delta) \Longleftrightarrow s < t-(\delta + 2\epsilon)$. Therefore, 

$$|t-s - (\epsilon+\delta)|^{(2H-2)p}I_{\epsilon,\delta}(s,t)\lesssim |t-s - (\epsilon+\delta)|^{(2H-2+ 2\bar{\gamma}) p}$$
on $G_{\epsilon,\delta}$. We observe

\begin{eqnarray*}
\mathbb{E}\Big[|J_{\epsilon,\delta}||t-s-(\epsilon+\delta)|^{(2H-2)p}|\Big]&\lesssim& |t-s - (\epsilon+\delta)|^{(2H-2+ 2\gamma H)p}
\end{eqnarray*}
on $G_{\epsilon,\delta}$ and using the fact that $(2H-2+2\gamma H)p +1>0$ and $(2H-2+2\bar{\gamma})p +1>0$, there exists a constant $C = C(H,\bar{\gamma},\gamma, p,T)$ such that 

$$\mathbb{E}\int_{G_{\epsilon,\delta}} |J_{\epsilon,\delta}(s,t)||t-s-(\epsilon+\delta)|^{(2H-2)p}dsdt$$
$$\lesssim \int_{G_{\epsilon,\delta}}|t-s-(\epsilon+\delta)|^{(2H-2+2\gamma H)p}dsdt\le C$$ 
and 
$$\int_{G_{\epsilon,\delta}} |I_{\epsilon,\delta}(s,t)||t-s-(\epsilon+\delta)|^{(2H-2)p}dsdt$$
$$\lesssim \int_{G_{\epsilon,\delta}}|t-s-(\epsilon+\delta)|^{(2H-2+2\bar{\gamma})p}dsdt\le C$$ 
for every $(\epsilon,\delta) \in (0,1)^2$. This concludes the proof.
\end{proof}

We now need to investigate

$$
\mathbb{E}\int_0^T \int_{t-\max\{2\epsilon+\delta; 2\delta +\epsilon\}}^t |Y_{s,t}|^2 \frac{1}{4\epsilon \delta} \big|\Delta_{I_{\epsilon,\delta}}R(s,t)\big|dsdt.
$$
First, we need the following elementary lemma.

\begin{lemma}\label{elemINEQ}
For each $0 < H \le \frac{1}{2}$, the following inequality holds

$$(1-u)^{2H}\ge 1-4Hu,$$
for every $ u\in [0, C^\star_H]$, where $C_H^\star= 1-2^{\frac{-1}{1-2H}}$.
\end{lemma}
\begin{proof}
Let $f(u) = (1-u)^{2H} - (1-4Hu)$ for $0\le u\le 1$. Then, $f(0)=0$ and $f'(u)=-2H(1-u)^{2H-1} + 4H$ for $0 < u < 1$. We observe that $f'(u) >0$ on $[0,C^\star_H)$ and $f'(C^\star_H)=0$, where $0 < C^\star_H \le 1$. This completes the proof.
\end{proof}
\begin{lemma}\label{inpd2}
Let $g \in C^{\bar{\gamma},\gamma}([0,T]\times \mathbb{R}^d)$ with $\bar{\gamma}> \frac{1}{2}-H, \gamma > \frac{1}{2H}-1$ and $Y_\cdot=g(\cdot,B_\cdot)$. Then, there exists a constant $C=C(H,\gamma,\bar{\gamma}, T)$ which only depends on $H$, $\gamma$, $\bar{\gamma}$ and $T$ such that

$$
\mathbb{E}\int_0^T \int_{t-\max\{2\epsilon+\delta; 2\delta +\epsilon\}}^t | Y_{s,t}|^2 \big|\frac{1}{4\epsilon \delta}\Delta_{I_{\epsilon,\delta}}R(s,t)\big|dsdt
$$
\begin{equation}\label{awayest1}
\le C\|g\|^2_{\bar{\gamma},\gamma}  \min \Big\{\epsilon^{[2H-1+2 (\bar{\gamma}\wedge \gamma H)] }, \delta^{[2H-1+2 (\bar{\gamma}\wedge \gamma H)  ]}\Big\},
\end{equation}
for every $(\epsilon,\delta)\in (0,1)^2$.   
\end{lemma}
\begin{proof}
First, we may suppose that $g(t,x) = g(x)$ because both $\gamma H > \frac{1}{2}-H$ and $\bar{\gamma}> \frac{1}{2}-H$ and $\mathbb{E}|B_{s,t}|^{2\gamma}\approxeq |t-s|^{2\gamma H}$. Let us denote 
$$G_{\epsilon,\delta} = \{(s,t) \in \Delta_T; t-\max\{2\epsilon+\delta; 2\delta +\epsilon\}\le s \le t\}.$$ 
One can easily check we can write

$$\frac{1}{4\epsilon\delta}\Delta_{I_{\epsilon,\delta}}R(s,t) = \frac{|t-s|^{2H}}{4\epsilon\delta}\frac{1}{2}\Psi\Bigg(\frac{\epsilon+\delta}{|t-s|},  \frac{\epsilon-\delta}{|t-s|} \Bigg),$$
for $(s,t) \in \Delta_T$, where

$$\Psi(x,y):= |1+x|^{2H} + |1-x|^{2H} - |1+y|^{2H} - |1-y|^{2H}; \quad (x,y) \in \mathbb{R}^2.$$
The H\"older property of $g$ yields  
\begin{small}
\begin{equation}\label{asze1}
\mathbb{E}\int_{G_{\epsilon,\delta}}| Y_{s,t}|^2 \big|\frac{1}{4\epsilon \delta} \Delta_{I_{\epsilon,\delta}}R(s,t)\big|dsdt \lesssim  \int_{G_{\epsilon,\delta}}\frac{|t-s|^{2H+2H\gamma}}{\epsilon\delta}\Bigg|\Psi\Bigg(\frac{\epsilon+\delta}{|t-s|},  \frac{\epsilon-\delta}{|t-s|} \Bigg)\Bigg|dsdt
\end{equation}
\end{small}
Next, we devote our attention to the evaluation of the right-hand side of (\ref{asze1}). Fix the variable $t$. In the sequel, for a given $(\epsilon,\delta) \in (0,1)^2$, we may suppose, for instance, $\delta\le\epsilon$. Observe $2\epsilon + \delta \ge 2\delta+\epsilon\Longleftrightarrow \epsilon \ge \delta$. By making the change of variables, $u = t-s$, we can write

$$\int_{t-\max\{2\epsilon+\delta; 2\delta +\epsilon\}}^t \frac{|t-s|^{2H+2H\gamma}}{\epsilon\delta}\Bigg|\Psi\Bigg(\frac{\epsilon+\delta}{|t-s|},  \frac{\epsilon-\delta}{|t-s|}\Bigg)\Bigg| ds$$
$$= \int_0^{2\epsilon+\delta} \frac{u^{2H+2H\gamma}}{\epsilon\delta}\Big|\Psi\Big(\frac{\epsilon+\delta}{u},  \frac{\epsilon-\delta}{u} \Big)\Big|du $$
$$=\int_0^{2\epsilon+\delta} \frac{u^{2H\gamma}}{\epsilon\delta}\Big| |u+\epsilon+\delta|^{2H} + |u-(\epsilon+\delta)|^{2H}$$
\begin{equation}\label{awa4}
- |u+(\epsilon-\delta)|^{2H} - |u-(\epsilon-\delta)|^{2H} \Big|du.
\end{equation}
Make another change of variable $z = \frac{u}{\epsilon +\delta}$. Then, the right-hand side of (\ref{awa4}) is given by

$$ \frac{(\epsilon+\delta)^{2H(\gamma +1)+1}}{(\epsilon\delta)}\int_0^{\frac{2\epsilon+\delta}{\epsilon+\delta}}z^{2H\gamma}\Big| |z+1|^{2H}$$
$$+ |z-1|^{2H} - \big|z+1 - \frac{2\delta}{\epsilon+\delta}\big|^{2H} - \big| z-1 + \frac{2\delta}{\epsilon+\delta}  \big|^{2H}\Big| dz.$$
Since $\frac{2\epsilon + \delta}{\epsilon+\delta} = 1 + \frac{\epsilon}{\epsilon+\delta}$, we first analyse

$$ \frac{(\epsilon+\delta)^{2H(\gamma +1)+1}}{(\epsilon\delta)}\int_0^{1}z^{2H\gamma}\Big| |z-1|^{2H}  - \big| z-1 + \frac{2\delta}{\epsilon+\delta}  \big|^{2H}\Big|  dz.$$
We observe $1-z- \frac{2\delta}{\epsilon+\delta}\ge 0 \Longleftrightarrow z \le \frac{\epsilon-\delta}{\epsilon+\delta}$. Then, we split
$$ \frac{(\epsilon+\delta)^{2H(\gamma +1)+1}}{(\epsilon\delta)}\int_0^{1}z^{2H\gamma}\Big| |z-1|^{2H}  - \big| z-1 + \frac{2\delta}{\epsilon+\delta}  \big|^{2H}\Big| dz$$
$$ = \frac{(\epsilon+\delta)^{2H(\gamma +1)+1}}{(\epsilon\delta)}\int_0^{\frac{\epsilon-\delta}{\epsilon+\delta}}z^{2H\gamma}\Big| |z-1|^{2H}  - \big| z-1 + \frac{2\delta}{\epsilon+\delta}  \big|^{2H}\Big|  dz$$
$$+ \frac{(\epsilon+\delta)^{2H(\gamma +1)+1}}{(\epsilon\delta)}\int_{\frac{\epsilon-\delta}{\epsilon+\delta}}^1z^{2H\gamma}\Big| |z-1|^{2H}  - \big| z-1 + \frac{2\delta}{\epsilon+\delta}  \big|^{2H}\Big|  dz$$
$$=J_1(k;\epsilon,\delta,t) + J_2(k;\epsilon,\delta,t). $$
Observe
$$
J_1(k;\epsilon,\delta,t) = \frac{(\epsilon+\delta)^{2H(\gamma +1)+1}}{\epsilon\delta}\int_0^{\frac{\epsilon-\delta}{\epsilon+\delta}}z^{2H\gamma }(1-z)^{2H}\Big[ 1- \Big(1-\frac{2\delta}{\epsilon+\delta}\frac{1}{1-z} \Big)^{2H}\Big] dz.
$$
By applying Lemma \ref{elemINEQ}, we get

$$\int_0^T J_1(k;\epsilon,\delta,t)dt \lesssim \frac{(\epsilon+\delta)^{2H(\gamma +1)+1}}{(\epsilon\delta)}\int_0^{\frac{\epsilon-\delta}{\epsilon+\delta}} z^{2H\gamma}(1-z)^{2H-1}dz \frac{2\delta}{\epsilon+\delta} $$
$$\lesssim \epsilon^{2H\gamma + 2H-1} \times \int_0^{\frac{\epsilon-\delta}{\epsilon+\delta}} z^{2H\gamma}(1-z)^{2H-1}dz.$$

We observe

\begin{eqnarray*}
J_2(k;\epsilon,\delta,t) &=&\frac{(\epsilon+\delta)^{2H(\gamma +1)+1}}{\epsilon\delta}\int_{1-\frac{2\delta}{\epsilon+\delta}}^1z^{2H\gamma}(1-z)^{2H}\Big| 1- \Big(\frac{2\delta}{\epsilon+\delta}\frac{1}{1-z}-1 \Big)^{2H}\Big|dz.
\end{eqnarray*}
By using the H\"older continuity of $x\mapsto x^{2H}$, we observe

$$(1-z)^{2H}\Big| 1- \Big(\frac{2\delta}{\epsilon+\delta}\frac{1}{1-z}-1 \Big)^{2H}\Big|$$
$$ = \Big| (1-z)^{2H} - \Big(\frac{2\delta}{\epsilon+\delta}- (1-z)  \Big)^{2H}\Big|$$
$$\lesssim \Big|2(1-z) - \frac{2\delta}{\epsilon+\delta}\Big|^{2H}\le \Big(\frac{2\delta}{\epsilon+\delta}\Big)^{2H},$$
whenever $1-\frac{2\delta}{\epsilon+\delta} < z < 1$. Therefore, integrating w.r.t.~the variable $t$, we have

\begin{eqnarray*}
\int_0^T J_2(k;\epsilon,\delta,t)dt &\lesssim& \frac{(\epsilon+\delta)^{2H(\gamma +1)+1}}{\epsilon\delta}\Big(\frac{2\delta}{\epsilon+\delta}\Big)^{2H} \int_{1-\frac{2\delta}{\epsilon+\delta}}^1z^{2H\gamma}  dzdt\\
&\lesssim& \frac{(\epsilon+\delta)^{2H(\gamma +1)+1}}{\epsilon\delta}\Big(\frac{2\delta}{\epsilon+\delta}\Big)^{2H+1}\\
&=& \frac{(\epsilon+\delta)^{2H\gamma}}{\epsilon} \delta^{2H}\lesssim \epsilon^{2H + 2H\gamma -1}.
\end{eqnarray*}

Next, we treat the integral over $[1,\frac{\epsilon}{\epsilon+\delta}]$. Observing that $(1+x)^{2H}-1\lesssim x$ for every $x\ge 0$, we have

$$ \frac{(\epsilon+\delta)^{2H(\gamma +1)+1}}{\epsilon\delta}\int_{1}^{1+\frac{\epsilon}{\epsilon+\delta}}z^{2H\gamma}\Big| |z-1|^{2H}  - \big| z-1 + \frac{2\delta}{\epsilon+\delta}  \big|^{2H}\Big| dz$$
$$ = \frac{(\epsilon+\delta)^{2H(\gamma +1)+1}}{\epsilon\delta}\int_{1}^{1+\frac{\epsilon}{\epsilon+\delta}}z^{2H\gamma}(z-1)^{2H}\Big[ \Big(1+\frac{2\delta}{\epsilon+\delta}\frac{1}{z-1} \Big)^{2H} -1\Big] dz$$
$$\lesssim \frac{(\epsilon+\delta)^{2H\gamma +2H }}{\epsilon}\int_{1}^{1+\frac{\epsilon}{\epsilon+\delta}}z^{2H\gamma}(z-1)^{2H-1} dz \lesssim \epsilon^{2H+2\gamma H -1}.$$
Summing up the above estimates, we arrive at

$$ \frac{(\epsilon+\delta)^{2H(\gamma +1)+1}}{(\epsilon\delta)}\int_0^T \int_0^{\frac{2\epsilon+\delta}{\epsilon+\delta}}z^{2H\gamma}\Big||z-1|^{2H} - \big| z-1 + \frac{2\delta}{\epsilon+\delta}  \big|^{2H}\Big|dz$$
$$\lesssim \epsilon^{2H+ 2H\gamma-1},$$
for every $\epsilon\ge \delta$. Similarly, 

$$ \frac{(\epsilon+\delta)^{2H(\gamma +1)+1}}{(\epsilon\delta)}\int_0^T \int_0^{\frac{2\epsilon+\delta}{\epsilon+\delta}}z^{2H\gamma}\Big||z+1|^{2H} - \big| z+1 - \frac{2\delta}{\epsilon+\delta}  \big|^{2H}\Big|dzdt$$
$$\lesssim \epsilon^{2H+ 2H\gamma-1},$$
for every $\epsilon\ge \delta$. In case $\epsilon\le \delta$, we have

$$\mathbb{E}\int_{t-\max\{2\epsilon+\delta; 2\delta +\epsilon\}}^t | Y_{s,t}|^2 \big|\frac{1}{4\epsilon \delta}\Delta_{I_{\epsilon,\delta}}R(s,t)\big|ds$$
$$\lesssim \|g\|_{\bar{\gamma},\gamma}^2 \int_0^{2\delta+\epsilon} \Bigg| \frac{u^{2H+2H\gamma}}{\epsilon\delta}\Big|\Psi\Big(\frac{\epsilon+\delta}{u},  \frac{\epsilon-\delta}{u} \Big)\Big|\Bigg|du $$
$$=\int_0^{2\delta + \epsilon} \frac{u^{2H\gamma}}{\epsilon\delta}\Big| |u+\epsilon+\delta|^{2H} + |u-(\epsilon+\delta)|^{2H}$$
\begin{equation}\label{awa5}
- |u+(\epsilon-\delta)|^{2H} - |u-(\epsilon-\delta)|^{2H} \Big| du.
\end{equation}
Make the change of variable $z = \frac{u}{\epsilon +\delta}$. Then, the right-hand side of (\ref{awa5}) is given by

$$ \frac{(\epsilon+\delta)^{2H(\gamma +1)+1}}{\epsilon\delta}\int_0^{\frac{2\delta + \epsilon}{\epsilon+\delta}}z^{2H\gamma}\Big| |z+1|^{2H}$$
$$+ |z-1|^{2H} - \big|z+1 - \frac{2\epsilon}{\epsilon+\delta}\big|^{2H} - \big| z-1 + \frac{2\epsilon}{\epsilon+\delta}  \big|^{2H}\Big|  dz.$$
By symmetry of our argument, we then conclude the proof. 
\end{proof}
By symmetry, Lemmas \ref{inpd1}, \ref{inpd2} and Remark \ref{SingpointwiseR}, we arrive at the following result.
\begin{lemma}\label{detlemma}
Let $g \in C^{\bar{\gamma},\gamma}([0,T]\times \mathbb{R}^d)$ with $\gamma > \frac{1}{2H}-1, \bar{\gamma}> \frac{1}{2}-H$ and $Y_\cdot=g(\cdot,B_\cdot)$. Then
$$-\frac{1}{2}\sum_{i,j=1}^d
\mathbb{E}\int_{[0,T]^2} Y^{(i)}_{s,t} Y^{(j)}_{s,t}\frac{1}{4\epsilon\delta}\mathbb{E}\big[B^{(i)}_{s-\epsilon,s+\epsilon} B^{(j)}_{t-\delta,t+\delta}\big]dsdt \rightarrow$$
$$ -\frac{1}{2}\mathbb{E}\int_{[0,T]^2} |Y_{s,t}|^2 \frac{\partial^2R}{\partial t\partial s}(s,t) dsdt,$$
as $(\epsilon,\delta)\downarrow 0$.
\end{lemma}

\subsection{Random component}

In this section, we investigate the limit of

$$\lim_{(\epsilon,\delta)\downarrow 0}-\frac{1}{2} \mathbb{E} \int_{[0,T]^2} \big\langle Y_{s,t}\otimes Y_{s,t}, \mathcal{W}^0(\epsilon,\delta;s,t)\big\rangle dsdt,$$
where 
$$\mathcal{W}^{0,ij}(\epsilon,\delta;s,t):= \frac{1}{4\epsilon\delta} \Big\{Z^{1;ij}_{s,t}(\epsilon)Z^{2;ij}_{s,t}(\delta) - \mathbb{E}\big[ Z^{1;ij}_{s,t}(\epsilon)Z^{2;ij}_{s,t}(\delta)\big]\Big\},$$
for $(s,t) \in [0,T]^2_\star$ and $(\epsilon,\delta) \in (0,1)^2$. Recall that $(Z^{1;ij}, Z^{2;ij})$ is given by (\ref{2dcexp2}), (\ref{Z1epsilon}) and (\ref{Z2delta}).

By Lemma \ref{Lambdalimit}, we know that

$$\lim_{(\epsilon,\delta)\downarrow 0}\mathcal{W}^{0,ij}(\epsilon,\delta;s,t) = \mathcal{W}^{ij}(s,t)$$
almost surely, for each $(s,t) \in [0,T]^2_\star$ and $1\le i,j\le d$. In this section, we will prove that

\begin{equation}\label{family}
\Big\{ Y^{(i)}_{s,t} Y^{(j)}_{s,t} \mathcal{W}^{0,ij}(\epsilon,\delta;s,t); (s,t) \in \Delta_T; (\epsilon,\delta) \in(0,1)^2 \Big\}
\end{equation}
is an uniformly integrable family in $L^1(\Omega\times \Delta_T)$, for each $1\le i,j\le d$.

\begin{remark}\label{remsimplex}
By symmetry and projecting onto the sigma algebra generated by $B_t, B_{t,s}$ for $0\le t< s\le T$, we will also conclude the convergence on the other part of the simplex in $[0,T]^2_\star$.
\end{remark}

The following elementary estimates are important in this section.
\begin{lemma}\label{growthELL}
For a given $0 < H < \frac{1}{2}$ and $a>0$, let $\ell_a(x) = |a+x|^{2H} - |a-x|^{2H}; x \in \mathbb{R}$. Then,

$$\sup_{ x >0} \Big|\frac{\ell_a(x)}{x}\Big|=2^{2H}a^{2H-1}.$$

\end{lemma}

\begin{proof}

First, we observe $\lim_{x\downarrow 0} \frac{\ell_1(x)}{x} = 4H$. Then, we can define the continuous function
$$
f(x)=\left\{
\begin{array}{rl}
\frac{(1+x)^{2H} - (1-x)^{2H}}{x}; & \hbox{if} \ 0 < x \le 1 \\
\frac{(1+x)^{2H} - (x-1)^{2H}}{x};& \hbox{if} \ 1 < x < \infty \\
4H;& \hbox{if} \ x=0.\\
\end{array}
\right.
$$
The function $x \mapsto x^{2H}$ is $2H$-H\"older continuous over $\mathbb{R}_+$ with H\"older seminorm equals 1 and hence
$$\Big|\frac{(1+x)^{2H} - (x-1)^{2H}}{x}\Big|\le 2^{2H}$$
for every $x\ge 1$. There is no critical point for $f$ in $(0,1)$ and $4H\le 2^{2H}$ and hence we conclude the proof for $a=1$. In the general case $a>0$, we just observe we shall write

$$\Big|\frac{\ell_a (x)}{x}\Big| = a^{2H-1}  \frac{a}{x} \Big| \big|1+\frac{x}{a}\big|^{2H} - \big|1-\frac{x}{a}\big|^{2H}  \Big|$$
and by applying the previous case, we conclude the proof.
\end{proof}
There is slight abuse of notation here. We recall $\ell_1 = \ell$ as described in (\ref{ellfunc}). The following elementary result is useful for the next lemmas. 
\begin{lemma}\label{covfest}
Fix $0 < H < \frac{1}{2}$. Let $\varphi(s,t) = R(s,t) - \gamma(s)$. Then,

$$|\varphi(s,t)|\le \min \big\{s^{H}|t-s|^H; |t-s|^{2H}; s^{2H}\big\},$$
for every $(s,t) \in \Delta_T$.
\end{lemma}
\begin{proof}
H\"older's inequality yields
$$|\varphi(s,t)|\le \mathbb{E}|B^{(1)}_s B^{(1)}_{s,t}|\le s^H |t-s|^{H},$$
 for $(s,t) \in [0,T]^2$. The function $x \mapsto x^{2H}$ is $2H$-H\"older continuous over $\mathbb{R}_+$ with H\"older seminorm equals 1 and hence

$$|\varphi(s,t)| = \frac{1}{2} \big|\big\{t^{2H} - s^{2H} - |t-s|^{2H}  \big\}\big|\le |t-s|^{2H},$$
for $(s,t) \in [0,T]^2$. If $0\le s\le t\le T$, then

\begin{eqnarray*}
|\varphi(s,t)| &=& \frac{1}{2}\big\{s^{2H}+ (t-s)^{2H}-t^{2H} \big\}\\
&\le& \frac{1}{2}s^{2H} + \frac{1}{2}|(t-s)^{2H} - t^{2H}|\\
&\le& s^{2H}.
\end{eqnarray*}
\end{proof}

In the sequel, for each $(\epsilon,\delta) \in (0,1)^2$, we will split the simplex set $\Delta_T = \cup_{i=1}^4 \Delta_{i,T}(\epsilon,\delta)$, where

$$\Delta_{1,T}(\epsilon,\delta):=\{(s,t) \in \Delta_T; \delta < t , \epsilon < s\},$$
$$\Delta_{2,T}(\epsilon,\delta):=\{(s,t) \in \Delta_T; 0\le t\le \delta, 0\le s \le \epsilon\},$$

$$\Delta_{3,T}(\epsilon,\delta):=\{(s,t) \in \Delta_T; 0\le s \le \epsilon, t > \delta\},$$

$$\Delta_{4,T}(\epsilon,\delta):=\{(s,t) \in \Delta_T; s > \epsilon, 0\le t\le \delta\}. $$
The analysis over $\Delta_{2,T}(\epsilon,\delta)$ is unnecessary as the following elementary lemma shows. It is a straightforward application of H\"older's inequality.

\begin{lemma}
Let $g \in C^{\bar{\gamma},\gamma}([0,T]\times \mathbb{R}^d)$ with $\bar{\gamma}> \frac{1}{2}-H, \gamma > \frac{1}{2H}-1$ and $Y_\cdot=g(\cdot, B_\cdot)$. Then, there exists a constant $C = C(H,T)$ which depends on $H$ and $T$ such that

$$\frac{1}{4\epsilon\delta}\mathbb{E}\int_0^\delta \int_0^\epsilon |Y^{(i)}_{s,t} Y^{(j)}_{s,t}| |B^{(i)}_{s-\epsilon,s+\epsilon}B^{(j)}_{t-\delta,t+\delta}|dsdt\le \|g\|^2_{\bar{\gamma},\gamma}  (2\delta)^H (2\epsilon)^H,$$
for every $(\epsilon,\delta) \in (0,1)^2$ and $i,j=1,\ldots, d$.
\end{lemma}

In the sequel, we repeatedly make use of the following inequality:  

\begin{eqnarray*}
\nonumber \mathbb{E}\int_{\Delta_{r,T}(\epsilon
,\delta)} |Y^{(i)}_{s,t}Y^{(j)}_{s,t}|^p |\mathcal{W}^{0,ij}(\epsilon,\delta;s,t)|^pdsdt&\le& \|g\|^{2p}_{\bar{\gamma},\gamma}\mathbb{E}\int_{\Delta_{r,T}(\epsilon,\delta)}\big|\textbf{d}^{\bar{\gamma},\gamma} \big((s,B_s); (t,B_t) \big)\big|^{2p}\\
&\times& |\mathcal{W}^{0,ij}(\epsilon,\delta;s,t)|^p dsdt,
\end{eqnarray*}
for $1\le i,j\le d$, $r=1,3,4$ and $\mathcal{W}^{0,ij}(\epsilon,\delta;s,t)$ is given by (see (\ref{Z1epsilon}) and (\ref{Z2delta}))





\begin{eqnarray}
\label{prodZ} \mathcal{W}^{0,ij}(\epsilon,\delta;s,t)  &=& \eta_{11}(\epsilon,s,t)\eta_{21}(\delta,s,t)\big[B^{(i)}_sB^{(j)}_s - \gamma(s)\delta_{ij}\big]\frac{1}{4\epsilon\delta}\\
\nonumber&+& \eta_{11}(\epsilon,s,t)\eta_{22}(\delta,s,t)\big[B^{(i)}_s B^{(j)}_{s,t}- \varphi(s,t)\delta_{ij}\big]\frac{1}{4\epsilon\delta}\\
\nonumber&+& \eta_{12}(\epsilon,s,t)\eta_{21}(\delta,s,t)\big[B^{(j)}_s B^{(i)}_{s,t} -\varphi(s,t)\delta_{ij}\big]\frac{1}{4\epsilon\delta} \\
\nonumber &+& \eta_{12}(\epsilon,s,t) \eta_{22}(\delta,s,t)\big[B^{(i)}_{s,t}B^{(j)}_{s,t} - \text{Var}(B^{(i)}_{s,t})\delta_{ij}\big]\frac{1}{4\epsilon\delta},
\end{eqnarray}
for $(s,t) \in \Delta_T$. In case $(s,t)\in \Delta_{1,T}(\epsilon,\delta)$, we have

\begin{equation}\label{etaexp}
\eta_{11}(\epsilon,s,t) =\frac{1}{\Theta_{s,t}} \Big\{\frac{s^{2H}}{2}\ell \Big(\frac{\epsilon}{s}\Big)|t-s|^{2H} - \frac{|t-s|^{2H}}{2}\varphi(s,t) \ell \Big( \frac{\epsilon}{t-s}\Big)\Big\},
\end{equation}

$$\eta_{12}(\epsilon,s,t) = \frac{1}{\Theta_{s,t}} \Bigg\{\frac{|t-s|^{2H}}{2} s^{2H} \ell \Big( \frac{\epsilon}{t-s} \Big)  -  \frac{s^{2H}}{2}\ell \Big( \frac{\epsilon}{s}\Big)\varphi(s,t)\Bigg\},$$

\begin{eqnarray*}
\eta_{21}(\delta,s,t) &=& \frac{1}{\Theta_{s,t}}\Big\{ \Delta \ell(\delta,s,t) |t-s|^{2H} - \varphi(s,t) \frac{|t-s|^{2H}}{2} \ell \Big(\frac{\delta}{t-s}\Big)\Big\},
\end{eqnarray*}

\begin{eqnarray*}
\eta_{22}(\delta,s,t) &=& \frac{1}{\Theta_{s,t}} \Big\{ s^{2H}\frac{|t-s|^{2H}}{2} \ell \Big(\frac{\delta}{t-s}\Big) -\Delta \ell(\delta,s,t)\varphi(s,t)\Big\}.
\end{eqnarray*}

In the following technical lemmas, we repeatedly make use of Lemmas \ref{fhlemma}, \ref{growthELL} and \ref{covfest}. The estimates for $\mathcal{W}^{0,ij}(\epsilon,\delta)$ will be based on the decomposition (\ref{prodZ}). 

\begin{lemma}\label{LemmaF1}
Let $g \in C^{\bar{\gamma},\gamma}([0,T]\times \mathbb{R}^d)$ with $\bar{\gamma}> \frac{1}{2}-H, \gamma > \frac{1}{2H}-1$ and $Y_\cdot=g(\cdot,B_\cdot)$. For each $1 < p < \frac{1}{2-2H-2(\gamma H\wedge \bar{\gamma})}$, there exists a constant $C = C(\gamma, \bar{\gamma}, H,T)$ such that 
\begin{equation}\label{F1ex}
\mathbb{E}\int_{\Delta_{1,T}(\epsilon,\delta)}\Big| \frac{1}{4\epsilon\delta}Y^{(i)}_{s,t}Y^{(j)}_{s,t} \eta_{12}(\epsilon,s,t) \eta_{22}(\delta,s,t)\big[B^{(i)}_{s,t} B^{(j)}_{s,t} - |t-s|^{2H}\delta_{ij}\big]\Big|^pdsdt\le C \| g\|^{2p}_{\bar{\gamma},\gamma},
\end{equation}
for every $(\epsilon,\delta) \in (0,1)^2$ and $i,j=1,\ldots, d$.  
\end{lemma}
\begin{proof}
Without loss of generality, we may assume $g$ is time-homogeneous, i.e., $g(t,x) = g(x)$ for $(t,x) \in [0,T]\times \mathbb{R}^d$. 
Fix $i,j=1,\ldots, d$, $\gamma > \frac{1}{2H}-1$ and $1 < p < \frac{1}{2-2H-2\gamma H}$. In order to shorten notation, we define

$$\beta^{ij}_1(\gamma,H, s,t):= \big|\textbf{d}^{\bar{\gamma},\gamma}\big((s,B_s); (t,B_t) \big)\big|^2  \big| B^{(i)}_{s,t}B^{(j)}_{s,t} - |t-s|^{2H}\delta_{ij}\big|,$$
for $(s,t) \in \Delta_T$. Throughout this proof, we repeatedly use the following elementary estimate

\begin{equation}\label{beta1esta}
\beta^{ij}_1(\gamma,H,s,t)\lesssim|B_{s,t}|^{2\gamma}|B^{(i)}_{s,t}B^{(j)}_{s,t}|+  |t-s|^{2H}|B_{s,t}|^{2\gamma}\delta_{ij}.
\end{equation}


For each $1 < q < \infty$, we observe

\begin{equation}\label{beta1estb}
\mathbb{E}|\beta^{ij}_1(\gamma,H,s,t)|^q\lesssim |t-s|^{(2\gamma + 2)Hq},
\end{equation}
for $(s,t) \in \Delta_T$. In the sequel, in order to shorten notation, we write


$$I^{ij}_1(\epsilon,\delta,s,t):=\beta^{ij}_1(\gamma,H, s,t)s^{2H}\Big|\ell \Big( \frac{\epsilon}{s}\Big)\Big|\varphi^2(s,t)|\Delta \ell (\delta,s,t)|\Theta^{-2}_{s,t}\frac{1}{\epsilon\delta}$$  

$$I^{ij}_2(\epsilon,\delta,s,t):=\beta^{ij}_1(\gamma,H, s,t) s^{4H}\Big|\ell \Big( \frac{\epsilon}{s}\Big)\ell \Big(\frac{\delta}{t-s}\Big)\Big||\varphi(s,t)| |t-s|^{2H}\Theta^{-2}_{s,t}\frac{1}{\epsilon\delta}$$ 
$$I^{ij}_3(\epsilon,\delta,s,t):=\beta^{ij}_1(\gamma,H, s,t) |t-s|^{2H} s^{2H} \Big|\ell \Big( \frac{\epsilon}{t-s} \Big)\Big| |\Delta \ell (\delta,s,t)\varphi(s,t)|\Theta^{-2}_{s,t}\frac{1}{\epsilon\delta}$$ 
$$ I^{ij}_4(\epsilon,\delta,s,t):=\beta^{ij}_1(\gamma,H, s,t) |t-s|^{4H} s^{4H} \Big|\ell \Big( \frac{\epsilon}{t-s} \Big) \ell \Big( \frac{\delta}{t-s} \Big)\Big|\Theta^{-2}_{s,t}\frac{1}{\epsilon\delta},$$
for $(s,t) \in \Delta_{1,T}(\epsilon,\delta)$. By using the expressions (\ref{etaexp}) for $\eta_{12}$ and $\eta_{22}$ and the H\"older property of $g$, we have

$$\mathbb{E}\int_{\Delta_{1,T}(\epsilon,\delta)}\Big| \frac{1}{4\epsilon\delta}  Y^{(i)}_{s,t} Y^{(j)}_{s,t} \eta_{12}(\epsilon,s,t)\eta_{22}(\delta,s,t)[B^{(i)}_{s,t} B^{(j)}_{s,t}  - |t-s|^{2H}\delta_{ij}]  \Big|^pdsdt$$

$$\lesssim \|g\|^{2p}_{\bar{\gamma},\gamma} \sum_{m=1}^4 \int_{\Delta_{1,T}(\epsilon,\delta)} \mathbb{E}|I^{ij}_m(\epsilon,\delta,s,t)|^{p} dsdt,$$
By using Lemmas \ref{fhlemma} and \ref{growthELL}, the fact $\Theta_{s,t} = |t-s|^{2H}A(s,t)$ for $(s,t) \in \Delta_T$ and (\ref{beta1esta}), we have

\begin{equation}\label{I4eT1}
I^{ij}_4(\epsilon,\delta,s,t)\lesssim  \beta^{ij}_1(\gamma,H,s,t)|t-s|^{-2},
\end{equation}
for $(s,t) \in \Delta_T$. Then, 


$$\int_{\Delta_{1,T}(\epsilon,\delta)}\mathbb{E}|I^{ij}_4(\epsilon,\delta,s,t)|^{p}dsdt\lesssim \int_{\Delta_T} |t-s|^{(2\gamma H + 2H-2)p}dsdt<\infty.$$


Next, we estimate $I^{ij}_3(\epsilon,\delta,s,t)$. Lemma \ref{growthELL} yields

\begin{equation}\label{brest}
|\Delta \ell(\delta,s,t)|\lesssim \delta t^{2H-1} + \delta (t-s)^{2H-1},
\end{equation}
for $(s,t) \in \Delta_T$ and $\delta \in (0,1)$. By applying (\ref{brest}) and (\ref{deterest}) we get

\begin{eqnarray*}
|I^{ij}_3(\epsilon,\delta,s,t)|&\lesssim & \beta^{ij}_1(\gamma,H,s,t)|t-s|^{-2H-1}s^{-2H}t^{2H-1}|\varphi(s,t)|\\
&+& |t-s|^{-2}s^{-2H}|\varphi(s,t)|\beta_1(\gamma,H,s,t),
\end{eqnarray*}
for $(s,t) \in \Delta_T$ and hence, the estimates (\ref{beta1esta}) and (\ref{beta1estb}) jointly with Lemma \ref{covfest} yield

\begin{eqnarray*}
\int_{\Delta_T}\mathbb{E}|I^{ij}_3(\epsilon,\delta,s,t)|^{p}dsdt &\lesssim& \int_{\Delta_T} |t-s|^{(2\gamma H-1)p} t^{(2H-1)p}dsdt\\
& + & \int_{\Delta_T}|t-s|^{(2\gamma H + 2H-2)p}dsdt<\infty. 
\end{eqnarray*}


Next, by applying again (\ref{deterest}), we have 

$$|I^{ij}_2(\epsilon,\delta,s,t)|\lesssim \beta^{ij}_1 (\gamma,H,s,t) s^{-1}|\varphi(s,t)||t-s|^{-2H-1}$$
for $(s,t) \in \Delta_T$. Therefore, 

\begin{eqnarray*}
\int_{\Delta_T}\mathbb{E}|I^{ij}_2(\epsilon,\delta,s,t)|^{p}dsdt &\lesssim& \int_{\Delta_T} |t-s|^{(2\gamma H -1)p}s^{-p}|\varphi(s,t)|^{p}dsdt\\
&=& \int_0^T \int_0^{\frac{t}{2}} |t-s|^{(2\gamma H-1)p}s^{-p}|\varphi(s,t)|^{p}dsdt\\
&+& \int_0^T \int_{\frac{t}{2}}^t |t-s|^{(2\gamma H-1)p}s^{-p}|\varphi(s,t)|^{p}dsdt.
\end{eqnarray*}

Lemma \ref{covfest} yield 
\begin{eqnarray*}
\int_0^T \int_0^{\frac{t}{2}} |t-s|^{(2\gamma H-1)p }s^{-p}|\varphi(s,t)|^{p} dsdt&\lesssim& \int_0^T t^{(2\gamma H-1)p}\int_0^{\frac{t}{2}} s^{-p+2Hp}dsdt\\
&\lesssim& \int_0^T t^{(2\gamma H + 2H-2)p +1}dt < \infty
\end{eqnarray*}
and 

\begin{eqnarray*}
\int_0^T \int_{\frac{t}{2}}^t |t-s|^{(2\gamma H-1)p}s^{-p}|\varphi(s,t)|^{p}dsdt&\lesssim&\int_0^T t^{-p}\int_{\frac{t}{2}}^t |t-s|^{(2\gamma H + 2H-1)p}dsdt\\
&\lesssim& \int_0^T t^{(2\gamma H + 2H-2)p +1}dt < \infty. 
\end{eqnarray*}
Next, we analyze $I^{ij}_1$ as follows. By using (\ref{brest}) and (\ref{deterest}), we get

\begin{eqnarray}
\nonumber|I^{ij}_1(\epsilon,\delta,s,t)|&\lesssim& \frac{\beta^{ij}_1(\gamma,H,s,t) s^{-1-2H} t^{2H-1}\varphi^2(s,t)}{|t-s|^{4H}}\\
\label{I1eT1a}&+& \frac{\beta^{ij}_1(\gamma,H,s,t) s^{-1-2H} |t-s|^{2H-1}\varphi^2(s,t)}{|t-s|^{4H}}
\end{eqnarray}
for $(s,t) \in \Delta_T$. A direct application of Lemma \ref{covfest} yields 

\begin{eqnarray*}
\int_{\Delta_T}\mathbb{E}|I^{ij}_1(\epsilon,\delta,s,t)|^{p} dsdt&\lesssim& \int_{\Delta_T}|t-s|^{(2\gamma-2)p H} s^{-(2H+1)p}t^{(2H-1)p} |\varphi(s,t)|^{2p}dsdt\\
&+& \int_{\Delta_T}  |t-s|^{(2\gamma H-1)p} s^{-(2H+1)p} |\varphi(s,t)|^{2p}dsdt < \infty. 
\end{eqnarray*}
This concludes the proof. 
\end{proof}

\begin{lemma}\label{deltaLlemma}
Let $\Delta \ell (\delta,s,t) = \frac{t^{2H}}{2}\ell \Big( \frac{\delta}{t}\Big) - \frac{(t-s)^{2H}}{2} \ell \Big( \frac{\delta}{t-s}\Big)$ for $(s,t) \in \Delta_T$. Then, there exists a constant $C$ which depends on $H$ such that

\begin{eqnarray}
\nonumber|\Delta \ell(\delta,s,t)|&\le& C \Big\{s^{2H}\frac{\delta}{t} + |t-s|^{2H-1}\Big(1-\frac{\delta}{t} \Big)^{2H-1}\frac{\delta s}{t}\Big\}\\
\label{dest1}&+& C \Big\{(t-s-\delta)^{2H-1}\frac{\delta s}{t}|t-s|^{-2H}\Big\},
\end{eqnarray}
for $0 < s < t-\delta$ and $\delta \in (0,1)$. Moreover,

\begin{equation}\label{dest2}
s^{-2H}|\Delta \ell(\delta,s,t)|\le 2 + (2T)^{2H},
\end{equation}
for every $t > \delta$ and $t-\delta < s < t\le T$.
\end{lemma}
\begin{proof}
Lemma \ref{growthELL} yields

$$|\Delta \ell(\delta,s,t)| = \Big|\frac{t^{2H}}{2} \ell \Big( \frac{\delta}{t} \Big) - \frac{|t-s|^{2H}}{2}\ell \Big( \frac{\delta}{t-s}\Big)\Big|$$
$$\lesssim \Big| \big(t^{2H} - (t-s)^{2H}\big) \ell \Big( \frac{\delta}{t}\Big)   \Big| +  \big(t - s \big)^{2H} \Big|\ell \Big( \frac{\delta}{t}\Big) - \ell \Big( \frac{\delta}{t-s}\Big)   \Big|$$
$$
\lesssim s^{2H} \frac{\delta}{t} + \big(t - s \big)^{2H} \Big|\ell \Big( \frac{\delta}{t}\Big) - \ell \Big( \frac{\delta}{t-s}\Big)   \Big|,
$$
for $(s,t)\in \Delta_T$. Fix $t$ and denote

$$f(s)=\ell\Big(\frac{\delta}{t}\Big) -\ell \Big( \frac{\delta}{t-s}\Big); 0\le s < t-\delta.$$
Observe $f(0)=0$, $f'(s) < 0$ for $s \in (0, t-\delta)$. The fact that $f$ is decreasing on $(0,t-\delta)$ and $f(0)=0$ imply $f(s)<0$ for $s \in (0,t-\delta)$. Then,

\begin{eqnarray*}
\Big|\ell\Big(\frac{\delta}{t}\Big) -\ell \Big( \frac{\delta}{t-s}\Big)\Big|&=& \Big( 1+ \frac{\delta}{t-s} \Big)^{2H} -  \Big( 1+ \frac{\delta}{t} \Big)^{2H}\\
&+& \Big( 1- \frac{\delta}{t} \Big)^{2H} -  \Big( 1- \frac{\delta}{t-s} \Big)^{2H},
\end{eqnarray*}
Because $2H-1 < 0$, a second order Taylor expansion yields

$$0 < \Big( 1+ \frac{\delta}{t-s} \Big)^{2H} -  \Big( 1+ \frac{\delta}{t} \Big)^{2H} < 2H \Big( 1  - \frac{\delta}{t} \Big)^{2H-1}\frac{\delta s}{t(t-s)},$$
for $s \in (0,t-\delta)$. Similarly,

\begin{eqnarray*}
0 < \Big( 1- \frac{\delta}{t} \Big)^{2H} -  \Big( 1- \frac{\delta}{t-s} \Big)^{2H} &<& 2H \Big( 1  - \frac{\delta}{t-s} \Big)^{2H-1}\frac{\delta s}{t(t-s)}\\
&=& 2H \Big( t-s-\delta \Big)^{2H-1} (t-s)^{1-2H}\frac{\delta s}{t(t-s)}
\end{eqnarray*}
for $s \in (0,t-\delta)$. This shows (\ref{dest1}). Next, we claim that

\begin{eqnarray}
\nonumber 2|\Delta \ell (\delta,s,t)|&=& t^{2H} \Bigg\{ \Big( 1 + \frac{\delta}{t} \Big)^{2H} -  \Big( 1-\frac{\delta}{t} \Big)^{2H}  \Bigg\}\\
\label{int3}&-& (t-s)^{2H} \Bigg\{ \Big( 1 + \frac{\delta}{t-s} \Big)^{2H} -  \Big(\frac{\delta}{t-s}-1 \Big)^{2H}  \Bigg\},
\end{eqnarray}
whenever $0 < t-\delta < s < t$. Indeed,

\begin{eqnarray*}
 - \Delta \ell(\delta,s,t) &=& (t-s)^{2H} \Bigg\{ \Big( 1 + \frac{\delta}{t-s} \Big)^{2H} -  \Big(\frac{\delta}{t-s}-1 \Big)^{2H}  \Bigg\}\\
 &-& t^{2H} \Bigg\{ \Big( 1 + \frac{\delta}{t} \Big)^{2H} -  \Big( 1-\frac{\delta}{t} \Big)^{2H}  \Bigg\}.
\end{eqnarray*}
We observe $-\delta \ell (\delta,t,t) = (t-\delta)^{2H} -(t+\delta)^{2H} < 0$ for $t > \delta$. Moreover, Lemma \ref{growthELL} yield

\begin{eqnarray*}
\nonumber-\Delta \ell (\delta,t-\delta,t) &=& \Big( 1 + \frac{\delta}{t}\Big)^{2H} - \Big( 1  - \frac{\delta}{t}\Big)^{2H} - 2^{2H}\\
\nonumber&<& \frac{t}{\delta} \Big\{ \Big( 1 + \frac{\delta}{t} \Big)^{2H} - \Big( 1 - \frac{\delta}{t} \Big)^{2H}  \Big\}-2^{2H}\\
&\le& 0,
\end{eqnarray*}
for $ t > \delta$. Since $-\frac{\partial \Delta \ell}{\partial s}(\delta,s,t) < 0$ for $t-\delta < s < t$, we can safely state that $-\Delta \ell(\delta,s,t) < 0$ for $0 < t-\delta < s< t$. The estimate (\ref{int3}) yields

$$s^{-2H}|\Delta \ell (\delta,s,t)|$$
$$  =  s^{-2H} \{(t+\delta)^{2H} - (t-\delta)^{2H} \} + s^{-2H} \{ (\delta-(t-s))^{2H}  - (t-s + \delta)^{2H}  \}$$
$$\le s^{-2H} \{(t+\delta)^{2H} - (t-s+ \delta)^{2H}\} + s^{-2H} \big| (\delta-(t-s))^{2H} - (t-\delta)^{2H}\big|$$
$$\le 2 + (2T)^{2H},$$
for every $t>\delta$ and $t-\delta < s < t\le T$. This concludes the proof.
\end{proof}

\begin{lemma}\label{LemmaF2}
Let $g \in C^{\bar{\gamma},\gamma}([0,T]\times \mathbb{R}^d)$ with $\bar{\gamma}> \frac{1}{2}-H, \gamma > \frac{1}{2H}-1$ and $Y_\cdot=g(\cdot, B_\cdot)$. For each $1 < p < \min \big\{\frac{1}{2-2H-2(\gamma H \wedge \bar{\gamma}) }; \frac{1}{1-(\gamma H\wedge \bar{\gamma})}\big\}$, there exists a constant $C = C(\gamma,\bar{\gamma},H,T,p)$ such that 

\begin{equation}\label{F2ex}
\mathbb{E}\int_{\Delta_{1,T}(\epsilon,\delta)}\Big| \frac{1}{4\epsilon\delta}Y^{(i)}_{s,t}Y^{(j)}_{s,t} \eta_{11}(\epsilon,s,t) \eta_{21}(\delta,s,t)\big[B^{(i)}_{s} B^{(j)}_s  - \gamma(s)\delta_{ij}\big]\Big|^pdsdt\le C \| g\|^{2p}_{\bar{\gamma},\gamma},
\end{equation}
for every $(\epsilon,\delta) \in (0,1)^2$ and $i,j=1,\ldots, d$. 
\end{lemma}
\begin{proof}
Fix $i,j=1,\ldots, d$, $\gamma > \frac{1}{2H}-1$ and $1 < p < \min \big\{\frac{1}{2-2H-2\gamma H}; \frac{1}{1-\gamma H}\big\}$. Similar to Lemma \ref{LemmaF1}, we may suppose $g$ is time-homogeneous. In order to shorten notation, we define
$$\beta^{ij}_2(\gamma,H, s,t):= \big|\textbf{d}^{\bar{\gamma},\gamma} \big((s,B_s); (t,B_t)\big)\big|^2 \big| B^{(i)}_{s} B^{(j)}_s - s^{2H}\delta_{ij}\big|,$$
for $(s,t) \in \Delta_T$. Throughout this proof, we repeatedly use the following elementary estimate

\begin{equation}\label{beta2esta}
\beta^{ij}_2(\gamma,H,s,t)\lesssim |B_{s,t}|^{2\gamma} |B^{(i)}_{s} B^{(j)}_s|    + s^{2H}|B_{s,t}|^{2\gamma}\delta_{ij},
\end{equation}
for $(s,t) \in \Delta_T$. 
For each $1 < q < \infty$, we observe

\begin{equation}\label{beta2estb}
\mathbb{E}|\beta^{ij}_2(\gamma,H,s,t)|^q\lesssim |t-s|^{2\gamma qH} s^{2qH}
\end{equation}
whenever $0 < t -s < 1$ or $s$ is close to the origin. In the sequel, in order to shorten notation, we write 

$$I^{ij}_1(\epsilon,\delta,s,t):=\beta^{ij}_2(\gamma,H, s,t)s^{2H}\Big|\ell \Big(\frac{\epsilon}{s}\Big) \Big| |t-s|^{4H}| \Delta \ell (\delta,s,t)|\Theta^{-2}_{s,t}\frac{1}{\epsilon\delta}$$

$$I^{ij}_2(\epsilon,\delta,s,t):=\beta^{ij}_2(\gamma,H, s,t) |t-s|^{4H}|\varphi(s,t)| \Big| \ell \Big(\frac{\epsilon}{t-s}\Big) \Big| | \Delta\ell(\delta,s,t)|  \Theta^{-2}_{s,t}\frac{1}{\epsilon\delta}$$ 

$$I^{ij}_3(\epsilon,\delta,s,t):=\beta^{ij}_2(\gamma,H, s,t)  s^{2H} \Big| \ell \Big(\frac{\epsilon}{s}\Big) \Big| |t-s|^{4H} |\varphi(s,t)|  \Big| \ell \Big(\frac{\delta}{t-s}\Big) \Big| \Theta^{-2}_{s,t}\frac{1}{\epsilon\delta}$$

$$ I^{ij}_4(\epsilon,\delta,s,t):=\beta^{ij}_2(\gamma,H, s,t) |t-s|^{4H}\varphi^2(s,t)  \Big| \ell \Big( \frac{\epsilon}{t-s}\Big) \ell \Big(\frac{\delta}{t-s}\Big) \Big| \Theta^{-2}_{s,t}\frac{1}{\epsilon\delta},$$
for $(s,t) \in \Delta_{T}$. By using (\ref{etaexp}) for $\eta_{11}$ and $\eta_{21}$ and the H\"older property of $g$, we have

$$\mathbb{E}\int_{\Delta_{1,T}(\epsilon,\delta)}\Big| \frac{1}{4\epsilon\delta}  Y^{(i)}_{s,t} Y^{(j)}_{s,t} \eta_{11}(\epsilon,s,t)\eta_{21}(\delta,s,t)[B^{(i)}_{s} B^{(j)}_{s}  - s^{2H}\delta_{ij}]  \Big|^pdsdt$$

$$\lesssim\|g\|^{2p}_{\bar{\gamma},\gamma} \sum_{m=1}^4 \int_{\Delta_{1,T}(\epsilon,\delta)} \mathbb{E}|I^{ij}_m(\epsilon,\delta,s,t)|^{p} dsdt.$$
We first analyse the components with $I^{ij}_4$ and $I^{ij}_3$. By Lemma \ref{fhlemma}, we have

$$I^{ij}_4(\epsilon,\delta,s,t)\lesssim \beta^{ij}_2(\gamma,H,s,t) \varphi^2(s,t) s^{-4H}|t-s|^{-2}$$
and 
$$I^{ij}_3(\epsilon,\delta,s,t)\lesssim \beta^{ij}_2(\gamma,H,s,t)s^{-2H-1}|t-s|^{-1}|\varphi(s,t)|,$$
for $(s,t) \in \Delta_T$. Lemma \ref{covfest} and (\ref{beta2estb}) yield 

$$
\int_{\Delta_{1,T}(\epsilon,\delta)} \mathbb{E}|I^{ij}_4(\epsilon,\delta,s,t)|^{p} dsdt\lesssim \int_{\Delta_T} |t-s|^{(2\gamma H-2)p}|\varphi(s,t)|^{2p}s^{-2Hp}dsdt<\infty
$$
and 

$$\int_{\Delta_T}\mathbb{E}|I^{ij}_3(\epsilon,\delta,s,t)|^{p}dsdt\lesssim\int_{\Delta_T}|t-s|^{(2\gamma H-1)p} s^{-p}|\varphi(s,t)|^{p}dsdt<\infty,$$
for $1 < p < \frac{1}{2-2\gamma H -2H}$. Next, we analyse the component with $I^{ij}_2$. By Lemma \ref{growthELL} and (\ref{brest}), we have




\begin{eqnarray*}
I^{ij}_2(\epsilon,\delta,s,t)&\lesssim& \beta^{ij}_2(\gamma,H,s,t) |t-s|^{-1} |\varphi(s,t)| t^{2H-1} s^{-4H}\\
&+&\beta^{ij}_2(\gamma,H,s,t) |t-s|^{2H-2} |\varphi(s,t)| s^{-4H},
\end{eqnarray*}
for $(s,t) \in \Delta_{T}$ and hence Lemma \ref{covfest} allows us to state

$$\int_{\Delta_T}\mathbb{E}|I^{ij}_2(\epsilon,\delta,s,t)|^{p}dsdt\lesssim \int_{\Delta_T} |t-s|^{(2\gamma H + 2H-2)p} |\varphi(s,t)|^{p}s^{-2pH}dsdt<\infty,$$
for $1 < p < \frac{1}{2-2\gamma H -2H}$. Next, we treat the term with $I^{ij}_1$. We will check that

\begin{equation}\label{claim1}
\int_{\Delta_{1,T}(\epsilon,\delta)}\mathbb{E}|I^{ij}_1(\epsilon,\delta,s,t)|^{p}\mathds{1}_{\{t-\delta\le s < t\}}dsdt \rightarrow 0
\end{equation}
as $(\epsilon,\delta)\downarrow 0$ and

\begin{equation}\label{claim2}
\sup_{(\epsilon,\delta) \in (0,1)^2}\int_{\Delta_{1,T}(\epsilon,\delta)} \mathbb{E}|I^{ij}_1(\epsilon,\delta,s,t)|^{p}\mathds{1}_{\{0\le s \le t-\delta\}}dsdt < \infty.
\end{equation}
First, we observe that Lemma \ref{growthELL} yields

\begin{equation}\label{int1}
I^{ij}_1(\epsilon,\delta,s,t)\lesssim  \beta^{ij}_2(\gamma,H,s,t)s^{-2H-1} \frac{1}{\delta}|\Delta \ell(\delta,s,t)|
\end{equation}
for $(s,t) \in \Delta_{1,T}(\epsilon,\delta)$. Then, 

$$\int_{\Delta_{1,T}(\epsilon,\delta)} \mathbb{E}|I^{ij}_1(\epsilon,\delta,s,t)|^{p}\mathds{1}_{\{0\le s \le t-\delta\}}dsdt$$
$$\lesssim \int_{\Delta_{1,T}(\epsilon,\delta)} |t-s|^{2\gamma H p } s^{-p}\mathds{1}_{\{0\le s \le t-\delta\}} \big|\frac{1}{\delta}\Delta \ell(\delta,s,t)\big|^{p}dsdt.$$

We now invoke (\ref{dest1}) in Lemma \ref{deltaLlemma}. We then have 

$$|t-s|^{2\gamma H p} s^{-p}\Big|\frac{1}{\delta} \Delta \ell(\delta,s,t) \Big|^{p}\lesssim |t-s|^{2\gamma p H} s^{(2H-1)p}t^{-p}$$

$$+ |t-s|^{2\gamma p H}\Bigg\{ \Big( 1 + \frac{\delta}{t}\Big)^{(2H-1)p}\frac{1}{|t(t-s)|^{p}} +  (t-s-\delta)^{(2H-1)p} (t-s)^{-2Hp}t^{-p}\Bigg\},$$
for $(s,t) \in \Delta_T$ such that $0 < s < t-\delta$. Clearly,

$$\int_0^T \int_0^t (t-s)^{2\gamma p H} s^{(2H-1)p}t^{-p}dsdr\lesssim \int_0^T t^{(2\gamma H + 2H-2)p+1}  < \infty ,$$
for $1 < p <  \frac{2}{2-2H-2\gamma H}$. Moreover,

$$\int_\delta^T \int_0^{t-\delta}(t-s-\delta)^{(2H-1)p}(t-s)^{2\gamma H p-2Hp}t^{-p}dsdt $$
$$\le \int_\delta^T \int_0^{t-\delta}(t-s-\delta)^{(2\gamma H-1)p}t^{-p}dsdt $$
$$ = \int_\delta^T(t-\delta)^{(2\gamma H-1)p+1}t^{-p}dt\lesssim \int_0^T z^{(2\gamma H-2)p +1}dz < \infty,$$
for $1 < p < \min\{\frac{1}{2-2H-2\gamma H}; \frac{1}{1-\gamma H}\}$. Similarly,

$$\int_0^T \int_0^t |t-s|^{2\gamma p H}\Big( 1 + \frac{\delta}{t}\Big)^{(2H-1)p}\frac{1}{|t(t-s)|^{p}}dsdt\lesssim \int_0^T t^{2\gamma p H-2p+1}dt < \infty.$$
This shows (\ref{claim2}). By (\ref{dest2}), we have


\begin{eqnarray*}
\int_\delta^T \int_{t-\delta}^t|t-s|^{2\gamma Hp}s^{-p} \Big|\frac{1}{\delta}\Delta \ell (\delta,s,t)\Big|^{p}dsdt&\lesssim& \frac{1}{\delta^{p}}\int_\delta^T \int_{t-\delta}^t (t-s)^{2\gamma H p} s^{(2H-1)p}dsdt\\
&\lesssim& \delta^{(2\gamma H + 2H-2)p +1} \rightarrow 0,
\end{eqnarray*}
as $\delta\downarrow0$. This concludes the proof. 
\end{proof}



\begin{lemma}\label{LemmaF3}
Let $g \in C^{\bar{\gamma},\gamma}([0,T]\times \mathbb{R}^d)$ with $\bar{\gamma}> \frac{1}{2}-H, \gamma > \frac{1}{2H}-1$ and $Y_\cdot=g(\cdot,B_\cdot)$. For each $1 < p < \min \Big\{\frac{1}{2-2H-2(\gamma H \wedge \bar{\gamma})  }; \frac{1}{1-(\gamma H\wedge \bar{\gamma}) - \frac{H}{2}}; \frac{1}{1-H}\Big\}$, there exists a constant $C = C(\bar{\gamma},\gamma,H,T,p)$ such that 

\begin{equation}\label{F3ex}
\mathbb{E}\int_{\Delta_{1,T}(\epsilon,\delta)}\Big| \frac{1}{4\epsilon\delta}Y^{(i)}_{s,t} Y^{(j)}_{s,t} \eta_{11}(\epsilon,s,t) \eta_{22}(\delta,s,t)\big[B^{(i)}_{s}B^{(j)}_{s,t}  - \varphi(s,t)\delta_{ij}\big]\Big|^pdsdt\le C \| g\|^{2p}_{\bar{\gamma},\gamma},
\end{equation}
for every $(\epsilon,\delta) \in (0,1)^2$ and $i,j=1,\ldots, d$.
\end{lemma}


\begin{proof}


Similar to the proof of Lemma \ref{LemmaF2}, we may suppose $g$ is time-homogeneous. Fix $i,j=1,\ldots, d$, $\gamma > \frac{1}{2H}-1$ and $1 < p < \min \Big\{\frac{1}{2-2H-2\gamma H}; \frac{1}{1-\gamma H - \frac{H}{2}}; \frac{1}{1-H}\Big\}$. In order to shorten notation, we define

$$\beta^{ij}_3(\gamma,H, s,t):= \big| \textbf{d}^{\bar{\gamma},\gamma}\big((s,B_s); (t,B_t) \big)\big|^2  \big| B^{(i)}_{s}B^{(j)}_{s,t} - \varphi(s,t)\delta_{ij}\big|,$$
for $(s,t) \in \Delta_T$. Throughout this proof, we repeatedly make use of the following elementary estimate

\begin{equation}\label{beta3esta}
\beta^{ij}_3(\gamma,H,s,t)\lesssim |B_{s,t}|^{2\gamma} |B^{(i)}_sB^{(j)}_{s,t}| + |B_{s,t}|^{2\gamma} |\varphi(s,t)|\delta_{ij},
\end{equation}
for $(s,t) \in [0,T]^2$. For each $1 < q < \infty$, we observe Lemma \ref{covfest} and H\"older's inequality yield


$$
\mathbb{E}|\beta^{ij}_3 (\gamma,H,s,t)|^q \lesssim |t-s|^{2\gamma Hq} \big[ \min \big\{s^H|t-s|^H; |t-s|^{2H}; s^{2H}  \big\}^q\delta_{ij} + s^{Hq}|t-s|^{H q}  \big],
$$
for $(s,t) \in [0,T]^2$. Therefore, for each $1 < q < \infty$, we have 

\begin{equation}\label{beta3estb}
\mathbb{E}|\beta^{ij}_3(\gamma,H,s,t)|^q\lesssim |t-s|^{2\gamma qH + qH}
\end{equation}
whenever $0 < t -s < 1$ and

\begin{equation}\label{beta3estbO}
\mathbb{E}|\beta^{ij}_3(\gamma,H,s,t)|^q\lesssim s^{qH},
\end{equation}
whenever the variable $s$ is close to the origin. In the sequel, in order to shorten notation, we write 

\begin{equation}\label{I1Lem3}
I^{ij}_1(\epsilon,\delta,s,t):=\beta^{ij}_3(\gamma,H, s,t) s^{2H}\Big|\ell \Big(\frac{\epsilon}{s}\Big)\Big||t-s|^{2H} \big|\Delta \ell (\delta,s,t)\varphi(s,t)\big|  \Theta^{-2}_{s,t}\frac{1}{\epsilon\delta}
\end{equation}

$$I^{ij}_2(\epsilon,\delta,s,t):=\beta^{ij}_3(\gamma,H, s,t) s^{4H}|t-s|^{4H} \Big| \ell \Big( \frac{\epsilon}{s}\Big)\ell \Big( \frac{\delta}{t-s}\Big)\Big|  \Theta^{-2}_{s,t}\frac{1}{\epsilon\delta}$$ 

$$I^{ij}_3(\epsilon,\delta,s,t):=\beta^{ij}_3(\gamma,H, s,t) |t-s|^{2H} \varphi^2(s,t)\big|\Delta \ell(\delta,s,t)\big| \big| \ell \Big( \frac{\epsilon}{t-s}\Big) \big|   \Theta^{-2}_{s,t}\frac{1}{\epsilon\delta}$$ 

\begin{equation}\label{I4Lem3} 
I^{ij}_4(\epsilon,\delta,s,t):=\beta^{ij}_3(\gamma,H, s,t) s^{2H}|t-s|^{4H} \Big| \ell \Big( \frac{\epsilon}{t-s}\Big)\ell \Big( \frac{\delta}{t-s}\Big)\Big| |\varphi(s,t)|    \Theta^{-2}_{s,t}\frac{1}{\epsilon\delta},
\end{equation}
for $(s,t) \in \Delta_{T}$. By using (\ref{etaexp}) for $\eta_{11}$ and $\eta_{22}$ and applying H\"older's inequality, we have 

$$\mathbb{E}\int_{\Delta_{1,T}(\epsilon,\delta)}\Big| \frac{1}{4\epsilon\delta}  Y^{(i)}_{s,t} Y^{(j)}_{s,t} \eta_{11}(\epsilon,s,t)\eta_{22}(\delta,s,t)[B^{(i)}_s B^{(j)}_{s,t}  - \varphi(s,t)\delta_{ij}]  \Big|^pdsdt$$

$$\lesssim \| g\|^{2p}_{\bar{\gamma},\gamma} \sum_{q=1}^4 \int_{\Delta_{1,T}(\epsilon,\delta)} \mathbb{E}|I^{ij}_q(\epsilon,\delta,s,t)|^{p} dsdt.$$
Lemmas \ref{fhlemma} and \ref{growthELL} yield

$$I^{ij}_4(\epsilon,\delta,s,t)\lesssim \beta^{ij}_3(\gamma,H, s,t)s^{-2H}|\varphi(s,t)| (t-s)^{-2},$$
for $(s,t) \in \Delta_{1,T}(\epsilon,\delta)$. By using Lemma \ref{covfest}, (\ref{beta3estb}) and (\ref{beta3estbO}), one can easily check

$$\int_{\Delta_T}\mathbb{E}|I^{ij}_4(\epsilon,\delta,s,t)|^{p}dsdt\lesssim \int_0^T \int_0^t \mathbb{E}|\beta^{ij}_3(\gamma,H, s,t)|^{p} s^{-2Hp}|\varphi(s,t)|^{p} (t-s)^{-2p}dsdt < \infty,
$$
for $1 < p < \frac{1}{2-2H-2\gamma H}$. Lemmas \ref{fhlemma} and \ref{growthELL} yield


$$I^{ij}_2(\epsilon,\delta,s,t)\lesssim \beta^{ij}_3(\gamma,H,s,t)s^{-1}|t-s|^{-1},$$
for $(s,t) \in \Delta_{1,T}(\epsilon,\delta)$. In view of (\ref{beta3esta}) and (\ref{beta3estb}), we then have

\begin{eqnarray*}
\int_{\Delta_T}\mathbb{E}|I^{ij}_2(\epsilon,\delta,s,t)|^{p}dsdt&\lesssim& \int_0^T \int_0^{\frac{t}{2}}  s^{p H-p}|t-s|^{-p}dsdt\\
&+& \int_0^T \int_{\frac{t}{2}}^t  |t-s|^{(2\gamma H + H-1)p}s^{-p}dsdt\\
&\lesssim& \int_0^T \int_0^{\frac{t}{2}}  s^{(H-1)p}t^{-p}dsdt\\
& +& \int_0^T \int_{\frac{t}{2}}^t  |t-s|^{[(2\gamma +1) H -1]p}t^{-p}dsdt. 
\end{eqnarray*}

Observe 
$$\int_0^T \int_0^{\frac{t}{2}}  s^{(H-1)p}t^{-p}dsdt\lesssim \int_0^T t^{(H-1)p+1-p}dt < \infty$$ 
for $1 < p < \frac{1}{1-H}$. For the second integral, we observe 

$$\int_0^T \int_{\frac{t}{2}}^t  |t-s|^{[(2\gamma +1) H -1]p}t^{-p}dsdt\lesssim \int_0^T t^{[(2\gamma+1)H-2]p+1}dt< \infty,$$
for $1 < p < \min \big \{  \frac{1}{2-2H-2H\gamma}; \frac{2}{2- (2\gamma+1)H}  \big\}.$

Next, we evaluate the terms $I^{ij}_1(\epsilon,\delta,s,t)$ and $I^{ij}_3(\epsilon,\delta,s,t)$. Lemmas \ref{growthELL}, \ref{fhlemma} and (\ref{brest}) yield

\begin{eqnarray*}
I^{ij}_1(\epsilon,\delta,s,t)&\lesssim& \beta^{ij}_3(\gamma,H,s,t) s^{-2H-1}|t-s|^{-2H}\frac{1}{\delta}|\Delta \ell(\delta,s,t)| |\varphi(s,t)|\\
&\lesssim & \beta^{ij}_3(\gamma,H,s,t) s^{-2H-1}|t-s|^{-2H} t^{2H-1}|\varphi(s,t)|\\
&+& \beta^{ij}_3(\gamma,H,s,t) s^{-2H-1}|t-s|^{-1}|\varphi(s,t)|,
\end{eqnarray*}
for $(s,t) \in \Delta_{1,T}(\epsilon,\delta)$. We will evaluate

$$\int_0^T \int_0^t \mathbb{E}|\beta^{ij}_3(\epsilon,\delta,s,t)|^{p}s^{-(2H+1)p} |t-s|^{-2Hp} t^{(2H-1)p} |\varphi(s,t)|^{p}dsdt$$
$$+ \int_0^T \int_0^t \mathbb{E}|\beta^{ij}_3(\epsilon,\delta,s,t)|^{p}s^{-(2H+1)p} |t-s|^{-p} |\varphi(s,t)|^{p}dsdt$$
$$=: J^{ij}_{11}(\epsilon,\delta,p) + J^{ij}_{12}(\epsilon,\delta,p).$$


It is sufficient to estimate $J^{ij}_{12}$. In view of (\ref{beta3estb}) and (\ref{beta3estbO}) and applying Lemma \ref{covfest}, we have

$${J}^{ij}_{12}(\epsilon,\delta,p)\lesssim \int_0^T \int_{0}^{\frac{t}{2}}s^{Hp -(2H+1)p}|t-s|^{-p} s^{2Hp}dsdt$$
$$+ \int_0^T \int_{\frac{t}{2}}^t |t-s|^{(2\gamma H+H-1)p} s^{-(2H+1)p}s^{Hp} |t-s|^{Hp}dsdt.$$
Clearly,

$$\int_0^T \int_{0}^{\frac{t}{2}}s^{Hp -(2H+1)p}|t-s|^{-p} s^{2Hp}dsdt\lesssim\int_0^T \int_0^{\frac{t}{2}} s^{(H-1)p}dst^{-p}dt < \infty,$$
for $1 <  p < \frac{1}{1-H}$. Moreover,

$$\int_0^T \int_{\frac{t}{2}}^t |t-s|^{(2\gamma H+H-1)p} s^{-(2H+1)p}s^{Hp} |t-s|^{Hp}dsdt$$
$$
\lesssim \int_0^T t^{-Hp -p} \int_0^{\frac{t}{2}} z^{(2\gamma H + 2H-1)p}dz dt \lesssim \int_0^T t^{(2\gamma H + H-2)p+1}dt < \infty,$$
for $1 <  p < \frac{2}{2- (2\gamma+1) H}$. This concludes the proof.

\end{proof}

\begin{lemma}\label{LemmaF4}
Let $g \in C^{\bar{\gamma},\gamma}([0,T]\times \mathbb{R}^d)$ with $\bar{\gamma}> \frac{1}{2}-H, \gamma > \frac{1}{2H}-1$ and $Y_\cdot=g(\cdot,B_\cdot)$. For each $1 < p < \min \Big\{\frac{1}{2-2H-2(\bar{\gamma}\wedge \gamma H)}; \frac{1}{1- \frac{H}{2}}\Big\}$, there exists a constant $C = C(\bar{\gamma},\gamma,H,T,p)$ such that 

\begin{equation}\label{F4ex}
\mathbb{E}\int_{\Delta_{1,T}(\epsilon,\delta)}\Big| \frac{1}{4\epsilon\delta}Y^{(i)}_{s,t} Y^{(j)}_{s,t} \eta_{12}(\epsilon,s,t) \eta_{21}(\delta,s,t)\big[(B^{(j)}_{s} B^{(i)}_{s,t})  - \varphi(s,t)\delta_{ij}\big]\Big|^pdsdt\le C \| g\|^{2p}_{\bar{\gamma},\gamma},
\end{equation}
for every $(\epsilon,\delta) \in (0,1)^2$ and $i,j=1,\ldots, d$.
\end{lemma}

\begin{proof}
Similar to the proof of Lemma \ref{LemmaF3}, we may suppose $g$ is time-homogeneous. Fix $i,j=1,\ldots, d$, $\gamma > \frac{1}{2H}-1$ and $1 < p < \min \Big\{\frac{1}{2-2H-2\gamma H}; \frac{1}{1- \frac{H}{2}}\Big\}$. In order to shorten notation, we define $\beta^{ji}_3(\gamma,H, s,t)$ as defined in the proof of Lemma \ref{LemmaF3}. In the sequel, in order to shorten notation, we write



$$I^{ij}_2(\epsilon,\delta,s,t):=\beta^{ji}_3(\gamma,H, s,t) s^{2H}|t-s|^{2H}\varphi^2(s,t)\Big| \ell \Big( \frac{\epsilon}{s}\Big)\ell \Big( \frac{\delta}{t-s}\Big)\Big| \Theta^{-2}_{s,t}\frac{1}{\epsilon\delta}$$ 

$$I^{ij}_3(\epsilon,\delta,s,t):=\beta^{ji}_3(\gamma,H, s,t) |t-s|^{4H} s^{2H}  |\Delta \ell(\delta,s,t)| \Big| \ell \Big( \frac{\epsilon}{t-s}\Big)\Big|    \Theta^{-2}_{s,t}\frac{1}{\epsilon\delta},$$ 
for $(s,t) \in \Delta_{T}$. We also recall $I^{ij}_1(\epsilon)$ and $I^{ij}_4(\epsilon,\delta)$ as described in (\ref{I1Lem3}) and (\ref{I4Lem3}), respectively. By using (\ref{etaexp}) for $\eta_{12}$ and $\eta_{21}$ and by applying H\"older's inequality, we have 

$$\mathbb{E}\int_{\Delta_{1,T}(\epsilon,\delta)}\Big| \frac{1}{4\epsilon\delta}  Y^{(i)}_{s,t} Y^{(j)}_{s,t} \eta_{12}(\epsilon,s,t)\eta_{21}(\delta,s,t)[B^{(j)}_s B^{(i)}_{s,t}  - \varphi(s,t)\delta_{ij}]  \Big|^pdsdt$$

$$\lesssim \| g\|^{2p}_{\bar{\gamma},\gamma} \sum_{q=1}^4 \int_{\Delta_{1,T}(\epsilon,\delta)} \mathbb{E}|I^{ij}_q(\epsilon,\delta,s,t)|^{p} dsdt.$$
The terms $I^{ij}_1(\epsilon,\delta,s,t)$ and $I^{ij}_4(\epsilon,\delta,s,t)$ are investigated in Lemma \ref{LemmaF3}. We only need to check $I^{ij}_2(\epsilon,\delta,s,t)$ and $I^{ij}_3(\epsilon,\delta,s,t)$. By using Lemmas \ref{fhlemma} and \ref{growthELL} and (\ref{brest}), we have

\begin{eqnarray*}
I^{ij}_3(\epsilon,\delta,s,t)&\lesssim& \beta^{ji}_3(\gamma,H,s,t) |t-s|^{-1}s^{-2H}\frac{1}{\delta}|\Delta \ell(\delta,s,t)|\\
&\lesssim& \beta^{ji}_3(\gamma,H,s,t) \Big[|t-s|^{-1}s^{-2H}t^{2H-1}+  s^{-2H}(t-s)^{2H-2}\Big],
\end{eqnarray*}
for $(s,t) \in \Delta_{1,T}(\epsilon,\delta)$. Then, (\ref{beta3estb}) and (\ref{beta3estbO}) yield

\begin{eqnarray}\label{fi1}
\int_{\Delta_T}\mathbb{E}|I^{ij}_3(\epsilon,\delta,s,t)|^{p}dsdt &\lesssim & \int_0^T \int_0^{\frac{t}{2}} s^{-Hp}(t-s)^{(2H-2)p}dsdt\\
\nonumber&+& \int_0^T \int_{\frac{t}{2}}^t (t-s)^{(2\gamma H + 3H-2)p} s^{-2H p}dsdt\\
\nonumber&+& \int_0^T \int_0^{\frac{t}{2}} s^{-Hp}(t-s)^{-p}dst^{(2H-1)p}dt\\
\nonumber&+& \int_0^T \int_{\frac{t}{2}}^t (t-s)^{(2\gamma H + H-1)p} s^{-2H p}dst^{(2H-1)p}dt. 
\end{eqnarray}
It is enough to estimate the first two integrals in the right-hand side of (\ref{fi1}). We have 
$$\int_0^T \int_0^{\frac{t}{2}} s^{-Hp}(t-s)^{(2H-2)p}dsdt\lesssim \int_0^T t^{(H-2)p+1}dt < \infty,$$
for $1 < p < \frac{1}{1-\frac{H}{2}}$. In order the second integral in the right-hand side of (\ref{fi1}) be finite, the following condition is necessary
\begin{equation}\label{fi2}
(2-3H - 2\gamma H)p < 1.
\end{equation}
If $2-3H - 2\gamma H <0$, any $p>1$ realizes (\ref{fi2}). In case $2-3H - 2\gamma H >0$, then $(2-3H - 2\gamma H)p < 1 \Longleftrightarrow 1 < p < \frac{1}{2-3H-2\gamma H}$ which is fulfilled for $1 < p < \min \Big\{\frac{1}{2-2H-2\gamma H}; \frac{1}{1- \frac{H}{2}}\Big\}$. Then,

$$\int_0^T \int_{\frac{t}{2}}^t (t-s)^{(2\gamma H + 3H-2)p} s^{-2H p}dsdt\lesssim \int_0^T t^{-2H p } t^{(2\gamma H + 3H-2)p +1}dt < \infty,$$
for $1 < p < \frac{1}{1-\frac{H}{2}-\gamma H}$. The analysis of $I^{ij}_2(\epsilon,\delta,s,t)$ follows the same lines. This concludes the proof.
\end{proof}

Lemmas \ref{LemmaF1}, \ref{LemmaF2}, \ref{LemmaF3} and \ref{LemmaF4} allow us to conclude uniformly integrability of the family of processes
\begin{equation}\label{proUI}
Y^{(i)}_{s,t} Y^{(j)}_{s,t}  \Big\{Z^{(1)}_{s,t}(\epsilon,\delta)Z^{(2)}_{s,t}(\delta) - \mathbb{E}\big[ Z^{(1)}_{s,t}(\epsilon,\delta)Z^{(2)}_{s,t}(\delta) \big]\Big\}\mathds{1}_{\Delta_{r,T}(\epsilon,\delta)}(s,t)
\end{equation}
in $L^1(\Omega\times \Delta_T)$ parameterized by $(\epsilon,\delta) \in (0,1)^2$ for $r=1$. It remains to analyse 

\begin{equation}\label{boundary}
\mathbb{E}\int_{\Delta_{r,T}(\epsilon,\delta)} Y^{(i)}_{s,t} Y^{(j)}_{s,t} \mathcal{W}^{0,ij}(\epsilon,\delta;s,t)dsdt
\end{equation}
over the boundary sets $\Delta_{r,T}(\epsilon,\delta)$ for $r=3,4$. In view of Lemmas \ref{LemmaF1}, \ref{LemmaF2}, \ref{LemmaF3} and \ref{LemmaF4}, it is quite intuitive that (\ref{boundary}) vanishes as $(\epsilon,\delta)\downarrow 0$. 

\begin{lemma}\label{LemmaF5}
Let $g \in C^{\bar{\gamma},\gamma}([0,T]\times \mathbb{R}^d)$ with $\bar{\gamma}> \frac{1}{2}-H, \gamma > \frac{1}{2H}-1$ and $Y_\cdot=g(\cdot, B_\cdot)$. Then, 
$$
\mathbb{E}\int_{\Delta_{r,T}(\epsilon,\delta)}| Y^{(i)}_{s,t} Y^{(j)}_{s,t} \mathcal{W}^{0,ij}(\epsilon,\delta;s,t)|dsdt\rightarrow 0,
$$
as $(\epsilon,\delta)\downarrow 0$, for $r=3,4$.  
\end{lemma}
\begin{proof}
By using (\ref{prodZ}) and the representations of $\eta_{11}(\epsilon,\delta), \eta_{12}(\epsilon,\delta), \eta_{21}(\epsilon,\delta)$ and $\eta_{22}(\epsilon,\delta)$ over the regions $\Delta_{r,T}(\epsilon,\delta)$, the proof is almost identical to Lemmas \ref{LemmaF1}, \ref{LemmaF2}, \ref{LemmaF3} and \ref{LemmaF4}. For instance, now $\Delta_{3,T}(\epsilon,\delta)$ is negligible as $(\epsilon,\delta)\downarrow 0$ and we should play with 
$\epsilon^{-1}R(s,s+\epsilon)$ and $\epsilon^{-1} [R(t,s+\epsilon) - R(s,s+\epsilon)]$ rather than $\frac{s^{2H}}{2}\ell \big(\epsilon s^{-1}\big)$ and $\frac{|t-s|^{2H}}{2}\ell \big( \epsilon(t-s)^{-1}\big)$ with the same arguments (with the obvious modification) for the region $\Delta_{4,T}(\epsilon,\delta)$. Therefore, we omit the details of the proof. 
\end{proof}

By Remark \ref{remsimplex}, Lemmas \ref{Lambdalimit}, \ref{LemmaF1}, \ref{LemmaF2}, \ref{LemmaF3}, \ref{LemmaF4} and \ref{LemmaF5}, we arrive at the following result. 

\begin{proposition}\label{proprandom}
If $g \in C^{\bar{\gamma},\gamma}([0,T]\times \mathbb{R}^d)$ with $\bar{\gamma}> \frac{1}{2}-H, \gamma > \frac{1}{2H}-1$ and $Y_\cdot = g(\cdot, B_\cdot)$, then 
$$\lim_{(\epsilon,\delta)\downarrow 0}-\frac{1}{2}\sum_{i,j=1}^d \int_{[0,T]^2_\star } \mathbb{E}\Big[Y^{(i)}_{s,t} Y^{(j)}_{s,t} \mathcal{W}^{0,ij}(\epsilon,\delta;s,t)\Big]dsdt = $$
$$ = -\frac{1}{2}\sum_{i,j=1}^d \int_{[0,T]^2_\star} \mathbb{E}\Big[Y^{(i)}_{s,t} Y^{(j)}_{s,t} \mathcal{W}^{ij}(s,t)\Big]dsdt.$$
\end{proposition}
By combining Proposition \ref{proprandom}, Lemmas \ref{detlemma} and \ref{marglemma}, we conclude the proof of Theorem \ref{mainTH1}.  

\section{The isometry and proof of Theorem \ref{mainTH2}}\label{Isosection}
This section presents the underlying Hilbert space which supports the second moment of the symmetric-Stratonovich. For this purpose, we recall that FBM admits the RKHS $L_R(\mathbb{R}^d)$ described as follows. From now on, any function $f$ defined on $[0,T]$ will be extended to the real line $\mathbb{R}$ as $f(t)=f(0)$ for $t\le 0$ and $f(t) = f(T)$ for $t\ge T$. Let $C^1_c(\mathbb{R}_+,\mathbb{R}^d)$ be the space of $\mathbb{R}^d$-valued $C^1$-functions with compact support in $\mathbb{R}_+$. Let us recall 

$$\mu(dsdt)= H(2H-1)|t-s|^{2H-2}dsdt$$
is a $\sigma$-finite measure on $[0,T]^2$. Let $I:C^1_c(\mathbb{R}_+,\mathbb{R}^d)\rightarrow L^2(\mathbb{P})$ be the linear mapping defined by

$$I(f):= \int_0^\infty f_s dB_s:= \langle f(+\infty), B_\infty\rangle - \int_0^{+\infty} \langle B_s, d f(s)\rangle, $$
where $\langle f(+\infty), B_\infty\rangle:= \lim_{t\rightarrow +\infty}\langle f(t), B_T\rangle=0$. Let $\tilde{L}_R(\mathbb{R}^d)$ be the linear space of all Borel functions $f:\mathbb{R}_+\rightarrow \mathbb{R}^d$ such that

\begin{description}
  \item[i] $\int_0^\infty |f|^2(s)|R|(ds,T) < \infty$,
  \item[ii] $\int_{\mathbb{R}^2_+} |f(s_1) - f(s_2)|^2|\mu|(ds_1 ds_2) < \infty$.
\end{description}
For $f\in \tilde{L}_R(\mathbb{R}^d)$, we define

\begin{equation}\label{Vinner}
\|f\|^2_{L_R(\mathbb{R}^d)}:=\int_0^\infty |f(s)|^2R(ds,T) - \frac{1}{2}\int_{\mathbb{R}^2_+} |f(s_1) - f(s_2)|^2 \mu(ds_1 ds_2).
\end{equation}
It is possible to show $\tilde{L}_R(\mathbb{R}^d)$ is a Hilbert space w.r.t.~the inner-product associated with (\ref{Vinner}) and

$$
\mathbb{E}|I(f)|^2 = \|f\|^2_{L_R(\mathbb{R}^d)},
$$
for every $f\in C^1_0(\mathbb{R}_+,\mathbb{R}^d)$. Let $L_R(\mathbb{R}^d)$ be the closure of $C^1_0(\mathbb{R}_+,\mathbb{R}^d)$ w.r.t.~$\|\cdot\|_{L_R(\mathbb{R}^d)}$ as a subset of $\tilde{L}_R(\mathbb{R}^d)$. Then, $I:C^1_0(\mathbb{R}_+,\mathbb{R}^d)\rightarrow L^2(\mathbb{P})$ can be uniquely extended to a linear isometry

\begin{equation}\label{paleyop}
I:L_R(\mathbb{R}^d)\rightarrow L^2(\mathbb{P}).
\end{equation}

One can check $L_R(\mathbb{R}^d)$ is a real separable Hilbert space and

\begin{equation}\label{BVcase}
\int_0^\infty \varphi dB  = -\int_0^\infty \langle B, d\varphi\rangle,
\end{equation}
for every bounded variation function $\varphi$ with compact support. This implies that

\begin{equation}\label{cov1}
R(s,t) = \langle \mathds{1}_{[0,t]}, \mathds{1}_{[0,s]}\rangle_{L_R}; s,t \in [0,T].
\end{equation}

See Propositions 6.18, 6.14, 6.22, 6.32 and 6.33 in \cite{krukrusso} for the proof of these results. The Paley - Wiener integral associated with $B$ is given by

$$I(f):= \int_0^\infty fdB; f \in L_R(\mathbb{R}^d).$$
At this point, it is important to recall the following result. Let $V:\mathbb{R}_+\rightarrow \mathbb{R}_+$ be an increasing function such that 

$$\int_{[0,T]^2}V^2(t-s)|\mu|(dsdt) < \infty.$$
Then, every Borel function $f:[0,T]\rightarrow \mathbb{R}^d$ such that 
$$|f_{s,t}|\le V(|t-s|),$$
for $(s,t) \in [0,T]^2_\star$ will be an element of $L_R(\mathbb{R}^d)$. See Proposition 6.9 in \cite{kruk2007} for the proof of this result.  

Next, we give an important example of this article. By using Prop 6.9 in \cite{kruk2007}, one can easily check the statement of the following example.  
\begin{example}\label{RKHSex}
Let $f:[0,T]\rightarrow \mathbb{R}^d$ be a continuous function with modulus of continuity of the form $V(t) = t^{\frac{1}{2}-H}\psi(t)$, where $\psi$ satisfies 

\begin{itemize}
  \item $\lim_{t\downarrow 0}\psi(t)=\psi(0)=0$
  \item $\lim_{t\downarrow 0}\frac{\psi(t)}{t^\alpha}=+\infty$ for every $\alpha \in (0,1]$. 
  \item There exists $\eta\in (0,1)$ such that $\int_0^\eta \frac{1}{t} \psi^2 (t) dt < \infty$.
\end{itemize}
Then, $f \in L_R(\mathbb{R}^d)$. For a concrete example, take $\alpha> \frac{1}{2}$, $\psi(t) = (\log(t^{-1}))^{-\alpha}$ for $0 < t < 1$ and $\psi(0)=0$.  
\end{example}
\begin{remark}\label{RKHSex1}
The Example \ref{RKHSex} shows that the RKHS $L_R(\mathbb{R}^d)$ contains a large class of functions beyond the H\"older scale of modulus of continuity. Of course, if $\psi$ in Example \ref{RKHSex} is of the form $\psi(t) =  t^\beta$ for some $\beta \in (0,1)$, then $f$ in Example \ref{RKHSex} is $\alpha$-H\"older continuous for $\alpha > \frac{1}{2}-H$. In other words, $C^{\alpha} \subset L_R (\mathbb{R}^ d)$ for every $\alpha > \frac{1}{2}-H$ and this inclusion fails for $\alpha=\frac{1}{2}-H$. Moreover, $L_R(\mathbb{R}^d)$ contains continuous functions with a more diverse modulus of continuity as illustrated in Section \ref{examplesection}.  
\end{remark}

With Theorem \ref{mainTH1} at hand, we now add randomness for the $L^2(\mathbb{P})$-norm of the symmetric-Stratonovich integral. Fix $\frac{1}{4} < H < \frac{1}{2}$. Let $\mathcal{E}^0$ be the set of all processes of the form $g(t,B_t)$ where $g:[0,T]\times \mathbb{R}^d\rightarrow \mathbb{R}^d \in C^{\bar{\gamma},\gamma}([0,T]\times \mathbb{R}^d)$ for some $\bar{\gamma}> \frac{1}{2}-H$ and $\gamma > \frac{1}{2H}-1$. By Theorem \ref{mainTH1}, we know that

$$\mathbb{E}\Big[ \int_0^T Y_td^{0}B_t \int_0^T Z_t d^0B_t\Big] = \mathbb{E}\big\langle Y, Z \big\rangle_{L_R(\mathbb{R}^d)} -\frac{1}{2}\mathbb{E}\int_{[0,T]^2} \big\langle Y_{s,t}\otimes Z_{s,t}, \mathcal{W}(s,t)\big\rangle dsdt
$$
\begin{equation}\label{iso}
+\mathbb{E}\int_0^T \big \langle Y_t\otimes Z_t, \mathcal{M}_t\big\rangle dt,
\end{equation}
for every $Y,Z \in \mathcal{E}^0$.


We can write (\ref{iso}) in the more compact form as a function of $(\Psi,\Lambda)$:  
\begin{eqnarray}\label{iso2}
\nonumber\mathbb{E}\Bigg[ \int_0^T Y_td^0B_t \int_0^T Z_t d^0B_t\Bigg]  &=& \mathbb{E}\int_{[0,T]} \big \langle Y_t \otimes Z_t, \Psi_t\big\rangle dt\\
&-& \frac{1}{2}\mathbb{E}\int_{[0,T]^2} \big \langle Y_{s,t}\otimes Z_{s,t}, \Lambda(s,t)\big\rangle dsdt,
\end{eqnarray}
for every $Y,Z \in \mathcal{E}^0$. Here, we recall 
$$
\Lambda(s,t)=\mathcal{W}(s,t) + \frac{\partial^2 R}{\partial t \partial s}(s,t)I_{d\times d }
$$
and

$$
\Psi_t= \mathcal{M}_t + \frac{\partial R}{\partial t}(T,t)I_{d\times d},
$$
where the matrix-valued processes $\mathcal{M}$ and $\mathcal{W}$ are given by (\ref{Mproc}) and (\ref{Wproc}), respectively. Let us define

\begin{eqnarray}\label{sesq}
\langle Y, Z \rangle_{\mathcal{E}^0}&:=&\mathbb{E}\int_{[0,T]} \big \langle Y_t \otimes Z_t, \Psi_t\big\rangle dt
\\
\nonumber &-& \frac{1}{2}\mathbb{E}\int_{[0,T]^2} \big \langle Y_{s,t}\otimes Z_{s,t}, \Lambda(s,t)\big\rangle dsdt
\end{eqnarray}
for $Y,Z \in \mathcal{E}^0$. The bracket in (\ref{sesq}) is a semi-inner product on the vector space $\mathcal{E}^0$. Indeed, the positiveness and bilinearity follow from $\langle Y, Y\rangle_{\mathcal{E}^0} = \|\int_0^T Y_t d^0 B_t\|_2^2$ and the linearity of the Lebesgue integral, respectively. The symmetry is immediate from the definition. We then set

$$
\|Y\|^2_{\mathcal{E}^0}:= \langle Y, Y \rangle_{\mathcal{E}^0}.
$$

We will identify $Y$ and $Z$ to be in the same equivalence class whenever 
$$\langle Y-Z,Y-Z\rangle_{\mathcal{E}^0} = 0$$ 
and quotient $\mathcal{E}^0$ w.r.t.~these classes. With a slight abuse of notation, we write $\mathcal{E}^0$ as the quotient space equipped with the inner product (\ref{sesq}). 

Let $\mathcal{L}_R(\mathbb{R}^d)$ be the (Hilbert space) completion of $\big(\mathcal{E}^0, \langle \cdot, \cdot \rangle_{\mathcal{E}^0}\big)$. Then, we can extend the domain of the symmetric-Stratonovich stochastic integral from $\mathcal{E}^0$ to $\mathcal{L}_R(\mathbb{R}^d)$ via (\ref{sesq}).

\begin{corollary}
Fix $\frac{1}{4}< H < \frac{1}{2}$. If $Y \in \mathcal{L}_R(\mathbb{R}^d)$, then symmetric-Stratonovich stochastic integral $\int_0^T Y_t\hat{d}^0B_t$ is defined as

$$\int_0^T Y_s\hat{d}^0B_s:=\lim_{n\rightarrow +\infty} \int_0^T Y^n_sd^0B_s~\quad \text{in}~L^2(\mathbb{P}),$$
for every sequence $\{Y_n; n\ge 1\} \subset \mathcal{E}^0$ such that $Y_n\rightarrow Y$ in $\mathcal{L}_R(\mathbb{R}^d)$. Moreover,


$$\| Y\|^2_{\mathcal{L}_R(\mathbb{R}^d)} = \mathbb{E}\Bigg|\int_0^T Y_t\hat{d}^0B_t\Bigg|^2,$$
for every $Y \in \mathcal{L}_R(\mathbb{R}^d)$.
\end{corollary}
Next, we are going to exploit the isometry property and we aim to provide a large class of processes in $\mathcal{L}_R(\mathbb{R}^d)$ under slightly stronger conditions than (\ref{constRKHS}) which determines the regularity of the functions in $L^2(\Omega; L_R(\mathbb{R}^d))$. By the very definition of (\ref{etadef1}) and (\ref{etadef2}), we have 

$$\eta_{11}(s,t) =   \Big\{ s^{2H-1}|t-s|^{2H} - \varphi(s,t) |t-s|^{2H-1}\Big\}\frac{H}{\Theta_{s,t}}$$
$$\eta_{21}(s,t) =  \Big\{ \Big[t^{2H-1} - |t-s|^{2H-1}\Big]|t-s|^{2H} - \varphi(s,t)|t-s|^{2H-1}\Big\}\frac{H}{\Theta_{s,t}}$$
$$\eta_{12}(s,t) =  \Big\{ s^{2H}|t-s|^{2H-1} - s^{2H-1}\varphi(s,t)  \Big\}\frac{H}{\Theta_{s,t}}$$
$$\eta_{22}(s,t) =  \Big\{ s^{2H} |t-s|^{2H-1} - \varphi(s,t)\Big[t^{2H-1} - |t-s|^{2H-1}\Big]   \Big\}\frac{H}{\Theta_{s,t}}.$$

For $i\neq j$, let $\mathbb{A}_{s,t}(i,j) = (a_{mn}(s,t))$ be the $4\times 4$-upper triangular matrix given by $a_{13}(s,t) = \eta_{11}\eta_{21}(s,t)$, $a_{14}(s,t) = \eta_{11}\eta_{22}(s,t)$, $a_{23}(s,t) = \eta_{12}\eta_{21}(s,t)$, $a_{24}(s,t) = \eta_{12}\eta_{22}(s,t)$ and $a_{mn}(s,t)=0$ otherwise. For $i\neq j$, we recall the Gaussian vector $(B^{(i)}_s,  B^{(i)}_{s,t},  B^{(j)}_s ,  B^{(j)}_{s,t})$ has the covariance matrix

$$\overline{\Sigma}_{s,t}(i,j)=\left(
  \begin{array}{cc}
    \Sigma_{s,t}(i,i) & 0 \\
    0 & \Sigma_{s,t}(j,j) \\
  \end{array}
\right).
$$
In case $i=j$, we just set $\overline{\Sigma}_{s,t} = \Sigma_{s,t}(i,i)$ and $\mathbb{A}_{s,t}(i,i) = \left(
                                                                                             \begin{array}{cc}
                                                                                              \eta_{11}\eta_{21}(s,t)  & \eta_{11}\eta_{22}(s,t) \\
                                                                                               \eta_{12}\eta_{21}(s,t) & \eta_{12}\eta_{22}(s,t) \\
                                                                                             \end{array}
                                                                                           \right)
$
for $(s,t) \in [0,T]^2_\star$. Finally, for $i\neq j$, we set $\mathbb{W}_{s,t}(i,j):= \text{sym}~(\mathbb{A}_{s,t}(i,j))$ and $\mathbb{W}_{s,t}:= \mathbb{W}_{s,t}(i,i)$, whenever $i=j$.

By construction, we can express $\mathcal{W}^{ij}(s,t)$ in terms of the quadratic form associated with the matrix $\mathbb{W}_{s,t}(i,j)$ and applied to the Gaussian vector $(B^{(i)}_s,  B^{(i)}_{s,t},  B^{(j)}_s ,  B^{(j)}_{s,t})$ equipped with the covariance matrix $\overline{\Sigma}_{s,t}(i,j)$ for $(s,t) \in \Delta_T$.

\begin{lemma}\label{sharpW}
The following formulas hold 
\begin{equation}\label{isoM}
\| \mathcal{M}_t\|_2=
c(d)\frac{R(t,T)}{\gamma(t)} \frac{d\gamma}{dt}(t),
\end{equation}
for $0 < t\le T$ and 

\begin{equation}\label{isoW}
\| \mathcal{W}(s,t)\|_2 = \sqrt{2\sum_{i,j=1}^d \text{Tr} \big[ (\mathbb{W}_{s,t}(i,j) \overline{\Sigma}_{s,t}(i,j))^2 \big]               },
\end{equation}
for $(s, t) \in \Delta_T$, where $c(d)$ is a constant which depend on $d$.

\end{lemma}
\begin{proof}
Fix $(s,t) \in \Delta_T$. Let $Q_{s,t}(i,j)$ be the quadratic form associated with the matrix $\mathbb{W}_{s,t}(i,j)$. We know that  

\begin{eqnarray*}
\mathbb{E}|\mathcal{W}^{ij}(s,t)|^2&=&\text{Var} (Q_{s,t} (B^{(i)}_s,  B^{(i)}_{s,t},  B^{(j)}_s ,  B^{(j)}_{s,t}) = 2 \text{Tr} \big[ (\mathbb{W}_{s,t}(i,j) \overline{\Sigma}_{s,t}(i,j))^2 \big].
\end{eqnarray*}
Therefore, (\ref{isoW}) holds true. To keep notation simple, we set $f(t) = \frac{R(t,T)}{t^{4H}}\frac{d\gamma}{dt}(t)$ and we write 

$$\mathcal{M}^{ij}_t = f(t)B^{(i)}_t B^{(j)}_t - \mathbb{E}\big[ f(t)B^{(i)}_t B^{(j)}_t\big],$$
for $t \in (0,T]$. Then, 

\begin{eqnarray*}
\mathbb{E}|\mathcal{M}^{ij}_t|^2 &=& \text{Var}\big( f(t)B^{(i)}_t B^{(j)}_t  \big) = \mathbb{E}\big|f(t) B^{(i)}_t B^{(j)}_t\big|^2 -   \big(\mathbb{E}\big[ f(t)B^{(i)}_t B^{(j)}_t\big]\big)^2\\
& =& f^2 (t)t^{4H},
\end{eqnarray*}
for $i\neq j$. For $i=j$, $\mathbb{E}|\mathcal{M}^{ij}_t|^2 = f^2(t) \text{Var}\big(|B^{(i)}_t|^2 \big)= 2f^2(t) t^{4H}$. Therefore, 
$$
\| \mathcal{M}^{ij}_t\|_2=\left\{
\begin{array}{rl}
\frac{R(t,T)}{\gamma(t)} \frac{d\gamma}{dt}(t); & \hbox{if} \ i\neq j \\
\frac{2R(t,T)}{\gamma(t)} \frac{d\gamma}{dt}(t);& \hbox{if} \ i=j. \\
\end{array}
\right.
$$
Then, we can find a constant $c(d)$ which depends on $d$ such that (\ref{isoM}) holds true.
\end{proof}



Next, we are going to show that the leading term for $\| \mathcal{W}(s,t)\|_{q}$ is actually $ \big| \frac{\partial^2 R}{\partial t\partial s}(s,t) \big|$.  Let us denote

\begin{equation}\label{kappaexp}
\kappa(s,t):= (s\wedge t)^{H-1}|t-s|^{H-1} + \frac{|t-s|^{2H-1}}{s\vee t};
\end{equation}
for $(s,t) \in [0,T]^2_\star$ and $s\wedge t >0$. One can easily check that $\kappa \in L^{q}([0,T]^2)$ for $1\le q < \frac{1}{1-H}$.


\begin{lemma}\label{West}
For $0 < H < \frac{1}{2}$ and $q>1$, there exists a constant $C = C(q,H,T)$ such that  
$$\| \mathcal{W}(s,t)\|_{q}\le C\Big| \frac{\partial^2 R}{\partial t\partial s}(s,t) \Big|+ \kappa(s,t),$$
for every $(s,t) \in [0,T]^2_\star$. 
\end{lemma}
\begin{proof}
By symmetry (see (\ref{INTREP1}) and (\ref{INTREP2})), it is sufficient to check over the simplex $\Delta_T$. From (\ref{INTREP1}), we can write

\begin{eqnarray*}
\mathcal{W}^{ij}(s,t)&=& \eta_{11}(s,t)\eta_{21}(s,t)\big(B^{(i)}_sB^{(j)}_s-s^{2H}\delta_{ij}\big) \\
\nonumber&+& \eta_{11}(s,t)\eta_{22}(s,t) \big( B^{(i)}_s B^{(j)}_{s,t} - \varphi(s,t)\delta_{ij} \big)\\
\nonumber&+& \eta_{12}(s,t) \eta_{21}(s,t) \big( B^{(j)}_s B^{(i)}_{s,t} - \varphi(s,t)\delta_{ij} \big)\\
\nonumber&+& \eta_{12}(s,t)\eta_{22}(s,t)\big(B^{(i)}_{s,t} B^{(j)}_{s,t} - |t-s|^{2H}\delta_{ij} \big),
\end{eqnarray*}
for $1\le i,j\le d$. By H\"older's inequality and using the Gaussian property of $B$, it is sufficient to estimate 
$$
|\eta_{11}(s,t)\eta_{21}(s,t)|s^{2H} + |\eta_{12}(s,t)\eta_{22}(s,t)||t-s|^{2H}$$ 
$$+ \big|\eta_{11}(s,t) \eta_{22}(s,t) + \eta_{12}(s,t) \eta_{21}(s,t) \big| s^{H}|t-s|^H$$
for $(s,t) \in \Delta_T$. 

We observe

\begin{equation}\label{parR}
\Big| \frac{\partial R}{\partial t}(s,t)\Big| = H \big[|t-s|^{2H-1}-t^{2H-1} \big] < |t-s|^{2H-1}\frac{s}{t}, 
\end{equation}
for $(s,t) \in \Delta_T$. By using Lemmas \ref{fhlemma}, \ref{covfest} and (\ref{parR}), one can easily check 

\begin{equation}\label{fia}
|\eta_{12}(s,t)\eta_{22}(s,t)| |t-s|^{2H}\lesssim |t-s|^{2H-2} + s^{H-1}|t-s|^{H-1} + \frac{1}{t}|t-s|^{2H-1},
\end{equation}
\begin{equation}\label{fib}
|\eta_{11}(s,t)\eta_{21}(s,t)| s^{2H} \lesssim |t-s|^{2H-2} + s^{H-1}|t-s|^{H-1} + \frac{1}{t}|t-s|^{2H-1} + t^{-2H}|t-s|^{2H-1},
\end{equation}

\begin{equation}\label{fic}
|\eta_{11}(s,t)\eta_{22}(s,t)|s^H|t-s|^H \lesssim  |t-s|^{2H-2} + s^{H-1}|t-s|^{H-1} + \frac{1}{t}|t-s|^{2H-1} + t^{-2H}|t-s|^{2H-1},
\end{equation}

\begin{equation}\label{fid}
|\eta_{12}(s,t)\eta_{21}(s,t)| s^{H}|t-s|^H\lesssim |t-s|^{2H-2} + s^{H-1}|t-s|^{H-1} + \frac{1}{t}|t-s|^{2H-1} + t^{-2H}|t-s|^{2H-1},
\end{equation}
for $(s,t) \in \Delta_T$.




\end{proof}
In the sequel, we recall

$$\rho(dsdt)=|\mu|(dsdt) + \kappa(s,t)dsdt,$$
where $|\mu|$ is given by (\ref{measures}). In the sequel, for a given $1 < p < \infty$, we denote $\mathcal{E}_{p,\rho}$ as the set of all processes such that

\begin{equation}\label{mod1}
\int_0^T \|Y_t\|^{2}_{2p}m(dt) + \frac{1}{2} \int_{[0,T]^2} \|Y_{s,t}\|^2_{2p} \rho(dsdt) < \infty.
\end{equation}

\begin{lemma}\label{almRK}
If $Y$ satisfies (\ref{mod1}) for $1 < p < \infty$, then the right-hand side of (\ref{iso2}) is finite and 
\begin{eqnarray}
\nonumber\|Y\|^2_{\mathcal{E}^0}&\lesssim& \mathbb{E}\int_0^T \|Y_t\|^{2}_{2p} m(dt) \\ 
&+&\label{st4} \frac{1}{2}\mathbb{E}\int_{[0,T]^2} \|Y_{s,t}\|^2_{2p} \rho(dsdt) < \infty.
\end{eqnarray}
\end{lemma}
\begin{proof}

First, Cauchy-Schwartz's inequality yields 
\begin{eqnarray*}
\big|\langle Y_t\otimes Y_t, \Psi_t\rangle \big|&\le& |Y_t\otimes Y_t||\Psi_t|\\
&\le&|Y_t|^2|\mathcal{M}_t| + |Y_t|^2\frac{\partial R}{\partial t}(T,t)
\end{eqnarray*}
Observe that $|R(T,t)|\lesssim t^{2H}$, $\frac{\partial R}{\partial t}(T,t) = H \{t^{2H-1} + (T-t)^{2H-1}\}$. We then apply H\"older's inequality and Lemma \ref{sharpW} to get 

$$\mathbb{E} \Big[|Y_t|^2 |\mathcal{M}_t|\Big]\lesssim \|Y_t\|^2_{2p} t^{2H-1}$$
for every $t \in (0,T]$. This shows that 
$$\int_0^T\mathbb{E}\big|\langle Y_t\otimes Y_t, \Psi_t\rangle \big|dt\lesssim \int_0^T \| Y_t\|^2_{2p}\frac{\partial R}{\partial t}(t,T)dt .$$
Similarly, Cauchy-Schwartz's inequality and Lemma \ref{West} yield 

\begin{eqnarray*}
\int_{[0,T]^2}\mathbb{E}\big|\langle Y_{s,t}\otimes Y_{s,t}, \Lambda(s,t)\rangle \big|dsdt &\le& \int_{[0,T]^2}\mathbb{E} \Big[|Y_{s,t}|^2 |\mathcal{W}(s,t)|\Big] dsdt\\
&+& \int_{[0,T]^2}\mathbb{E}|Y_{s,t}|^2|\mu|(dsdt)\\
&\lesssim& \int_{[0,T]^2}\|Y_{s,t}\|^2_{2p}\rho(dsdt). 
\end{eqnarray*}

This concludes the proof.
\end{proof}

\begin{theorem}\label{mainTH3}
Let $g:[0,T]\times \mathbb{R}^d\rightarrow \mathbb{R}^d$ be a bounded function. Assume there exists $p>1$ such that 

\begin{equation}\label{fundASS1}
\int_{[0,T]^2} \|g(t,B_t)-g(s,B_s)\|^2_{2p} \rho(dsdt) < \infty,
\end{equation}
and

\begin{equation}\label{fundASS2}
\int_{[0,T]^2} \Big(\sup_{n\ge 1}\int_{[0,T]\times \mathbb{R}^d} \Big\|g\big(t-\frac{r}{n},B_t-\frac{z}{n}\big)-g\big(s-\frac{r}{n},B_s-\frac{z}{n}\big)\Big\|^{2p}_{2p}\alpha(r,z)drdz\Big)^{\frac{1}{p}} \rho(dsdt)< \infty,
\end{equation}
for every non-negative $\alpha \in C^\infty_c((0,T)\times \mathbb{R}^d)$ with $\int \alpha(r,z)drdz=1$. Then $Y = g(B)  \in \mathcal{L}_R(\mathbb{R}^d)\cap \mathcal{E}_{p,\rho}$ and 

\begin{eqnarray}
\nonumber\|Y\|^2_{\mathcal{L}_R(\mathbb{R}^d)}&\lesssim& \int_0^T \|Y_t\|^{2}_{2p} m(dt) \\
\label{st3}&+& \frac{1}{2} \int_{[0,T]^2} \|Y_{s,t}\|^2_{2p} \rho(dsdt).
\end{eqnarray}
\end{theorem}
\begin{proof}
We will show that $Y_\cdot =g(\cdot,B_\cdot)$ is the limit of a sequence $(Y_n)_{n\ge 1}$ in $\mathcal{E}^0$, where $Y_n(t) = g_n(t,B_t)$ and $g_n$ is a sequence of smooth functions. For simplicity of exposition, we set $d=1$ and we assume $g(t,x) = g(x)$. Let $\alpha_n =n \alpha(nx) $ be a sequence of mollifiers and we denote $g_n = g \star \alpha_n$ for $n\ge 1$. Since $g$ is bounded, then $g_n$ is globally Lipschitz (in particular, $g_n \in C^\gamma(\mathbb{R})$ for $\gamma > \frac{1}{2H}-1$) and hence $Y_n \in \mathcal{E}^0$ for every $n\ge 1$. By change the variables, we have 

$$g_n(x) = \int_{0}^\infty g(x-y)\alpha_{n}(y)dy = \int_{0}^\infty g\Big(x-\frac{y}{n}\Big)\alpha(y)dy.$$
We observe that 

$$\lim_{n\rightarrow +_\infty}\mathbb{E}|g_n(B_t) - g(B_t)|^{2p}=0$$
and 

$$\lim_{n\rightarrow \infty}\mathbb{E}|g_n(B_t) - g(B_t) - g_n(B_{s})+ g(B_s)|^{2p}=0$$
for each $s,t \in [0,T]^2$. Since $g$ is bounded and the positive measure $m(dt)=\frac{\partial R}{\partial t}(t,T)dt$ is finite, then we can apply bounded convergence theorem to get 

$$\lim_{n\rightarrow+\infty}\int_0^T \| g_n(B_t) - g(B_t)\|^2_{2p}\frac{\partial R}{\partial t}(t,T)dt=0.$$




It remains to prove

$$\int_{[0,T]^2} \|g_n(B_t) - g(B_t) - g_n(B_{s})+ g(B_s)\|^2_{2p} \rho(dsdt)\rightarrow 0,$$
as $n\rightarrow +\infty$, for $p>1$. We will check there exists $G \in L^1(\rho)$ such that 

$$\Big(\mathbb{E}|g_n(B_t) - g(B_t) - g_n(B_{s})+ g(B_s)|^{2p}\Big)^{\frac{1}{p}}\lesssim G(s,t)$$
for every $(s,t) \in [0,T]^2$. Indeed, Jensen's inequality and Fubini's theorem yield 

$$\Big(\mathbb{E}|g_n(B_t) - g(B_t) - g_n(B_{s})+ g(B_s)|^{2p}\Big)^{}$$
$$\le  \mathbb{E}\int_{\mathbb{R}}\Big| g\Big(B_s-\frac{y}{n}\Big) - g(B_s) - g\Big(B_t-\frac{y}{n}\Big)+g(B_t)\Big|^{2p}\alpha(y)dy $$
$$
\lesssim \mathbb{E}\int_{\mathbb{R}} \Big| g\Big(B_s-\frac{y}{n}\Big) - g\Big(B_t-\frac{y}{n}\Big)\Big|^{2p}\alpha(y)dy +\mathbb{E}\Big| g(B_s) - g(B_t)\Big|^{2p}$$
for $(s,t) \in [0,T]^2$. Assumptions (\ref{fundASS1}) and (\ref{fundASS2}), Lebesgue dominated theorem and Lemma \ref{almRK} allow us to conclude the proof. 









     
\end{proof}

\begin{remark}\label{UNcont}
Observe that in case $|g(t,x) - g(s,y)|$ only depends on $\textbf{d}^{\bar{\gamma},\gamma} ((t,x);(s,y))$ given by (\ref{dmetric})  (e.g. $g$ is uniformly continuous), then condition (\ref{fundASS2}) becomes (\ref{fundASS1}). Moreover, the boundedness condition is not essential. One can extend Theorem \ref{mainTH3} for unbounded functions by using localization arguments.   
\end{remark}

\subsection{Examples}\label{examplesection}
This section presents a class of fundamental examples of functions $g$ which satisfy assumptions (\ref{fundASS1}) and (\ref{fundASS2}) but which are not $\gamma$-H\"older continuous functions for \textit{any} $\gamma > \frac{1}{2H}-1$. For simplicity of exposition, we treat the autonomous case $g(t,x) = g(x)$. We assume that $g:\mathbb{R}^d\rightarrow \mathbb{R}^d$ is a continuous bounded function with modulus of continuity of the form 

\begin{equation}\label{modulus}
V(x) = x^{\frac{1}{2H}-1}\psi(x); 0 \le x < 1,
\end{equation}
where $\frac{1}{4}< H  < \frac{1}{2}$ and $\psi:\mathbb{R}_+\rightarrow \mathbb{R}_+$ satisfies the following properties: 

\begin{enumerate}
  \item There exists $\eta>0$ such that   
  $$\int_0^\eta \frac{1}{x}\psi^2(x)dx < \infty.$$
  
  \item There exist constants $r>1$ and $A>0$ such that $\psi^{2r}$ is a concave non-decreasing function on $(0,A)$ with $\psi(0)=0$.
  \item $\lim_{x\downarrow 0} \frac{\psi(x)}{x^\delta}=+\infty$ for every $\delta>0$. 
  \item $\psi^2(x^H)\lesssim \psi^2(x)$ for every $ x \in (0,1)$.  
\end{enumerate}

The intent of the decomposition (\ref{modulus}) is to employ slowly varying functions for $\psi$. Recall that a function $\psi$ is slowly varying at $0$ if for any $a>0$ fixed, $\psi$ and $x \mapsto \psi(ax)$ are asymptotically equivalent near $0$ (see e.g. Chapter 1 in \cite{bingham}). If $\psi$ is slowly varying, then the space of functions with $V$ as their modulus of continuity is stricly larger than the space of functions which are $\gamma$-H\"older-continuous for some $\gamma > \frac{1}{2H}-1$. However, Condition (2) above implies that $\psi$ is bounded near 0, and therefore, the functions with $V$ as their modulus of continuity are $(\frac{1}{2H}-1)$-H\"older-continuous. In fact, the integrability condition (1) is slightly stronger than boundedness and continuity at 0, as the examples following Proposition \ref{examPROP} show. Slowly varying functions need not have agreeable monotonicity and concavity assumptions; they can contain oscillatory terms. Since our use of these functions is to defined moduli of continuity, we impose those additional restrictions in the form of Condition (2). In our practical examples presented below, the properties of logarithms that lead to Conditions (2) and (3) also imply Condition (4). Having said this, the decomposition (\ref{modulus}) does not require $\psi$ to be slowly varying. In fact, the entire scale of admissible examples can be generated from (\ref{modulus}) by asking that $\psi$ be regularly varying, which, by the so-called Karamata representation theorem (see e.g. Th 1.3.1 in \cite{bingham}), implies that $\psi$ is the product of a slowly varying function and a power function. If the power is zero, we are back to the case where $\psi$ is slowly varying. If the power is positive, then the functions will be $\gamma$-H\"older-continuous for some $\gamma > \frac{1}{2H}-1$. The case where the power is negative cannot lead to examples for us since it would imply that Condition (1) is violated.

\begin{proposition}\label{examPROP}
If $g:\mathbb{R}^d\rightarrow \mathbb{R}^d$ is a bounded continuous function with modulus of continuity of the form $V(x) = x^{\frac{1}{2H}-1}\psi(x)$ for $x \in [0,1)$. If $\psi$ satisfies properties (1), (2), (3) and (4), then $Y = g(B) \in \mathcal{L}_{R}(\mathbb{R}^d)\cap \mathcal{E}_{p,\rho}$ for some $p>1$. 
\end{proposition}

\begin{proof}
We may assume that $T=1$. By property (2), there exist $R>1$ and $A >0$ such that $\psi^{2r}$ is concave on $(0,A)$. Recall the reverse Jensen's inequality 

\begin{eqnarray}
\nonumber\mathbb{E} \Big[\psi^{2r} \big( |B_{s,t}|  \mathds{1}_{\{|B_{s,t}|\le A\}}\big)\Big]&\le& \psi^{2r} \Big( \mathbb{E}\big[|B_{s,t}|\mathds{1}_{\{|B_{s,t}|\le A\}}\big]  \Big) \nonumber\\
\nonumber&\le& \psi^{2r} \Big( \mathbb{E}|B_{s,t}|\Big)\\
\label{f4}&=& \psi^{2r}(|t-s|^H).
\end{eqnarray}
Then, we proceed as follows. Let $Y_t = g(B_t)$. For every $1 < p < \infty$, we have 

$$|Y_{s,t}|^{2p}\le \|g\|^{2p}_\infty A^{-2p}  |B_{s,t}|^{2p} + |Y_{s,t}|^{2p}\mathds{1}_{\{|B_{s,t}|\le A\}},$$
and therefore, 

\begin{equation}\label{last1}
\big(\mathbb{E}|Y_{s,t}|^{2p}\big)^{\frac{1}{p}}\lesssim |t-s|^{2H} + \Big( \mathbb{E} |Y_{s,t}|^{2p} \mathds{1}_{\{ |B_{s,t}|\le A\}} \Big)^{\frac{1}{p}}
\end{equation}
By assumption $|Y_{s,t}|\le |B_{s,t}|^{\frac{1}{2H}-1} \psi (|B_{s,t}|)$ and hence H\"older's inequality jointly with (\ref{f4}) and property (4) yield

$$\Big( \mathbb{E} |Y_{s,t}|^{2p} \mathds{1}_{\{ |B_{s,t}|\le A\}} \Big)^{\frac{1}{p}}|t-s|^{2H-2}$$
$$\lesssim |t-s|^{-1} \Big(\mathbb{E}\psi^{2pq}\big(|B_{s,t}|\big)\mathds{1}_{\{|B_{s,t}|\le A\}}\Big)^{\frac{1}{pq}}  $$
$$ = |t-s|^{-1} \Big(\mathbb{E}\big[\psi^{2pq}\big(|B_{s,t}|\mathds{1}_{\{|B_{s,t}|\le A\}}\big)\big]\Big)^{\frac{1}{pq}}$$
$$
\le |t-s|^{-1}\psi^2(|t-s|^H)
$$
\begin{equation}\label{last2}
= H ^{-\beta-1}|t-s|^{-1}\psi^2 (|t-s|),
\end{equation}
for every $(s,t) \in [0,T]^2_\star$, where we choose $p,q>1$ such that $pq=r$. From (\ref{last1}), (\ref{last2}), property (1) and the fact that $\frac{1}{4} < H < \frac{1}{2}$, we have  

\begin{equation}\label{kapparem0}
\int_{[0,T]^2} \| Y_{s,t}\|^2_{2p}|t-s|^{2H-2}dsdt < \infty,
\end{equation}
for some $p >1$. By applying Theorem \ref{mainTH3}, we conclude the proof.

\end{proof}

We now present some concrete examples of functions $\psi$ which satisfy properties (1), (2), (3) and (4). Our examples are well defined on intervals of the form $[0,a]$ for some $a<1$. Rescaling $x$ in these formulas would result in functions which are well defined and bounded on $[0,1]$; this rescaling is omitted, to avoid overloading the notation. 

\begin{example}
For $\alpha > \frac{1}{2}$, 
$$
\psi(x):=\left\{
\begin{array}{rl}
\big(\log \big(\frac{1}{x}\big)\big)^{-\alpha}; & \hbox{if} \  0 < x < 1 \\
0;& \hbox{if} \ x= 0. \\
\end{array}
\right.
$$
Condition (2) and (4) are trivial and condition (3) is easy to check. Details are left to the reader. By making the change of variable $u=-\log(x)$, one checks that Condition (1) is fulfilled.  
\end{example}

Note that the function $\psi$ in the previous example with $\alpha \le {1/2}$ does not satisfy the integrability condition (1). The same comment also applies to the next example. These examples thus explore the edge cases of our framework.

\begin{example}
For $\alpha > \frac{1}{2}$, 
$$
\psi(x):=\left\{
\begin{array}{rl}
\big(\log\big(\frac{1}{x}\big)\big)^{-\frac{1}{2}}\big(\log(\log(\frac{1}{x}))\big)^{-\alpha}; & \hbox{if} \  0 < x < 1/e \\
0;& \hbox{if} \ x= 0. \\
\end{array}
\right.
$$
Condition (2) is trivial and condition (3) is easy; we leave it to the committed reader to check them. Condition (4) is equally straightforward as soon as one notes that for $x$ small enough, 
$$
\frac{1}{2}\log(\log(\frac{1}{x})) \le \log(\log(\frac{1}{x^H})) \le \log(\log(\frac{1}{x})) . 
$$
 By making the change of variable $u=\log(-\log(x))$, we observe that condition (1) is fulfilled.  
\end{example}

As is well-known, one may modify the previous example by letting $\alpha = 1/2$ and including a further factor which is a thrice-iterated logarithm, to a negative power of $\alpha > 1/2$, and so forth with higher iterations of the log, to explore the threshold of functions which barely satisfy the integrability condition (1). At the other end of the spectrum, the following example still yields functions which are not $\gamma$-H\"older-continuous for any $\gamma > \frac{1}{2H}-1$, but which are more regular than any function from Example 6.2. 

\begin{example}
With $\beta >1$, 
$$
\psi(x):=\left\{
\begin{array}{rl}
\exp \left[ -\left( \log \left( \log \left( \frac{1}{x}\right) \right) \right) ^{\beta }\right] ; & \hbox{if} \  0 < x < 1/e \\
0;& \hbox{if} \ x= 0. \\
\end{array}
\right.
$$
Again, Conditions (2) and (3) are straightforward. Condition (4) is treated as in the previous example.
Condition (1) follows from noting that $\psi(x) = o\big(\big(\log (1/x)\big)^{-\alpha}\big)$, and using Condition (1) from Example 6.2. 
\end{example}

\

\noindent \textbf{Acknowledgments:} The authors gratefully acknowledge funding by FAPDF which supports the project ``Rough paths, Malliavin calculus and related topics'' 00193-00000229 2021-21.
The research of FR
was also partially supported by the  ANR-22-CE40-0015-01 project (SDAIM).
AO gratefully acknowledges the hospitality and support of
the Department of Statistics at Rice University. 

\bibliographystyle{acm}
\begin{quote}
\bibliography{../../BIBLIO_FILE_ORV/refORV}

\def\cprime{$'$}
\begin{thebibliography}{10}

\bibitem{adler}
{\sc Adler, R.~J., Taylor, J.~E., et~al.}
\newblock {\em Random fields and geometry}, vol.~80.
\newblock Springer, 2007.

\bibitem{alosleon}
{\sc Alos, E., Le{\'o}n, J.~A., and Nualart, D.}
\newblock Stochastic {S}tratonovich calculus fbm for fractional {B}rownian
  motion with {H}urst parameter less than 1/2.
\newblock {\em Taiwanese Journal of Mathematics 5}, 3 (2001), 609--632.

\bibitem{alos2001stochastic}
{\sc Al\`os, E., Mazet, O., and Nualart, D.}
\newblock Stochastic calculus with respect to {G}aussian processes.
\newblock {\em Annals of probability\/} (2001), 766--801.

\bibitem{biaginibook}
{\sc Biagini, F., Hu, Y., {\O}ksendal, B., and Zhang, T.}
\newblock {\em Stochastic calculus for fractional Brownian motion and
  applications}.
\newblock Springer Science \& Business Media, 2008.

\bibitem{bingham}
{\sc Bingham, N.~H., Goldie, C.~M., and Teugels, J.~L.}
\newblock {\em Regular variation}, vol.~27 of {\em Encycl. Math. Appl.}
\newblock Cambridge University Press, Cambridge, 1987.

\bibitem{carmona2003}
{\sc Carmona, P., Coutin, L., and Montseny, G.}
\newblock Stochastic integration with respect to fractional brownian motion.
\newblock In {\em Annales de l'IHP Probabilit{\'e}s et statistiques\/} (2003),
  vol.~39, pp.~27--68.

\bibitem{cass2019}
{\sc Cass, T., and Lim, N.}
\newblock A {S}tratonovich-{S}korohod integral formula for {G}aussian rough
  paths.
\newblock {\em Annals of Probability 47}, 1 (2019), 1--60.

\bibitem{cheridito2005}
{\sc Cheridito, P., and Nualart, D.}
\newblock {Stochastic integral of divergence type with respect to fractional
  {B}rownian motion with {H}urst parameter \(H \in (0,\frac {1}{2})\)}.
\newblock {\em {Ann. Inst. Henri Poincar\'e, Probab. Stat.} 41}, 6 (2005),
  1049--1081.

\bibitem{darses}
{\sc Darses, S., and Nourdin, I.}
\newblock {Stochastic derivatives for fractional diffusions}.
\newblock {\em The Annals of Probability 35}, 5 (2007), 1998 -- 2020.

\bibitem{darses1}
{\sc Darses, S., Nourdin, I., and Peccati, G.}
\newblock Differentiating $\sigma$-fields for gaussian and shifted gaussian
  processes.
\newblock {\em Stochastics: An International Journal of Probability and
  Stochastics Processes 81}, 1 (2009), 79--97.

\bibitem{decreusefond}
{\sc Decreusefond, L., and {\"U}st{\"u}nel, A.~S.}
\newblock Stochastic analysis of the fractional {B}rownian motion.
\newblock {\em Potential analysis 10\/} (1999), 177--214.

\bibitem{friz2020rough}
{\sc Friz, P., and Zorin-Kranich, P.}
\newblock Rough semimartingales and $ p $-variation estimates for martingale
  transforms.
\newblock {\em ArXiv preprint arXiv:2008.08897\/} (2020).

\bibitem{hairerbook}
{\sc Friz, P.~K., and Hairer, M.}
\newblock {\em A course on rough paths}.
\newblock Universitext. Springer, Cham, 2014.
\newblock With an introduction to regularity structures.

\bibitem{friz2021}
{\sc Friz, P.~K., Hocquet, A., and L{\^e}, K.}
\newblock Rough stochastic differential equations.
\newblock {\em ArXiv preprint arXiv:2106.10340\/} (2021).

\bibitem{friz}
{\sc Friz, P.~K., and Victoir, N.~B.}
\newblock {\em Multidimensional stochastic processes as rough paths}, vol.~120
  of {\em Cambridge Studies in Advanced Mathematics}.
\newblock Cambridge University Press, Cambridge, 2010.
\newblock Theory and applications.

\bibitem{gerencser}
{\sc Gerencs{\'e}r, M.}
\newblock Regularisation by regular noise.
\newblock {\em Stochastics and Partial Differential Equations: Analysis and
  Computations 11}, 2 (2023), 714--729.

\bibitem{GOR}
{\sc Gomes, A., Ohashi, A., Russo, F., and Teixeira, A.}
\newblock Rough paths and regularization.
\newblock {\em Journal of Stochastic Analysis (JOSA). 2 (4)\/} (12 2021),
  1--21.

\bibitem{gradno}
{\sc Gradinaru, M., and Nourdin, I.}
\newblock Approximation at first and second order of {$m$}-order integrals of
  the fractional {B}rownian motion and of certain semimartingales.
\newblock {\em Electron. J. Probab. 8\/} (2003), no. 18, 26 pp. (electronic).

\bibitem{gnrv}
{\sc Gradinaru, M., Nourdin, I., Russo, F., and Vallois, P.}
\newblock {$m$}-order integrals and generalized {I}t\^o's formula: the case of
  a fractional {B}rownian motion with any {H}urst index.
\newblock {\em Ann. Inst. H. Poincar\'e Probab. Statist. 41}, 4 (2005),
  781--806.

\bibitem{grv}
{\sc Gradinaru, M., Russo, F., and Vallois, P.}
\newblock Generalized covariations, local time and {S}tratonovich {I}t\^o's
  formula for fractional {B}rownian motion with {H}urst index {$H\ge\frac14$}.
\newblock {\em Ann. Probab. 31}, 4 (2003), 1772--1820.

\bibitem{gubinelli2004}
{\sc Gubinelli, M.}
\newblock Controlling rough paths.
\newblock {\em Journal of Functional Analysis 216}, 1 (2004), 86--140.

\bibitem{ito}
{\sc It{\^o}, K.}
\newblock 109. stochastic integral.
\newblock {\em Proceedings of the Imperial Academy 20}, 8 (1944), 519--524.

\bibitem{krukrusso}
{\sc Kruk, I., and Russo, F.}
\newblock Malliavin-{S}korohod calculus and {P}aley-{W}iener integral for
  covariance singular processes.
\newblock {\em ArXiv preprint arxiv:1011.6478v1\/} (2010).

\bibitem{kruk2007}
{\sc Kruk, I., Russo, F., and Tudor, C.~A.}
\newblock Wiener integrals, {M}alliavin calculus and covariance measure
  structure.
\newblock {\em J. Funct. Anal. 249}, 1 (2007), 92--142.

\bibitem{KW}
{\sc Kunita, H., and Watanabe, S.}
\newblock On square integrable martingales.
\newblock {\em Nagoya Mathematical Journal 30\/} (1967), 209--245.

\bibitem{Le}
{\sc L{\^e}, K.}
\newblock {A stochastic sewing lemma and applications}.
\newblock {\em Electronic Journal of Probability 25}, none (2020), 1 -- 55.

\bibitem{tindelliu}
{\sc Liu, Y., Selk, Z., and Tindel, S.}
\newblock Convergence of trapezoid rule to rough integrals.
\newblock {\em ArXiv preprint arXiv:2005.06500\/}.

\bibitem{matsuda}
{\sc Matsuda, T., and Perkowski, N.}
\newblock An extension of the stochastic sewing lemma and applications to
  fractional stochastic calculus.
\newblock {\em ArXiv preprint arXiv:2206.01686\/} (2022).

\bibitem{viensSK}
{\sc Mocioalca, O., and Viens, F.}
\newblock Skorohod integration and stochastic calculus beyond the fractional
  {B}rownian scale.
\newblock {\em Journal of Functional analysis 222}, 2 (2005), 385--434.

\bibitem{nelson}
{\sc Nelson, E.}
\newblock {\em Dynamical theories of Brownian motion}, vol.~106.
\newblock Princeton university press, 2020.

\bibitem{nualart2006}
{\sc Nualart, D.}
\newblock {\em The Malliavin calculus and related topics}, vol.~1995.
\newblock Springer, 2006.

\bibitem{OhashiRusso}
{\sc Ohashi, A., and Russo, F.}
\newblock {R}ough paths and symmetric-{S}tratonovich integrals driven by
  singular covariance {G}aussian processes.
\newblock {\em ArXiv preprint arXiv:2206.06865. To appear: Bernoulli\/} (2023).

\bibitem{russo1993forward}
{\sc Russo, F., and Vallois, P.}
\newblock Forward, backward and symmetric stochastic integration.
\newblock {\em Probability theory and related fields 97}, 3 (1993), 403--421.

\bibitem{Russo_Vallois_Book}
{\sc Russo, F., and Vallois, P.}
\newblock {\em Stochastic Calculus via Regularizations}.
\newblock Springer International Publishing, 2022.

\bibitem{tindelsong}
{\sc Song, J., and Tindel, S.}
\newblock Skorohod and {S}tratonovich integrals for controlled processes.
\newblock {\em Stoch Process Their Appl.\/} (2022).

\end{thebibliography}
\end{quote}

\end{document}